\documentclass[12pt,reqno]{amsart}

\usepackage{plain}
\usepackage{amssymb}
\usepackage{amscd}
\usepackage{graphicx}
\usepackage{amsmath}
\usepackage{version}
\usepackage{mathrsfs}  
\usepackage{tikz}
\usepackage{mathtools}
\usetikzlibrary{arrows,decorations.markings, matrix}
\usepackage[pagebackref=true]{hyperref}
\usepackage[all]{xy}
\usepackage{tikz-cd}
\usepackage{float} 

\usepackage{bbm}
\newcommand{\bbk}{\mathbbm{k}}

\newskip\stdskip                      % standard vertical space
\DeclareMathAlphabet\EuScript{U}{eus}{m}{n}
\SetMathAlphabet\EuScript{bold}{U}{eus}{b}{n}

\usepackage[normalem]{ulem}

\numberwithin{equation}{section}
\numberwithin{figure}{section}
\numberwithin{table}{section}

\newtheorem{thm}{Theorem}[section]

\newtheorem{lemma}[thm]{Lemma}
\newtheorem{prop}[thm]{Proposition}

\newtheorem{cor}[thm]{Corollary}
\newtheorem{corollary}[thm]{Corollary}

\theoremstyle{definition}
\newtheorem{definition}[thm]{Definition}
\newtheorem{example}[thm]{Example}
\newtheorem{rmk}[thm]{Remark}
\newtheorem{remark}[thm]{Remark}
\newtheorem{conj}[thm]{Conjecture}

\newcommand{\Z}{\mathbb{Z}}

\newcommand{\z}{\mathbf{z}}
 
\newcommand{\Gm}{\mathbb{G}_m}

\newcommand{\Funex}{\operatorname{Fun}^{\mathrm{ex}}}

\newcommand{\id}{\operatorname{id}}

\newcommand{\Hom}{\operatorname{Hom}}

\newcommand{\Ext}{\operatorname{Ext}}
\newcommand{\End}{\operatorname{End}}
\newcommand{\Res}{\operatorname{Res}}
\newcommand{\Jac}{\operatorname{Jac}}
\newcommand{\Aut}{\operatorname{Aut}}

\newcommand{\C}{{\mathbb C}}

\newcommand{\Fun}{\operatorname{Fun}}

\newcommand{\linspan}{\operatorname{span}}
\newcommand{\PGL}{\operatorname{PGL}}

\newcommand{\Cone}{\operatorname{Cone}}

\newcommand{\Tor}{\operatorname{Tor}}

\newcommand{\Sym}{\operatorname{Sym}}

\newcommand{\G}{{\mathbb G}}

\newcommand{\Spec}{\operatorname{Spec}}

\newcommand{\Proj}{\operatorname{Proj}}
\renewcommand{\P}{{\mathbb P}}

\newcommand{\Pic}{\operatorname{Pic}}

\renewcommand{\ker}{\operatorname{ker}}

\newcommand{\perf}{\operatorname{perf}}
\newcommand{\w}{{\bf w}}

\usepackage{pst-all}
\usepackage{enumerate}
\usepackage{booktabs}

\newcommand{\Qcoh}{\operatorname{Qcoh}}
\newcommand{\coh}{\operatorname{coh}}
\newcommand{\gr}{\operatorname{gr}}
\newcommand{\sing}{{\mathrm{sing}}}
\newcommand{\ev}{\operatorname{ev}}
\newcommand{\coev}{\ev^\dual}
\newcommand{\bGm}{\Gm}
\newcommand{\eu}{\mathrm{eu}}
\newcommand{\op}{\mathrm{op}}
\newcommand{\Mod}{\operatorname{Mod}}

\newcommand{\SH}{\operatorname{SH}}

\newcommand{\xto}{\xrightarrow}

\newcommand{\vspan}{\operatorname{span}}
\newcommand{\HH}{\operatorname{HH}}
\newcommand{\colim}{\operatorname{colim}}

\newcommand{\diag}{\operatorname{diag}}

\newdimen\argwidth
\def\db[#1\db]{%
 \setbox0=\hbox{$#1$}\argwidth=\wd0
 \setbox0=\hbox{$\left[\box0\right]$}
  \advance\argwidth by -\wd0
 \left[\kern.3\argwidth\box0 \kern.3\argwidth\right]}

\newcommand{\bP}{\ensuremath{\mathbb{P}}}
\newcommand{\bC}{\ensuremath{\mathbb{C}}}
\newcommand{\bZ}{\ensuremath{\mathbb{Z}}}
\newcommand{\bA}{\ensuremath{\mathbb{A}}}
\newcommand{\bG}{\ensuremath{\mathbb{G}}}
\newcommand{\bN}{\ensuremath{\mathbb{N}}}

\newcommand{\cD}{\ensuremath{\mathcal{D}}}

\newcommand{\cS}{\ensuremath{\mathcal{S}}}
\newcommand{\cP}{\ensuremath{\mathcal{P}}}
\newcommand{\cT}{\ensuremath{\mathcal{T}}}
\newcommand{\cW}{\ensuremath{\mathcal{W}}}
\newcommand{\cF}{\ensuremath{\mathcal{F}}}
\newcommand{\cA}{\ensuremath{\mathcal{A}}}

\newcommand{\frakm}{\ensuremath{\mathfrak{m}}}

\newcommand{\scrA}{\ensuremath{\mathscr{A}}}
\newcommand{\scrB}{\ensuremath{\mathscr{B}}}

\newcommand{\scrF}{\ensuremath{\mathscr{F}}}
\newcommand{\scrS}{\ensuremath{\mathscr{S}}}

\newcommand{\scrY}{\ensuremath{\mathscr{Y}}}
\newcommand{\scrW}{\ensuremath{\mathscr{W}}}

\newcommand{\bsw}{{\boldsymbol{w}}}

\newcommand{\bsR}{{\boldsymbol{R}}}
\newcommand{\bsT}{{\boldsymbol{T}}}
\newcommand{\bsRbar}{\overline{\bsR}}

\newcommand{\bsk}{{\boldsymbol{k}}}
\newcommand{\bszero}{{\boldsymbol{0}}}

\newcommand{\wtilde}{{\widetilde{w}}}
\newcommand{\bswtilde}{\widetilde{\bsw}}
\newcommand{\Utilde}{{\widetilde{U}}}

\newcommand{\lb}{\left(}
\newcommand{\rb}{\right)}
\newcommand{\la}{\left\langle}
\newcommand{\ra}{\right\rangle}
\newcommand{\lc}{\left\{}
\newcommand{\rc}{\right\}}
\newcommand{\ld}{\left[}
\newcommand{\rd}{\right]}

\newcommand{\Zbar}{{\overline{Z}}}

\usepackage{prettyref}

\newcommand{\pref}{\prettyref}

\newrefformat{th}{Theorem~\ref{#1}}
\newrefformat{cr}{Corollary~\ref{#1}}
\newrefformat{lm}{Lemma~\ref{#1}}
\newrefformat{dl}{Definition-Lemma~\ref{#1}}
\newrefformat{df}{Definition~\ref{#1}}
\newrefformat{cl}{Claim~\ref{#1}}
\newrefformat{sl}{Sublemma~\ref{#1}}
\newrefformat{pr}{Proposition~\ref{#1}}
\newrefformat{cj}{Conjecture~\ref{#1}}
\newrefformat{st}{Step~\ref{#1}}
\newrefformat{sc}{Section~\ref{#1}}
\newrefformat{df}{Definition~\ref{#1}}
\newrefformat{rm}{Remark~\ref{#1}}
\newrefformat{q}{Question~\ref{#1}}
\newrefformat{pb}{Problem~\ref{#1}}
\newrefformat{cd}{Condition~\ref{#1}}
\newrefformat{eg}{Example~\ref{#1}}
\newrefformat{he}{Heore~\ref{#1}}
\newrefformat{fg}{Figure~\ref{#1}}
\newrefformat{tb}{Table~\ref{#1}}
\newrefformat{as}{Assumption~(\ref{#1})}

\newcommand{\dgend}{\operatorname{end}}

\newcommand{\per}{\operatorname{perf}}

\newcommand{\emf}[3]{\operatorname{mf} \left( \left[ #1 \middle/ #2 \right], #3 \right)}
\newcommand{\dual}{\vee}

\newcommand{\Yl}{\scrY^{\mathrm{l}}}
\newcommand{\Yr}{\scrY^{\mathrm{r}}}
\newcommand{\CC}{\operatorname{CC}}

\newcommand{\relmid}{\mathrel{}\middle|\mathrel{}}

\newcommand{\bsi}{\boldsymbol{i}}
\newcommand{\bsj}{{\boldsymbol{j}}}
\newcommand{\Gammahat}{\widehat{\Gamma}}
\newcommand{\bfW}{\mathbf{W}}
\newcommand{\cU}{\mathcal{U}}

\newcommand{\cY}{\mathcal{Y}}

\renewcommand{\wtilde}{\widetilde{\w}}
\newcommand{\cO}{\mathcal{O}}
\newcommand{\cM}{\mathcal{M}}
\newcommand{\cMbar}{\overline{\cM}}

\newcommand{\FunL}{\Fun^{\mathrm L}}
\newcommand{\bfk}{\mathbf{k}}
\newcommand{\bmu}{\boldsymbol{\mu}}
\newcommand{\simto}{\xto{\sim}}
\newcommand{\bCx}{\bC^\times}
\renewcommand{\Gm}{\bG_{\mathrm{m}}}
\newcommand{\rep}{\operatorname{rep}}
\newcommand{\Rbar}{\overline{R}}

\newcommand{\wv}{\check{\w}}
\newcommand{\dv}{\check{d}}

\newcommand{\hv}{\check{h}}
\newcommand{\Av}{\check{A}}
\newcommand{\Dv}{\check{D}}

\newcommand{\Vv}{\check{V}}

\newcommand{\Xv}{\check{X}}
\newcommand{\Yv}{\check{Y}}

\newcommand{\Ytilde}{\tilde{Y}}
\newcommand{\eA}{\EuScript{A}}
\newcommand{\eE}{\EuScript{E}}

\newcommand{\eS}{\EuScript{S}}

\newcommand{\eV}{\EuScript{V}}
\newcommand{\eX}{\EuScript{X}}
\newcommand{\eY}{\EuScript{Y}}
\newcommand{\eZ}{\EuScript{Z}}
\newcommand{\fA}{\mathfrak{A}}

\newcommand{\Bv}{\check{B}}
\newcommand{\Qgr}{\operatorname{Qgr}}

\title[Moduli of $A_\infty$-structures]{Homological mirror symmetry for Milnor fibers via moduli of $A_\infty$-structures}
\author[Y.~Lekili]{Yank\i\ Lekili}
\address{
Department of Mathematics,
Imperial College London,
South Kensington,
London,
SW7 2AZ,
United Kingdom.
}
\email{y.lekili@imperial.ac.uk}
\author[K.~Ueda]{Kazushi Ueda}
\address{
Graduate School of Mathematical Sciences,
The University of Tokyo,
3-8-1 Komaba,
Meguro-ku,
Tokyo,
153-8914,
Japan.}
\email{kazushi@ms.u-tokyo.ac.jp}
\begin{document} 

\begin{abstract}
We show that the base spaces
of the semiuniversal unfoldings
of some weighted homogeneous singularities
%(appearing in Arnold's strange duality
%and its generalizations due to Berglund--H\"ubsch--Krawitz)
%satisfying some conditions,
%such as the existence of a tilting object
%on the singularity category
can be identified with moduli spaces of $A_\infty$-structures
on the trivial extension algebras of the endomorphism algebras
of the tilting objects.
The same algebras also appear in the Fukaya categories
of their mirrors.
Based on these identifications,
we discuss applications
to homological mirror symmetry for Milnor fibers, and
give a proof of homological mirror symmetry
for an $n$-dimensional affine hypersurface of degree $n+2$
and the double cover of the $n$-dimensional affine space
branched along a degree $2n+2$ hypersurface.
Along the way, we also give a proof of a conjecture of Seidel from \cite{seidelICM}
which may be of independent interest.
\end{abstract}

\maketitle

\setcounter{tocdepth}{1}
%\tableofcontents

\section{Introduction}
 \label{sc:introduction}

\subsection{Moduli of elliptic curves}
 \label{sc:moduli of elliptic curves}

Our basic starting point is an algebraic variety with an isolated singularity
admitting a $\Gm$-action. The primordial example is the cusp singularity
defined by
\begin{align} \label{eq:A_2}
 \lc (x,y) \in \bA^2 \relmid \w(x,y) \coloneqq x^3 + y^2 = 0 \rc.
\end{align}
%in $\bA^2 = \Spec \bfk[x,y]$.
The main construction that we study in this paper originates from \cite{LP1}, where the
case of the cusp singularity was studied in detail. We recall this construction
in order to ease the reader to our topic before discussing higher-dimensional
singularities with a $\Gm$-action.

The cuspidal curve \pref{eq:A_2} has a $\Gm$-action
given by $t \cdot (x,y) = (t^2 x, t^3 y)$.
Thus the coordinate ring gets a grading with $\deg(x)=2$ and $\deg(y)=3$.
It can be compactified to the projective cone
\begin{align}
 \lc [x:y:z] \in \bP(2,3,1) \relmid \w(x,y) = 0 \rc
\end{align}
by adding one point.

The semiuniversal unfolding of $\w$ is given by 
\begin{align}
 \wtilde(x,y;u_4,u_6) \coloneqq x^3 + y^2 + u_4 x + u_6,
\end{align}
whose homogenization
\begin{align} \label{warmcusp}
 \bfW(x,y,z;u_4,u_6) \coloneqq x^3 + y^2 + u_4 xz^4 + u_6 z^6
\end{align}
defines the Weierstrass family
$
 \pi_{\cY} \colon \cY \to U \coloneqq \Spec \bfk[u_4,u_6]
$
of curves in $\bP(2,3,1)$.
Each curve $Y_u \coloneqq \pi_{\cY}^{-1}(u)$ is of arithmetic genus 1
and comes with a point
$
 p \coloneqq \{ z = 0 \}
$
at infinity
and a section
\begin{align}
 \Omega_u
  \coloneqq \Res \frac{z d x \wedge d y}{\bfW(x,y,z,u_4,u_6)}
\end{align}
of the dualizing sheaf,
which is given by
$
 dx/\bfW_y = -dy/\bfW_x
$
on the affine part.
The $\Gm$-action extends to the compactified family by
\begin{align}
 t \cdot ([x:y:z];u_4,u_6)
  &= ([t^2 x:t^3 y:z];t^4 u_4, t^6 u_6) \\
  &= ([x:y:t^{-1} z];t^4 u_4, t^6 u_6),
\end{align}
which preserves the section $z=0$ and satisfies 
\begin{align}
 t^* (\Omega_{t\cdot u})
%  &= \Res \frac{z d (t^2 x) \wedge d (t^3 y)}{\bfW(t^2 x,t^3 y,z;t^4u_4,t^6 u_6)} \\
%  &= \Res \frac{t^5 z d x \wedge d y}{t^6 \bfW(x,y,z;u_4,u_6)} \\
  &= t^{-1}\Omega_{u}.
\end{align}
The curves $Y_u$ are elliptic curves
outside the discriminant
\begin{align}
 \Delta \coloneqq \lc (u_4, u_6) \in U \relmid
  4u_4^3 - 27u_6^2=0\rc.
\end{align}
If $u \in \Delta \setminus \bszero$,
then $Y_u$ is a rational curve with a single ordinary double point.
Note that all curves above a $\Gm$-orbit are isomorphic. 

The base space $U$
%of the Weierstrass family
can be identified with the moduli space of triples $(Y,p,\Omega)$
consisting of a reduced connected curve $Y$ of arithmetic genus 1,
a smooth marked point $p$ on $Y$ such that
$h^0(\cO_Y(p))=1$ and $\cO_Y(p)$ is ample, and
a non-zero section $\Omega$ of the dualizing sheaf of $Y$
(see \cite[Theorem 1.4.2]{LPol}).
Furthermore,
we have an isomorphism
\begin{align}
 \cMbar_{1,1} \cong
 \ld \lb U \setminus \bszero \rb \middle/ \Gm \rd
 \quad ( \cong \bP(4,6) )
\end{align}
with the moduli stack of stable curves
of genus one with one marked point.

\subsection{Moduli of $A_\infty$-structures}

The condition that $\cO_{Y_u}(p)$ is ample is equivalent to 
\begin{align}
 \eS_u \coloneqq \cO_{Y_u} \oplus \cO_p
\end{align}
being a generator of the perfect derived category $\perf Y_u$. 
On the other hand, the fact that $h^0(\cO_{Y_u}(p))=1$ implies that
the isomorphism class of the Yoneda algebra
\begin{align}
 A \coloneqq \End \lb \eS_u \rb
\end{align}
as a graded algebra
is independent of $u \in U$.
Indeed,
it is easy to show that for any $u$,
there is a canonical isomorphism
(where we use the fixed basis $\Omega_u$ of $H^0 \lb \omega_{Y_u} \rb$)
between $A$ and
the degree one trivial extension algebra
of the path algebra of the $A_2$-quiver.
More concretely, this is given by the quiver with relations given in \pref{fg:cy1}. 

\begin{figure}[!h]

\begin{tikzpicture}
    \tikzset{vertex/.style = {style=circle,draw, fill,  minimum size = 2pt,inner    sep=1pt}}

        \tikzset{edge/.style = {->,-stealth',shorten >=8pt, shorten <=8pt  }}

% vertices
\node[vertex] (a) at  (0,0) {};
\node[vertex] (a1) at (1.5,0) {};

\node at  (0,0.3) {\tiny 1};
\node at (1.5,0.3) {\tiny 2};

% edges

\draw[edge] (a)  to[in=150,out=30] (a1);
\draw[edge] (a1)  to[in=330,out=210] (a);

\node at (0.75,-0.5) {\tiny $v$};
\node at (0.75,0.5) {\tiny $u$};

\node at (5,0) {\small $|u|=0, \ |v|=1, \ uvu= vuv=0$};
\end{tikzpicture}
		\caption{Quiver algebra description of $A$}
\label{fg:cy1}
\end{figure} 

Thus,
considering the algebra $A$ results
in a dramatic loss of information
hidden in $\perf Y_u$,
even though $\eS_u$ is a generator.
This is, of course, no surprise as we have forgotten to derive.

Recall that an $A_\infty$-algebra $\cA$ over $\bfk$
is a graded $\bfk$-module
with a collection
$
 \lb \mu^d \rb_{d=1}^\infty
$
of $\bfk$-linear maps
$
 \mu^d \colon \cA^{\otimes d} \to \cA [2-d]
$
satisfying the $A_\infty$-associativity equations
\begin{align}
 \sum_{m,n} (-1)^{|a_1|+\ldots + |a_n|-n}
 \mu^{d-m+1}(a_d,\ldots, a_{n+m+1},
  \mu^m (a_{n+m},\ldots, a_{n+1}), a_n,\ldots a_1) = 0.
\end{align}
In particular,
$\mu^1 \colon \cA \to \cA[1]$ is a differential,
i.e.~$\mu^1 \circ \mu^1 =0$, and the product
\begin{equation}
a_2 \cdot a_1 = (-1)^{|a_1|} \mu^2(a_2,a_1)
\end{equation}
on $\cA$ is associative up to homotopy.

A \emph{minimal $A_\infty$-structure}
on a graded associative $\bfk$-algebra $A$
is an $A_\infty$-structure $(\mu^k)_{k=1}^\infty$
on the graded vector space underlying $A$
such that $\mu^1=0$ and $\mu^2$ coincides
with the given product on $A$.
It is said to be \emph{formal}
if $\mu^k = 0$ for $k > 2$.

Recall that the Hochschild cochain complex of a graded algebra $A$ has a bigrading,
where $\CC^{r+s}(A)_s$ consists of maps $A^{\otimes r} \to A[s]$.
The space of first-order deformations of $A$
as a graded algebra is given by $\HH^2(A)_{0}$,
and deformations to minimal $A_\infty$-structures on $A$
without changing $\mu^2$
is controlled by
$
 \HH^2(A)_{<0}
  \coloneqq \bigoplus_{i=1}^\infty \HH^2(A)_{-i}.
$
Moreover,
if
%the negative part of the first Hochschild cohomology group
$\HH^1(A)_{<0}$ vanishes,
then
\cite[Corollary 3.2.5]{Pol}
shows that
the functor
sending a $\bfk$-algebra $R$
to the set of gauge equivalence classes of minimal $A_\infty$-structures on $A \otimes R$
is represented by an affine scheme $\cU_\infty(A)$,
%This scheme is the limit
%$
% \cU_\infty(A) = \varprojlim_d \cU_d(A)
%$
%of affine schemes $\cU_d(A)$ of finite type
%representing the functor of minimal $A_d$-structures over $A$.
which is of finite type
if $\dim \HH^2(A)_{<0} < \infty$.
%Indeed, the deformation theory developed in \cite[Section 3]{Pol} shows that
%the formal completion of $\cU_\infty(A)$ is
%the zero locus of the Kuranishi map from $\HH^2(A)_{<0}$ to $\HH^3(A)_{<0}$. 
%
There is a natural $\Gm$-action on $\cU_\infty(A)$
given by
\begin{align} \label{eq:Gm-action_on_U}
 \bGm \ni t \colon \lb \mu^d \rb_{d=2}^\infty
 \mapsto \lb t^{d-2} \mu^d \rb_{d=2}^\infty,
\end{align}
and the formal $A_\infty$ structure on $A$
is the fixed point of this action.

Returning back to the Weierstrass family,
as explained in \cite{LP2},
the natural dg enhancement
$\dgend(\eS)$ of $\End(\eS)$
gives a family $\eA$ of minimal $A_\infty$-structures on $A$ over $U$,
and hence a morphism
\begin{align} \label{eq:LP2}
 U \to \cU_{\infty}(A).
\end{align}
We recall the following theorem from \cite{LP2}.
For simplicity,
we state it over a field $\bfk$ with $\operatorname{char} \bfk \neq 2,3$,
see \cite{LP2} for a more general statement. 

\begin{thm} \label{th:cuspthm}
If $\operatorname{char} \bfk \ne 2, 3$,
%be a field of characteristic not equal to 2 or 3.
then
\eqref{eq:LP2}
%the morphism
%$
% U \to \cU_{\infty}(A)
%$
is a $\Gm$-equivariant isomorphism,
sending the cuspidal curve $Y_0$ to the formal $A_\infty$-structure on $A$.
\end{thm}

There are two main ingredients that enter in the proof of this result:
\begin{enumerate}[(i)]
 \item \label{it:item1}
The formality of the $A_\infty$-algebra $\eA_0$
for the cuspidal curve $Y_0$.
 \item
One has $\HH^1(A)_{<0}=0$, so that $\cU_\infty(A)$ is an affine scheme,
and
\begin{align} \label{eq:HH2A0}
\HH^2(A)_{<0} = \bfk(4) \oplus \bfk(6),
\end{align}
so that \pref{eq:LP2} induces an isomorphism
on tangent spaces
at the fixed points of the $\Gm$-action.
\end{enumerate}
Here \pref{eq:HH2A0} means that $\HH^2(A)_s=\bfk$ for $s=-4,-6$
and zero otherwise.

The Hochschild cohomology computation is done
in two different ways in \cite{LP1} and \cite{LP2}.
We will give yet another way in \pref{sc:sylvester} below.

To elaborate on \eqref{it:item1},
first one shows the existence of a chain level $\Gm$-action
by taking the \v{C}ech complex
with respect to a $\Gm$-invariant affine cover.
This gives a dg model for $\cA_0$.
Then, one arranges a $\Gm$-equivariant homotopy to a minimal $A_\infty$-structure,
which follows from the fact that one
can choose chain level representatives of a basis of $\End(\eS_0)$
in such a way that each of them is in a one-dimensional representation of $\Gm$.
Finally, to deduce formality, one shows that the weight of the $\Gm$-action
on $\End(\eS_0)$ agrees with the cohomological grading.
But $\mu^d$ lowers the cohomological degree by $d-2$,
so any $\Gm$-equivariant $A_\infty$-structure must have vanishing $\mu^d$ for $d \neq 2$. 

Other examples of the above construction were subsequently studied in \cite{Pol,LPol},
but all of these work with examples in dimension one.
In this paper, we begin to explore higher dimensions.

\subsection{Application to homological mirror symmetry}

Let $\Vv$ be a once-punctured torus
viewed as a Weinstein manifold,
and
\begin{align}
 Z \coloneqq \lc [x:y:z] \in \bP(2,3,1) \relmid x^3 + y^2 + xyz=0 \rc
\end{align}
be a rational curve with a single ordinary double point.
\pref{th:cuspthm} was obtained in \cite{LP2}
as a tool for proving a quasi-equivalence
\begin{align} \label{eq:punctured}
 \cF \lb \Vv \rb \simeq \perf Z
\end{align}
of pretriangulated $A_\infty$-categories over $\Z$
of the split-closed derived Fukaya category
of compact exact Lagrangians in $\Vv$
and the perfect derived category of $Z$.
The strategy is first to identify generators on both sides,
and then match their endomorphism algebras as $A_\infty$-algebras.
It is often difficult to explicitly compute such $A_\infty$-algebras,
but even if one does,
finding a quasi-isomorphism between two different chain models is usually a hard task.
The computation of cohomology level structures (and matching them) is much easier, and knowing the moduli of $A_\infty$-structures allows one to appeal to indirect methods to conclude the proof of the existence of a chain level isomorphism. Such a strategy was applied also for proving homological mirror symmetry in a number of other cases in dimension one. Namely, in \cite{LPol} a class of curve singularities $C_{1,n}$ for $n \geq 1$ were considered, where $C_{1,1}$ is the cuspidal curve, $C_{1,2}$ is tacnodal curve given by the equation $y^2 = yx^2$, and $C_{1,n}$ is the elliptic $n$-fold singularity given by $n$ lines in $\mathbb{A}^{n-1}$. These are all the Gorenstein singularities of arithmetic genus one \cite[Appendix A]{Smyth}. Carrying out the above strategy has led to a proof of homological mirror symmetry for $n$-punctured tori \cite{LPol2}. 

The equivalence \pref{eq:punctured} is an instance
of homological mirror symmetry at the large volume limit.
The equivalence is known to extend to a formal neighborhood of this limit
to give an equivalence
\begin{align} \label{eq:nov}
 \cF \lb \Yv \rb \simeq \perf \hat{\cY} 
\end{align}
over $\bZ\db[q\db]$
where $\Yv$ is the compactification of $\Vv$ and
$\hat{\cY}$ is the Tate elliptic curve,
a formal neighborhood of the nodal curve $Z$
(see \cite{LP2} for a proof).
A general strategy for proving homological mirror symmetry as in \pref{eq:nov}
introduced in \cite{seidelICM}
is to view the categories in \pref{eq:nov}
as deformations of the categories given in \pref{eq:punctured}.
Hence,
in this context
deducing homological mirror symmetry for the compact manifold $\Yv$
from homological mirror symmetry for the Weinstein manifold $\Vv$
ultimately reduces to a problem in deformation theory.

\subsection{New results and a general conjectural picture}

In this paper, we lay out a program that aims to extend the above results to
higher dimensions, leading to new homological mirror symmetry conjectures for
higher-dimensional Calabi--Yau manifolds at the large volume limit and in its
formal neighborhood.  It is based on the relation between homological mirror
symmetry for Calabi--Yau manifolds and homological mirror symmetry for
singularities, which goes back to \cite{MR3044454,MR3294421,MR3223358}.

A weighted homogeneous polynomial
$
 \w \in \bC[x_1,\ldots,x_n]
$
with an isolated critical point at the origin
is \emph{invertible}
if there is an integer matrix
$
 A = (a_{ij})_{i, j=1}^n
$
with non-zero determinant
such that
\begin{align}
 \w = \sum_{i=1}^n \prod_{j=1}^n x_j^{a_{ij}}.
%  \in \bC[x_1,\ldots,x_n]
\end{align}
The corresponding
weight system
$
 (d_1, \ldots, d_n; h)
%  \coloneqq \deg(x_1, \ldots, x_n; f)
$
satisfying $\mathrm{gcd}(d_1,\ldots, d_n, h)=1$ is determined uniquely.
(See the beginning of \pref{sc:whsing} for the definition
of a weight system.)

The \emph{transpose} of $\w$ is defined in \cite{MR1214325} as
\begin{align}
 \wv =  \sum_{i=1}^n \prod_{j=1}^n x_j^{a_{ji}},
\end{align}
whose exponent matrix $\Av$
is the transpose matrix of $A$.
We write
$
 \lb \dv_1, \dv_2,\ldots, \dv_n; \hv \rb
$
for the weight system associated with $\wv$.

The group
\begin{align}
 \Gamma_\w \coloneqq
 \lc (t_1, \ldots, t_n) \in (\bGm)^{n} \relmid
  t_1^{a_{11}} \cdots t_n^{a_{1n}}
   = \cdots
   = t_1^{a_{n1}} \cdots t_n^{a_{nn}}
   \rc \end{align}
acts naturally on $\bA^{n}$.
One has a homomorphism
$
 \phi \colon \bGm \to \Gamma_{\w}
$
sending $t \in \bGm$ to
$
 \lb t^{d_1}, \ldots, t^{d_n} \rb \in \Gamma_{\w}.
$
Let
$
 \emf{\bA^n}{\Gamma_{\w}}{\w}
$
be
the idempotent completion of
the dg category of $\Gamma_{\w}$-equivariant matrix factorizations
of $\w$.

Homological mirror symmetry conjecture for invertible polynomials
is the following:

\begin{conj} \label{cj:HMS_sing}
For any invertible polynomial $\w$,
one has a quasi-equivalence
\begin{align}
 \emf{\bA^n}{\Gamma_{\w}}{\w}
  \simeq \cW(\wv).
\end{align}
\end{conj}

Here $\cW(\wv)$ is the partially wrapped Fukaya category of $\wv$,
which is quasi-equivalent to the Fukaya--Seidel category of
(a Morsification of) $\wv$.
\pref{cj:HMS_sing} is stated for Brieskorn--Pham singularities in 3 variables
in \cite{0604361},
for polynomials in 3 variables associated with a regular system of weights of dual type
in the sense of Saito
in \cite{MR2683215}
(with a prototype appearing earlier in \cite{0711.3907v1}),
and
for invertible polynomials in 3 variables
in \cite{MR2834726}.
It is proved
for $n=2$ in \cite{1903.01351}, and
for Sebastiani--Thom sums of polynomials of type A and D
in \cite{FU1,FU2}.

The conjecture that
$
\emf{\bA^n}{\Gamma_\w}{\w}
$
has a full exceptional collection,
%of length equal to the Milnor number of $\wv$,
which is implied by \pref{cj:HMS_sing},
%since $\cW(\wv)$ always has a full exceptional collection,
is stated in
\cite[Conjecture 1.4]{1809.09940},
%\cite{1809.09940},
and proved
in \cite{2001.06500}.

The following conjecture
is stated for $n=3$ in \cite{MR2834726}:

\begin{conj} \label{cj:tilting}
For any invertible polynomial $\w$,
the category
$
\emf{\bA^n}{\Gamma_\w}{\w}
$
has a tilting object.
\end{conj}

%When $\w$ comes from
%one of Arnold's 14 exceptional unimodal singularities
%and $\Gamma = \phi(\Gm)$,
%a tilting object of $\mf(\bA^n,\w,\Gamma)$ is given
%in \cite{MR2493621,MR2805427}.

A slightly stronger conjecture
that
$
\emf{\bA^n}{\Gamma_\w}{\w}
$
has a full strong exceptional collection,
%of length equal to the Milnor number of $\wv$,
stated in \cite[Conjecture 1.2]{1809.09940},
is known
for $n \le 3$
by \cite{1911.09859},
and
for a class of invertible polynomials called of chain type
by \cite{1809.09940}.

In view of \cite[Theorem 16]{MR2641200},
one may also ask whether
for an invertible polynomial $\w$,
the derived category of coherent sheaves
on the stack
\begin{align}
\eX_\w
\coloneqq \ld \lb \lb \Spec \bC[x_1,\ldots,x_n]/(\w) \rb \setminus \bszero \rb
\middle/ \Gamma_\w \rd
\end{align}
has a tilting object.
If $\w$ is of Brieskorn--Pham type,
then $\eX_\w$ has a full strong exceptional collection
of line bundles
\cite{MR3050701}.
Note that $\eX_\w$ is always a smooth proper rational stack
of Picard number one.
It is known that
for a smooth proper toric Deligne--Mumford stack
of Picard number at most two,
there exists
a full strong exceptional collection
of line bundles
\cite{MR2509327}.
On the other hand,
the stack $\eX_\w$
does not have a full strong exceptional collection
of line bundles in general
--- a counterexample was given in
\cite{2004.04982}.

We write (the Liouville completion of) the Milnor fiber of $\wv$ as
\begin{align}
 \Vv_{\wv}
  \coloneqq \wv^{-1}(1)
  = \lc (x_1,\ldots,x_n) \in \bC^n \relmid \wv = 1 \rc.
\end{align}

The main conjecture that we introduce in this paper is the following:

\begin{conj} \label{cj:affine_HMS}
For any invertible polynomial $\w$,
one has a quasi-equivalence
\begin{align} \label{eq:affine_HMS}
 \emf{\bA^{n+1}}{\Gamma_\w}{\w + x_0 \cdots x_n}
  \simeq \cW \lb \Vv_{\wv} \rb.
\end{align}
\end{conj}

%The integer
%$
% \dv_0 \coloneqq \hv - \sum_{i=1}^n \dv_i
%$
%gives rise to a trichotomy;
The affine variety $\Vv_{\wv}$ is
log Fano,
log Calabi--Yau, or
of log general type
depending on whether
$
 \dv_0 \coloneqq \hv - \sum_{i=1}^n \dv_i
$
is negative, zero, or positive respectively.
In dimension 2,
the log Fano case corresponds
to simple singularities
which have a well-known ADE classification.
Fukaya categories of their Milnor fiber are identified
in \cite{EtLe, EtLe2}
with module categories of the corresponding
(derived) preprojective algebras,
and \pref{cj:affine_HMS} is proved in
\cite{2004.07374}.
The log Calabi--Yau case follows from
homological mirror symmetry
for the wrapped Fukaya categories of the Milnor fibers of
hypersurface cusp singularities
proved in \cite{keating}
by a variation of Orlov's theorem.
%In this case the mirror was conjectured by Gross-Hacking-Keel \cite{GHK} and is the complement of an anticanonical divisor on a certain (non-toric) blow-up of $\mathbb{C}P^2$. This can be identified with $Z_\w$.
% \pref{cj:affine_HMS} is proved for double suspensions of Brieskorn--Pham singularities
% in \cite{2104.10050}.
% A proof
% of a $\bZ/2\bZ$-graded version
% of \pref{cj:affine_HMS}
% is given in \cite{2010.15570}.
In this paper,
we almost exclusively concentrate
on the case of log general type.
See e.g.~\cite[Section 2]{uedasurvey}
%and references therein
for more on this trichotomy in dimension 2.

In the log general type case,
Orlov's theorem gives an equivalence
of the left hand side of \pref{eq:affine_HMS}
with the derived category $\coh \eZ_\w$
of coherent sheaves on
%\begin{align} \label{eq:affine_HMS2}
% \coh Z_\w \simeq \cW \lb \Vv_{\wv} \rb
%\end{align}
%by Orlov equivalence.
\begin{align}
 \eZ_\w \coloneqq \ld \lb \Spec \bC[x_0,\ldots,x_n] / (\w + x_0 x_1 \cdots x_n) \setminus \bszero \rb
   \middle/ \Gamma_\w \rd,
\end{align}
where the action of $\Gamma_\w$
comes from the identification
\begin{align}
 \Gamma_\w \cong
 \lc (t_0,t_1, \ldots, t_n) \in (\bGm)^{n+1} \relmid
  t_1^{a_{11}} \cdots t_n^{a_{1n}}
   = \cdots
   = t_1^{a_{n1}} \cdots t_n^{a_{nn}} = t_0t_1\ldots t_n
   \rc.
\end{align}

Recall that an object $X$ of $\coh \eZ$ on a proper stack $\eZ$ is perfect
if and only if
it is \emph{$\Ext$-finite},
i.e., 
%$\dim \bigoplus_{i \in \bZ} \Hom^i(X, Y) < \infty$
the dimension of $\bigoplus_{i \in \bZ} \Hom^i(X, Y)$ is finite
for any object $Y$.
It is reasonable to expect that the full subcategory
of the wrapped Fukaya category $\cW \lb \Vv_{\wv} \rb$
consisting of $\Ext$-finite objects is equivalent
to the compact Fukaya category $\cF \lb \Vv_{\wv} \rb$,
so that \pref{cj:affine_HMS} would imply
\begin{align} \label{eq:affinec_HMS2}
 \perf \eZ_\w \simeq \cF \lb \Vv_{\wv} \rb.
\end{align}
The first instance of an equivalence of this form was obtained in \cite{LP1} for $\w = x_1^3+x_2^2$ and recently Habermann proved this equivalence when $\w$ is an arbitrary invertible polynomial of two variables \cite{Habermann}.

The way that the wrapped Floer cohomology can be infinite
%$\cW \lb \Vv_{\wv} \rb$ is not $\Ext$-finite
depends on the sign of $\dv_0$;
it can be infinite in the negative cohomological degrees
with finite graded pieces
in the log Fano case,
infinite in finite cohomological degrees
in the log Calabi--Yau case,
and infinite in the positive cohomological degrees
with finite graded pieces
in the log general type case.
In the log Fano and log Calabi--Yau case,
the quotient $\cW(\Vv_{\wv})/\cF(\Vv_{\wv})$ are
generalized cluster categories
(see e.g.~\cite[Section 9]{MR2681708} and references therein).
In the log general type case,
we make the following conjecture,
which is a compact analog of \cite[Conjecture 1.2]{MR3838112}:

\begin{conj} \label{cj:W/F}
Let $\Xv$ be a smooth ample divisor
in a Calabi--Yau manifold $\Yv$
and $\Vv \coloneqq \Yv \setminus \Xv$ be the complement.
Then one has a quasi-equivalence
\begin{align} \label{eq:W/F}
 \cW\lb \Vv \rb / \cF\lb \Vv \rb \simeq \cF(\Xv).
\end{align}
\end{conj}

\pref{cj:W/F} reduces
homological mirror symmetry for the manifold $\Xv$ of general type
to that for the affine manifold $\Vv$.
If $\dv_0 = 1$,
then $\Vv_{\wv}$ admits a compactification
to a Calabi--Yau orbifold $\Yv_{\wv}$
such that
$
 \Xv_{\wv} \coloneqq \Yv_{\wv} \setminus \Vv_{\wv}
$
is a smooth ample divisor,
and \pref{cj:affine_HMS} together with \pref{cj:W/F} implies
\begin{align}
 D^b_\sing \lb \eZ_\w \rb \simeq \cF \lb \Xv_{\wv} \rb.
\end{align}

Recall that
the \emph{degree $d$ trivial extension algebra}
(also known as the \emph{Frobenius completion of degree $d$})
of a finite-dimensional $\bfk$-algebra $A^0$
has $A^0 \oplus \Hom_\bfk(A^0,\bfk)[-d]$
as the underlying graded vector space,
and the multiplication is given by
\begin{align} \label{eq:trivial extension}
(a,f)\cdot (b,g) = (ab, ag + fb).
\end{align}

\begin{thm} \label{th:main}
Let $\w \in \bfk[x_1,\ldots,x_n]$ be a weighted homogeneous polynomial
and $\Gamma$ be a subgroup of $\Gamma_\w$
containing $\phi(\bGm)$ as a subgroup of finite index.
Assume that
\begin{enumerate}
 \item
$\w$ has an isolated critical point at the origin,
 \item
$d_0$ defined by \pref{eq:degree_condition}
is positive,
 \item
%$X$ has a full strong exceptional collection, and
$\emf{\bA^n}{\Gamma}{\w}$ has a tilting object $E$, and
 \item
the pair $(\w, \Gamma)$ does not have twisted deformations
in the sense of \pref{df:cond_U}.
\end{enumerate}
Let $A^0$ be the endomorphism algebra
of the tilting object $E$
and $A$ be the degree $n-1$ trivial extension algebra of $A^0$.
Then there is a $\Gm$-equivariant isomorphism
\begin{align} \label{eq:main}
 U \simto \cU_{\infty}(A) 
\end{align}
of affine schemes from
the affine subspace $U$
of the base space $\Utilde$
of the semiuniversal unfolding of $\w$
defined in \pref{sc:whsing}
to the moduli space of $A_\infty$-structures on $A$
sending the origin $0 \in U$
to the formal $A_\infty$-structure on $A$.
\end{thm}

Although the existence of a tilting object
and the non-existence of twisted deformations
are restrictive assumptions
on a pair $(\w, \Gamma)$,
there are many interesting examples
where both of them holds.
\pref{cj:tilting} states that
the former holds when $\w$ is an invertible polynomial
and $\Gamma = \Gamma_\w$.
We will see examples
where the latter holds
in Sections \ref{sc:fermat}--\ref{sc:EUS}.
%\ref{sc:double_cover}, \ref{sc:sylvester}, and 

%\begin{thm}
%Let $\w$ be a weighted homogeneous polynomial
%defining an exceptional unimodal singularity,
%and $A$ be the finite-dimensional associative graded algebra defined above.
%Then there is a $\Gm$-equivariant isomorphism
%$
% U \simto \cU_{\infty}(A) 
%$
%of affine varieties
%sending the origin $0 \in U$
%to the formal $A_\infty$-structure on $A$.
%\end{thm}

To apply \pref{th:main} to homological mirror symmetry,
one needs to find a generator of the Fukaya category
whose Yoneda algebra is isomorphic to $A$.
When $\w$ is a Sebastiani--Thom sum of polynomials of type A or D,
i.e.,
a decoupled sum of polynomials of the form $x^{n+1}$ or $x^2y + y^{n-1}$,
homological mirror symmetry for singularities
\cite{FU1,FU2}
gives
a collection $(S_i)_{i=1}^\mu$
of Lagrangian spheres in $\Vv_{\wv}$
such that the Yoneda algebra of their direct sum
$
 S = \bigoplus_{i=1}^\mu S_i
$
in the Fukaya category $\cF \lb \Vv_{\wv} \rb$
is isomorphic to the trivial extension algebra
of the tensor product
of the path algebras of the Dynkin quivers of the corresponding types.
%then we have a choice of a generator $\eS_u$ of $\perf Y_u$, such that  algebra $A$ is  In this case Theorem \pref{th:main} holds. The singularities $Q_{10}, Q_{12}, W_{12}, E_{12}, E_{14},U_{12}$ from \pref{tb:unimodal} are of this type,
%as well as $x_1^{n+1}+x_2^{n+1}+\ldots + x_n^{n+1}$, $x_1^2+x_2^{2n}+x_3^{2n}+\ldots +x_n^{2n}$, and many more.
For example,
the algebra $A^0$ in the case of $x^4+y^4+z^4$
is the path algebra of the quiver
%in Figures \ref{fg:ladder37} and
in \pref{fg:ladder444},
with the relations that
the composition of arrows along the sides of each small square commutes.
%This alternative generator has the advantage
%that we know how to construct

%\begin{figure}[h!]
%\centering
%\begin{tikzpicture}
%\tikzset{vertex/.style = {style=circle,draw, fill,  minimum size = 2pt,inner sep=1pt}}
%\tikzset{edge/.style = {->,>=stealth',shorten >=8pt, shorten <=8pt  }}
%
%\node[vertex] at  (1,1) {};
%\node[vertex] at (2,1) {};
%\node[vertex] at (3,1) {};
%\node[vertex] at (4,1) {};
%\node[vertex] at (5,1) {};
%\node[vertex] at (6,1) {};
%\node[vertex] at  (1,2) {};
%\node[vertex] at (2,2) {};
%\node[vertex] at (3,2) {};
%\node[vertex] at (4,2) {};
%\node[vertex] at (5,2) {};
%\node[vertex] at (6,2) {};
%
%
%\draw[edge] (1,1) to (2,1);
%\draw[edge] (2,1) to (3,1);
%\draw[edge] (3,1) to (4,1);
%\draw[edge] (4,1) to (5,1);
%\draw[edge] (5,1) to (6,1);
%\draw[edge] (1,2) to (2,2);
%\draw[edge] (2,2) to (3,2);
%\draw[edge] (3,2) to (4,2);
%\draw[edge] (4,2) to (5,2);
%\draw[edge] (5,2) to (6,2);
%
%\draw[edge] (1,1) to (1,2);
%\draw[edge] (2,1) to (2,2);
%\draw[edge] (3,1) to (3,2);
%\draw[edge] (4,1) to (4,2);
%\draw[edge] (5,1) to (5,2);
%\draw[edge] (6,1) to (6,2);
%
%
%\end{tikzpicture}
%
%\caption{An alternative quiver for the $E_{12}$-singularity
%%with $\w = x^2+y^3+z^7$
%}
%\label{fg:ladder37}
%
%\end{figure}

\begin{figure}[h!]
\centering
\begin{tikzpicture}[scale=0.9]
\tikzset{vertex/.style = {style=circle,draw, fill,  minimum size = 2pt,inner sep=1pt}}
\tikzset{edge/.style = {->,>=stealth',shorten >=8pt, shorten <=8pt  }}

\tikzset{>=latex}

\def \dx{2};
\def \dy{2};
\def \dz{-1.5};
\def \nbx{3};
\def \nby{3};
\def \nbz{3};

\foreach \x in {1,...,\nbx} {
    \foreach \y in {1,...,\nby} {
        \foreach \z in {1,...,\nbz} {
            \node[vertex] at (\x*\dx,\y*\dy,\z*\dz)  {};
        }
    }
}

% z lines
\foreach \x in {1,...,\nbx} {
    \foreach \z in {1,...,\nbz}{
        \foreach \y in {2,...,\nby}{
            \draw [edge](\x*\dx,\y*\dy - \dy,\z*\dz) -- ( \x*\dx , \y*\dy, \z*\dz);
        }
    }
}

% x lines
\foreach \y in {1,...,\nbx} {
    \foreach \z in {1,...,\nbz}{
        \foreach \x in {2,...,\nbx}{
            \draw[edge](\x * \dx - \dx,\y*\dy,\z*\dz) -- ( \x * \dx,\y*\dy,\z*\dz);
        }
    }
}

% y lines
\foreach \x in {1,...,\nbx} {
    \foreach \y in {1,...,\nbz}{
        \foreach \z in {2,...,\nby}{
            \draw[edge](\x*\dx,\y*\dy,\z*\dz - \dz) -- ( \x*\dx,\y*\dy,\z*\dz);
        }
    }
}

\end{tikzpicture}

\caption{A quiver for $\w = x^4+y^4+z^4$
}
\label{fg:ladder444}

\end{figure}

By combining
the proof of a special case of \cite[Conjecture 4]{seidelICM}
which states, under assumptions satisfied for $\Vv_{\wv}$,
an isomorphism
\begin{align}
 \SH^* \lb \Vv_{\wv} \rb \simeq \HH^* \lb \cF \lb \Vv_{\wv} \rb \rb
\end{align}
of the symplectic cohomology
and the Hochschild cohomology of the Fukaya category,
with the computation of the symplectic cohomology
$\SH^* \lb \Vv_{\wv} \rb$
using a spectral sequence,
originally due to McLean \cite{mcleantalk}
and full detail of which was written later
by Ganatra and Pomerleano \cite{1811.03609}
(who in addition proved that this spectral sequence is multiplicative),
we show that
the Yoneda $A_\infty$-algebra $\cA$
of the generator of the Fukaya category is not formal.
%
%In other words, at the chain level $\cA_\w$ is a not given by the trivial extension algebra construction. There are non-trivial higher products, hence $\cA_\w$ is a non-trivial extension of its directed subalgebra. There is a conceptual explanation of this phenomenon via homological mirror symmetry conjectures as provided in \pref{sc:invertible}.
%
Hence $\cA$ can be identified with a point in the moduli space
\begin{align} \label{eq:moduli of non-formal A-infinity structures}
 \cM_\infty (A) \coloneqq \ld \lb \cU_\infty(A) \setminus \bszero \rb \middle/ \Gm \rd
\end{align}
of non-formal $A_\infty$-structures.
\pref{cj:affine_HMS} identifies exactly which point this is,
and in order to prove it,
one has to distinguish points on $\cM_\infty(A)$
by computable invariants of $\cF \lb \Vv_{\wv} \rb$.
For
$
 \w = x_1^{n+1}+x_2^{n+1}+\cdots+x_n^{n+1}
$
and
$
 \w = x_1^2+x_2^{2n}+\cdots+ x_n^{2n},
$
this space is one-dimensional,
and we can prove \pref{cj:affine_HMS}
by computing the dimensions of the Hochschild cohomologies
in this case:

\begin{thm} \label{th:HMS}
\begin{enumerate}[(i)]
 \item
Let
\begin{align}
 \Vv \coloneqq \lc (x_1,x_2,\ldots, x_n) \in \bC^n \relmid
  x_1^{n+1}+x_2^{n+1}+\cdots+x_n^{n+1} = 1 \rc
\end{align}
be the Milnor fiber considered as an exact symplectic manifold, and
\begin{align}
 K \coloneqq
  \lc \ld \diag(t_0,t_1,\ldots, t_n) \rd \in \PGL_{n+1}(\bC) \relmid
   t_1^{n+1} = \cdots = t_n^{n+1} = t_0 t_1\cdots t_n =1 \rc
\end{align}
be a finite group
acting on the projective hypersurface
\begin{align}
 Z \coloneqq \lc [x_0:x_1:\cdots:x_n] \in \P^{n} \relmid x_1^{n+1}+x_2^{n+1}+\cdots+x_n^{n+1}+x_0x_1\cdots x_n=0\rc.
\end{align}
Then we have quasi-equivalences
\begin{align} \label{eq:HMS_Fermat1}
 \cF \lb \Vv \rb \simeq \perf \ld Z/K \rd
\end{align}
and
\begin{align} \label{eq:HMS_Fermat2}
  \cW \lb \Vv \rb \simeq \coh \ld Z/K \rd
\end{align}
of pretriangulated $A_\infty$-categories over $\bC$. \\
 \item
Let
\begin{align}
 \Vv \coloneqq \lc(x_1,x_2,\ldots,x_n)\in \bC^n \relmid x_1^2+x_2^{2n}+\cdots+x_n^{2n} =1 \rc
\end{align}
be the Milnor fiber considered as an exact symplectic manifold, and
\begin{align}
 K \coloneqq \lc \ld \diag(t_0,\ldots, t_n) \rd \in \Aut \bP \relmid
  t_1^2 = t_2^{2n} = \cdots =  t_n^{2n} = t_0t_1\cdots t_n =1 \rc
\end{align}
be a finite group
acting on the weighted projective hypersurface
\begin{align}
 Z \coloneqq \lc \ld x_0:x_1:\cdots:x_n \rd \in \bP \relmid
  x_1^2+x_2^{2n}+\cdots+x_n^{2n}+x_0x_1\cdots x_n=0 \rc,
\end{align}
where $\bP = \bP(1,n,1,\ldots,1)$ is a weighted projective space
considered as a smooth stack.
Then we have quasi-equivalences
\begin{align}
 \cF \lb \Vv \rb \simeq \perf \ld Z/K \rd
\end{align}
and
\begin{align}
 \cW \lb \Vv \rb \simeq \coh \ld Z/K \rd
\end{align}
of pretriangulated $A_\infty$-categories over $\bC$.
\end{enumerate}
\end{thm}

\subsection{The relation with results of Seidel and Sheridan}

%We expect that versality results of \cite{sheridanversality} implies that the homological mirror symmetry result in Theorem \ref{lcslquart} could be extended to a formal neighborhood over $\C[[\mathbf{q}]]$ (see \pref{sc:invertible} but we leave that to a future work.). We would like to point out that 
The large complex structure limits in \pref{th:HMS} are different
from those appearing in \cite{MR3364859}
and its generalizations
\cite{%MR2863919,
%MR2975297,
MR3294958, sheridansmith}.
In his construction,
Seidel removes the divisor $\{x_1x_2x_3=0\}$
from the Milnor fiber $\Vv$ on the $A$-side and
considers the reducible singular variety $\{x_0x_1x_2x_3=0\}$
instead of $Z$ on the $B$-side (cf.~\cite[Section 5]{LPol3}).

%The divisor $\Xv = \Yv \setminus \Vv_{\wv}$ at infinity
%is a smooth curve of genus 3 in \pref{th:hms}(i)
%and a smooth curve of genus 2 in \pref{th:main}(ii).
%\pref{cj:affine_HMS} and \pref{cj:W/F} implies
%$\cF(\Xv) \cong D^b_\sing(Z_{\wv})$,
%which should also be proved along the lines of
%\cite{MR2819674,MR2914956}.

%The Fukaya category $\cF(\Yv_{\wv})$
%% of $\Yv_{\wv}$
%is a deformation of
%%the Fukaya category
%$\cF(\Vv_{\wv})$
%%of $\Vv_{\wv}$
%\cite{seidelICM}.
The generator $S$ of $\cF \lb \Vv \rb$
used in the proof of \pref{th:HMS}.(i)
is the direct sum of vanishing cycles
of the Lefschetz fibration
$
\wv = x_1^{n+1}+\cdots+x_n^{n+1} \colon \bC^{n+1} \to \bC,
$
which is also an object of $\cF \lb \Yv \rb$.
The Yoneda algebra computed in $\cF\lb \Yv \rb$ is a deformation
\cite{seidelICM}
of the Yoneda algebra $A$ computed in $\cF(\Vv)$,
and hence isomorphic to it since $\HH^2(A)_0 \cong 0$,
so that the Yoneda $A_\infty$-algebra computed in $\cF\lb \Yv \rb$
is described by a Novikov ring-valued point of $\cM_\infty(A)$,
which is the open-string mirror map.
%and opens up a way to prove homological mirror symmetry
%for $\Yv$,
%e.g.~, by using versality of the relative Fukaya category
%\cite{1709.07874}
%and automatic split-generation of the Fukaya category
%\cite{1605.07702,1712.03924}.
%Since they depend on announced works
%of Perutz--Sheridan and Abouzaid--Fukaya--Oh--Ohta--Ono
%respectively,
%we do not discuss it here.

The generator used by Seidel in \cite{MR3364859}
is the direct sum of the vanishing cycles
of the Lefschetz fibration
$
\wv' \coloneqq (\wv+1)/(x_1 x_2 x_3) \colon (\bCx)^3 \to \bC
$
mirror to the toric variety
whose fan polytope is polar dual to that of $\bP^3$.
The generator used by Sheridan in \cite{MR3294958}
is the cover of an immersed Lagrangian sphere
in a pair of pants,
which is shown to be the direct sum of vanishing cycles
of the Lefschetz fibration
$
\wv' \coloneqq (\wv+1)/(x_1 \cdots x_n) \colon (\bCx)^n \to \bC
$
in \cite{MR2975297}.
One has
$
\Dv
\coloneqq
\Vv
\setminus
(\wv')^{-1}(0)
=
\Dv_1 \cup \cdots \cup \Dv_n
$
where
$
\Dv_i \coloneqq \Vv \cap \{ x_i = 0 \}.
$
% for $i=1,\ldots,n$
% are smooth divisors.
Let
$\cF(\Vv, \Dv)$
be the relative Fukaya category,
which is an $A_\infty$-category
over $\bC[q_1,\ldots,q_n]$
whose objects are Lagrangian submanifolds
of $\Vv \setminus \Dv$
and compositions are counted
with intersection numbers with $\Dv_i$.
Since $\Vv$ is Stein,
the definition of $\cF(\Vv,\Dv)$
involves only the classical theory of pseudo-holomorphic maps,
and the coefficient ring is a polynomial ring.
The argument of Seidel
% \cite{MR3364859}
and Sheridan
% \cite{MR3578916,1709.07874}
shows that
the idempotent-complete pretriangulated $A_\infty$-category
generated by the full subcategory of
$
\cF(\Vv,\Dv)
$
consisting of the cover of the immersed Lagrangian sphere
is equivalent to
$
\perf \ld \mathcal{Z}/K \rd
$
where
$
\mathcal{Z}
\coloneqq
\Proj
\bC[q_1,\ldots,q_n][x_1,\ldots,x_n]
/
(q_1 x_1^{n+1} + \cdots + q_n x_n^{n+1} + x_0 \cdots x_n).
$
% Moreover,
% the same argument works
% for any invertible polynomial
% (with a little care
% in the case $n =2$
% where $\Dv_i$ may be empty).
% Note that 
% the symmetry with respect to the covering transformation group
% suffices,
% and one does not need to use
% the signed group action
% induced by anti-holomorphic involutions.
This suggests generalizations of \pref{cj:affine_HMS}
to more general partial compactifications
of covers of a pair of pants.

% The advantage of our approach,
% compared with that via covers of a pair of pants,
% is the relation with singularity theory.
Even if one's goal is to prove homological mirror symmetry
for a compact Calabi--Yau manifold
over the Novikov field,
it is useful
not to go directly
from a cover of a pair of pants
to the compact Calabi--Yau,
but to divide it into two steps,
first to the Milnor fiber
and then to the compact Calabi--Yau:
The Fukaya category of a cover of a pair of pants
has many deformations,
but it is easy to control the deformation
to the Milnor fiber,
essentially because the Milnor fiber is Stein
and the deformation is locally constant
along a stratification of the base space.
% of the deformation,
% which can be taken to be an affine space.
Once one comes to the Milnor fiber,
and take the direct sum of vanishing cycles
as a generator,
then we can understand
not only formal deformations
but the global moduli space of $A_\infty$-structures.
It is an interesting problem
to obtain the same level of understanding
for deformations of the Fukaya category
of a cover of a pair of pants,
which would have non-smoothing components in general.

\subsection{Moduli of lattice polarized K3 surfaces}

Special cases of the moduli space
\pref{eq:moduli of non-formal A-infinity structures}
give modular compactifications
of moduli spaces
of a certain class of lattice polarized K3 surfaces.
The point is that
the choice of a generator $\eS$
and
an isomorphism
$\psi \colon \End \eS \simto A$
with a fixed graded algebra $A$
is a derived category analog
of a choice of a lattice polarization.
Similar identification
of a choice of a full strong exceptional collection
as an analog of a choice of a marking
(an isomorphism of the Picard lattice
with a fixed lattice)
of a del Pezzo surface
was a starting point of
\cite{1411.7770,1903.06457}.

Let $P$
be a \emph{lattice},
i.e.,
a free abelian group
equipped with a symmetric bilinear form.
A \emph{$P$-polarized K3 surface}
is a pair $(Y, j)$
of a K3 surface
and a primitive lattice embedding
$j \colon P \hookrightarrow \Pic Y$.
It follows from
the global Torelli theorem
and the surjectivity of the period map
that the coarse moduli space of
$P$-polarized K3 surfaces
is the quotient of a symmetric domain of type IV
by a discrete group.
As an example,
consider the case
$
P = E_8 \bot U.
$
This is the complement
of $U$
of the `half'
% $E_8 \bot U \bot U$
of the extended K3 lattice
$
E_8 \bot E_8 \bot U \bot U \bot U \bot U,
$
and as such is self-mirror,
since mirror symmetry for lattice polarized K3 surfaces
interchanges the algebraic lattice
and the transcendental lattice
inside the extended K3 lattice
\cite{MR1420220}.
The Satake--Baily--Borel compactification
of the coarse moduli space
of $E_8 \bot U$-polarized K3 surfaces
is known to be the 10-dimensional weighted projective space
$\mathbf{P}(\bsw)$
of weight
$\bsw = (4, 10, 12, 16, 18, 22, 24, 28, 30, 36, 42)$
\cite{MR0642697}.
Similar
descriptions exist
for lattices
coming from exceptional unimodal singularities
by \cite{MR761312},
which lead to a `striking'
(\cite[page 586]{MR2003125})
conclusion that certain rings of meromorphic
automorphic forms are polynomial rings.
\pref{th:main}
together with the discussion
in \pref{sc:EUS}
gives an interpretation
of the spectrum of all of these polynomial rings
as moduli spaces of $A_\infty$-structures.
This is a K3 analog of the description
of $\cMbar_{1,1}$
as moduli of $A_\infty$-structures
recalled in \pref{sc:moduli of elliptic curves}.
Similarly,
the coarse moduli space
of
\pref{eq:moduli of non-formal A-infinity structures}
for the $n=3$ case of
\pref{th:HMS}.(i)
can be identified
with the coarse moduli space of
$E_8 \bot E_8 \bot U \bot \la -4 \ra$-polarized
K3 surfaces.
This is a K3 analog
of the Hesse pencil of cubic curves,
which are elliptic curves with level 3 structures.
These examples are the first
of infinite series,
discussed in
\pref{sc:sylvester}
and \pref{sc:fermat}
respectively,
where \pref{th:main} applies.

\subsection{Sebastiani--Thom summation}

Yet another motivation
for \pref{cj:affine_HMS},
besides moduli of $A_\infty$-structures
and partial compactifications of covers of a pair of pants,
comes from a conjectural compatibility
of
\pref{cj:HMS_sing}
and
\pref{cj:affine_HMS}
under the Sebastiani--Thom summation.
Let
$
\wv_i \colon \bC^{n_i} \to \bC^1
$
for $i=1,2$
be Lefschetz fibrations
coming from transpositions of invertible polynomials
$
\w_i \colon X_i \coloneqq \bA^{n_i} \to \bA^1
$
and
\begin{align}
  Y_i \coloneqq \lc (x_{i,1},\ldots, x_{i,n_i}) \in \bA^{n_i}
   \relmid x_{i,1} \cdots x_{i,n_i} = 0 \rc
\end{align}
be the unions of coordinate hyperplanes.
We also write the union of coordinate hyperplanes in
$
X
\coloneqq X_1 \times X_2
% \cong \bA^{n_1+n_2}
$
as $Y$.
Let
$
\w = \w_1 + \w_2 \colon \bA^{n_1+n_2} \to \bA^1
$
be the Sebastiani--Thom sum
of $\w_1$ and $\w_2$,
and set
\begin{align}
\Gamma
\coloneqq
  \lc
  ((t_{1,0},\ldots,t_{1,n_1}),(t_{2,0},\ldots,t_{2,n_2}))
    \in \Gamma_1 \oplus \Gamma_2
  \relmid
    t_{1,0} = t_{2,0}
  \rc
\end{align}
where
$
\Gamma_i \coloneqq \Gamma_{\w_i}.
$
It follows from \cite{MR2437083} that
\begin{align}
\emf{\bA^{n_i+1}}{\w_i + x_{i,0} \cdots x_{i,n_i}}{\Gamma_i}
  \simeq \emf{Y_i}{\w_i}{\Gamma_i}
\end{align}
The push-out diagram
\begin{align} \label{eq:push-out for X and Y}
\begin{CD}
Y_1 \times Y_2 @>>> X_1 \times Y_2 \\
@VVV @VVV \\
Y_1 \times X_2 @>>> Y
\end{CD}
\end{align}
should induce the push-out diagram
\begin{align}
\begin{CD}
\emf{Y_1 \times Y_2}{\Gamma}{\w} @>>> \emf{Y_1 \times X_2}{\Gamma}{\w} \\
@VVV @VVV \\
\emf{X_1 \times Y_2}{\Gamma}{\w} @>>> \emf{Y}{\Gamma}{\w},
\end{CD}
\end{align}
which gives
\begin{align}
\begin{CD}
\emf{Y_1}{\Gamma_1}{\w_1} \otimes \emf{Y_2}{\Gamma_2}{\w_2} @>>> \emf{Y_1}{\Gamma_1}{\w_1} \otimes \emf{X_2}{\Gamma_2}{\w_2} \\
@VVV @VVV \\
\emf{X_1}{\Gamma_1}{\w_1} \otimes \emf{Y_2}{\Gamma_2}{\w_2}
@>>> \emf{Y}{\Gamma}{\w}
\end{CD}
\end{align}
by the Sebastiani--Thom theorem for matrix factorizations
\cite{1101.5834}.
This matches
the push-out diagram
\begin{align}
\begin{CD}
 \cW(\wv_1^{-1}(0)) \otimes \cW(\wv_2^{-1}(0)) @>>> \cW(\wv_1^{-1}(0)) \otimes \cW(\wv_2) \\
 @VVV @VVV \\
 \cW(\wv_1) \otimes \cW(\wv_2^{-1}(0)) @>>> \cW \lb (\wv_1+\wv_2)^{-1}(0) \rb
\end{CD}
\end{align}
coming from the cosheaf property
of the wrapped Fukaya categories \cite{1809.03427}.

\begin{remark}
Similar compatibility exists
for homological mirror symmetry
for toric Fano manifolds
and that for their toric boundaries
giving large complex structure limits
of their anti-canonical Calabi--Yau hypersurfaces.
If
$
\wv_i \colon (\bCx)^{n_i} \to \bC
$
for $i=1,2$
are mirror to toric Fano manifolds $X_i$
with toric boundaries $Y_i$
and
$
\wv \coloneqq \wv_1+\wv_2 \colon (\bCx)^{n_1+n_2} \to \bC
$
is mirror to
$
X \coloneqq X_1 \times X_2
$
with its toric boundary $Y$,
then one has the push-out diagram
\pref{eq:push-out for X and Y}
inducing
the push-out diagram
\begin{align}
  \begin{CD}
    \coh Y_1 \otimes \coh Y_2 @>>> \coh X_1 \otimes \coh Y_2 \\
    @VVV @VVV \\
    \coh Y_1 \otimes \coh X_2 @>>> \coh Y
  \end{CD}
\end{align}
obtained from \cite[Theorem 8.A.1.2]{MR3701353}
as explained in \cite[Section 1.1.2]{1707.02959}
(see also \cite[Section 2.4]{1604.00114}).
\end{remark}

\subsection{}

This paper is organized as follows:
In \pref{sc:whsing},
we set up basic notations for weighted homogeneous polynomials and their semiuniversal unfoldings.
In \pref{sc:HH_mf},
we compute Hochschild cohomologies of (not necessarily smooth) proper algebraic stacks
associated with weighted homogeneous polynomials
using matrix factorizations.
In \pref{sc:generator},
we give a generator $\eS$ of $\perf \scrY$,
and prove the formality of $\dgend \eS_0$.
We prove \pref{th:main} in \pref{sc:moduli_K3}.
In \pref{sc:HHSH},
we prove that $\HH^* \lb \cF \lb \Vv \rb \rb$ is isomorphic to the symplectic cohomology of $\Vv$.
In \pref{sc:SH},
we give computations of symplectic cohomology of $\Vv$ and
deduce the non-formality result in $\cF \lb \Vv \rb$.
\pref{th:HMS} is proved in \pref{sc:HMS}.

Through the rest of the paper,
we will work over an algebraically closed field $\bfk$
of characteristic 0.
The bounded derived category of coherent sheaves,
its full subcategory consisting of perfect complexes,
and the unbounded derived category of quasi-coherent sheaves
on an algebraic stack $\eY$,
considered as pretriangulated dg categories,
will be denoted by $\coh \eY$, $\perf \eY$, and $\Qcoh \eY$ respectively.
All Fukaya categories are completed with respect to cones and direct summands.

\textit{Acknowledgment}: We are grateful to the referee for their careful reading and corrections. In particular, Section 6 has been thoroughly revised to address referee's comments. Y.~L.~is partially supported by the Royal Society URF\textbackslash R\textbackslash180024. K.~U.~is partially supported by Grant-in-Aid for Scientific Research
(15KT0105, 16K13743, 16H03930).

\section{Weighted hypersurface singularities}
 \label{sc:whsing} 

%\subsection{Generalities on weighted homogeneous hypersurface singularities}
% \label{sc:generalities}
%
%We recall some basic facts about weighted homogeneous hypersurface singularities,
%and at the same time introduce the setting that we will work in throughout.
%containing roots of unity.
%and use notations close to
%those in \cite[Section 2]{PolVai}.
% and \cite{BFK}.
%We use a mixture of notations from \cite[Section 2]{PolVai} and \cite{BFK}.

%Let
%$
% S \coloneqq \Sym V
%%  \cong \bfk[x_1,\ldots,x_n] 
%$
%be the symmetric algebra
%over the vector space
%$
% V \coloneqq \vspan_{\bfk} \lc x_1, \ldots, x_n \rc
%$
%of dimension $n$, and
%$
% \bA^n = \Spec S
%$
%be the affine space.
A \emph{weight system} is a sequence
$(d_1, \ldots, d_n;h)$
of positive integers satisfying
\begin{align}
 h > \max \lc d_1,\ldots,d_n \rc.
\end{align}
We will always assume
\begin{align} \label{eq:reduced}
 \gcd(d_1, \ldots, d_n, h) = 1
\end{align}
in this paper.
Let $\w(x_1,\ldots,x_n) \in \bfk[x_1,\ldots,x_n]$ be a polynomial in $n$ variables,
which is weighted homogeneous of weight $(d_1, \ldots, d_n; h)$;
\begin{align}
 \w \lb t^{d_1} x_1, \ldots, t^{d_n} x_n \rb
  = t^h \w(x_1,\ldots,x_n), \quad
  t \in \Gm.
\end{align}
It is written as the sum of monomials
\begin{align}
 \w(x_1,\ldots, x_n)
  = \sum_{\bsi=(i_1,\ldots, i_n) \in I_\w}  c_{\bsi} x_1^{i_1} x_2^{i_2} \ldots x_n^{i_n},
 \quad c_{\bsi} \in \Gm,
\end{align}
where the index set $I_\w$ is a subset
of the set of non-negative integers satisfying
\begin{align}
 d_1 i_1 + d_2 i_2 + \cdots + d_n i_n = h.
\end{align}
We will always assume that $\w$ determines
the weight system satisfying \pref{eq:reduced} uniquely.

Let $\Gamma_\w$ be the commutative algebraic group defined by
\begin{align}
 \Gamma_\w \coloneqq
  \lc (t_1,\ldots, t_{n+1}) \in \Gm^{n+1}
   \relmid t_1^{i_1} t_2^{i_2}\cdots t_n^{i_n} = t_{n+1}  \text{\ for all\ } (i_1,\ldots, i_n) \in I_\w \rc.
\end{align}
The group
$
 \Gammahat_\w \coloneqq \Hom(\Gamma_\w, \Gm)
$
of characters of $\Gamma_\w$ is written as
\begin{align}
 \Gammahat_\w
  = \left. \bZ \chi_1 \oplus \cdots \oplus \bZ \chi_{n+1} \middle/ \lb i_1 \chi_1 + \cdots + i_n \chi_n - \chi_{n+1} \rb_{\bsi \in I_\w} \right.,
\end{align}
where $\chi_i \in \Gammahat_\w$ for $1 \le i \le n+1$ is defined by $(t_1, \ldots, t_{n+1}) \mapsto t_i$.
Since the composition $\Gamma_\w \hookrightarrow \Gm^n \times \Gm \to \Gm^n$
with the first projection is injective,
we will think of $\Gamma_\w$ as a subgroup of $\Gm^n$,
and set $\chi_\w \coloneqq \chi_{n+1}$.
The group $\Gamma_\w$ consists of diagonal transformations of $\bA^n$
which keeps $\w$ semi-invariant;
\begin{align}
 \w ( t \cdot (x_1,\ldots x_n) ) = \chi_\w(t) \w(x_1,\ldots, x_n),
  \quad t \in \Gamma_{\w}.
\end{align}
%We set $K_\w \coloneqq \ker \chi$,
%so that we have an exact sequence
%\begin{align}
% 1 \to K_\w \to \Gamma_\w \xrightarrow{\chi} \G_m \to 1
%\end{align}
%of algebraic groups.
%where $K_\w$ is the commutative algebraic group defined by
%\begin{align}
%K_\w \coloneqq \lc (t_1,\ldots, t_n) \in \Gm^n \relmid
% t_1^{i_1} t_2^{i_2} \ldots t_n^{i_n} = 1 \text{\ for all \ } (i_1, \ldots i_n) \in I_\w \rc.
%\end{align}
The injective homomorphism
\begin{align}
 \phi \colon \G_m &\to \Gamma_\w, \quad
 t \mapsto \lb t^{d_1},\ldots, t^{d_n} \rb
\end{align}
fits into the exact sequence
\begin{align}
 1 \to \Gm \xrightarrow{\phi} \Gamma_\w \to \ker \chi_\w / \langle j_\w \rangle \to 1,
\end{align}
where
$
 j_\w \coloneqq \lb e^{2\pi \sqrt{-1} d_1/h}, \ldots, e^{2 \pi \sqrt{-1} d_n/h} \rb
$
is the \emph{grading element}
generating the cyclic group
$
 \ker \chi_\w \cap \phi(\Gm)
$
of order $h$.

Let $\Gamma$ be a subgroup of $\Gamma_\w$
containing $\phi(\G_m)$ as a subgroup of finite index.
For such $\Gamma$,
%we have an exact sequence
%\begin{align}
% 1 \to K \to \Gamma \xrightarrow{\chi} \Gm \to 1,
%\end{align}
%where
the kernel of $\chi \coloneqq \chi_\w|_\Gamma$ is a finite group,
and such subgroups $\Gamma$ are in bijection
with finite subgroups of $\ker \chi_\w$
containing the grading element $j_\w$.
%\begin{example} 
%When $\w \colon \bA^2 \to \bA$ is given by $(x,y) \mapsto x^2 + y^3$,
%we have $n = 2$,
%$(d_1, d_2; h) = (3, 2; 6)$,
%$I_\w = \{ (2,0), (0,3) \}$,
%$\Gamma_\w \coloneqq \lc (t_1, t_2) \in \Gm^2 \relmid t_1^2 = t_2^3 \rc \cong \Gm$,
%$j_\w = (\exp(2 \pi \sqrt{-1}/3, \exp(2 \pi \sqrt{-1}/2)$, and
%$K_\w = \la j_\w \ra \cong \Z/6\Z$.
%The only choice for $\Gamma$ is the whole of $\Gamma_\w$ in this case.
%\end{example}

The group $\Gamma$ acts naturally on
the spectrum of
$
 \Rbar \coloneqq \bfk[x_1,\ldots,x_n]/(\w),
$
and
we write the quotient stack
of the complement of the origin $\bszero$ as
\begin{align} \label{eq:X}
 \eX \coloneqq \ld \lb \Spec \Rbar \setminus \bszero \rb \middle/ \Gamma \rd.
\end{align}
We let $\Gamma$ act on
$
 \bA^{n+1} \coloneqq \Spec \bfk[x_0,\ldots,x_n]
$
diagonally via $\chi_0 \oplus \cdots \oplus \chi_n$
%\begin{align}  \label{eq:CY_weights}
% (t_1,\ldots, t_n) \cdot (x_0, x_1,\ldots, x_n)
%  = (\chi_0 x_0, t_1x_1, \ldots, t_n x_n)
%\end{align}
where
\begin{align} \label{eq:CY_weight}
% \chi_0(t_1,\ldots,t_n) \coloneqq \chi(t_1,\ldots,t_n) t_1^{-1} \cdots t_n^{-1}.
 \chi_0 \coloneqq \chi - \chi_1 - \cdots - \chi_n.
\end{align}
By abuse of notation,
we write the image of $\w$
by the inclusion
of $\bfk[x_1,\ldots,x_n]$ to $\bfk[x_0,\ldots,x_n]$
%$
% \bfk[x_1,\ldots,x_n] \hookrightarrow \bfk[x_0,x_1,\ldots,x_n]
%$
by the same symbol,
and set
$
 R \coloneqq \bfk[x_0,\ldots,x_n] / (\w).
$

%\begin{definition} \label{df:cond_P}
%We say $\w$ satisfies \emph{Condition P}
%if
%\begin{align} \label{eq:degree_condition}
% d_0 \coloneqq h - d_1 - \cdots - d_n
%\end{align}
%is positive.
%\end{definition}
If
\begin{align} \label{eq:degree_condition}
 d_0 \coloneqq h - d_1 - \cdots - d_n
\end{align}
is positive,
then
$
 \ld \lb \bA^{n+1} \setminus \bszero \rb \middle/ \Gamma \rd
$
is proper, and hence so is its closed substack
\begin{align}
 \eY_0 \coloneqq \ld \lb \Spec R \setminus \bszero \rb \middle/ \Gamma \rd.
\end{align}
Here, the subscript $``0"$ is placed
in anticipation of the deformation
that we will study later on.
It is a projective cone over $\eX$,
which is obtained from
$
 \eV_0 \coloneqq \ld \Spec \Rbar \middle/ \ker \chi_0 \rd
$
by adding $\eX$ at infinity.
The character of the $\Gamma$-action on the $x_0$ variable
in \pref{eq:CY_weight}
is chosen so that the dualizing sheaf of $\eY_0$ is trivial.

Assume that
$
 \w \colon \bA^n \to \bA
$
has an isolated critical point at the origin.
This is equivalent to the finiteness
of the dimension $\mu$,
called the \emph{Milnor number} of $\w$,
of the Jacobi algebra
\begin{align}
 \Jac_\w \coloneqq \bfk[x_1,\ldots,x_n]/(\partial_1 \w, \ldots, \partial_n \w).
\end{align}
Let $J_\w$ be the set
of exponents
of monomials representing a basis of $\Jac_\w$,
and
\begin{align}
 \wtilde \coloneqq \w(x_1,\ldots,x_n)
  + \sum_{\bsj = (j_1,\ldots, j_n) \in J_\w} u_{\bsj} x_1^{j_1} \ldots x_n^{j_n}
  \colon
%   \bA_{\Utilde}^n \coloneqq
   \bA^n \times \Utilde \to \bA^1
\end{align}
be a semiuniversal unfolding of $\w$.
The base space
$
 \Utilde \coloneqq \Spec \bfk[u_1,\ldots,u_\mu]
$
is an affine space of dimension $\mu$.
Let $U$ be the affine subspace of $\Utilde$
defined by the condition that $u_\bsj$ may be non-zero
only if there exists a positive integer $w_\bsj$
satisfying
\begin{align} \label{eq:equivunfold}
 \chi = w_\bsj \chi_0 + j_1 \chi_1 + \cdots + j_n \chi_n.
% \chi^{w_{\bsj}-1}= t_1^{w_\bsj-j_1} t_2^{w_\bsj-j_2} \ldots t_n^{w_\bsj-j_n}.
\end{align}
Let $J$ be the set of $\bsj \in J_\w$
satisfying this condition.
Then we have the family
\begin{align} \label{eq:scrY}
 \pi_\scrY \colon \scrY \coloneqq \ld \lb \bfW^{-1}(0) \setminus \lb \bszero \times U \rb \rb \middle/ \Gamma \rd \to U
\end{align}
of stacks over $U$ defined by
\begin{align} \label{eq:bfW}
 \bfW
  \coloneqq \w(x_1,\ldots,x_n) + \sum_{\bsj \in J} u_{\bsj} x_0^{w_\bsj}  x_1^{j_1} \ldots x_n^{j_n}
  \colon \bA^{n+1} \times U \to \bA^1,
\end{align}
whose fiber
over $u \in U$
will be denoted by
$
 \eY_u \coloneqq \pi^{-1}(u).
$
Here the action of $\Gamma$ on $\bA^{n+1} \times U$
in such a way that $\deg x_i = \chi_i$ for $i = 0, 1, \ldots, n$ and
$\deg u_\bsj = 0$ for all $\bsj \in J_{\w}$.
The divisor at infinity
defined by $x_0$
is isomorphic to $\eX \times U$.
The relative dualizing sheaf $\omega_{\scrY/U}$ is identified with
$
 \omega_{(\bfW^{-1}(0) \setminus (\bszero \times U))/U}
$
considered as a $\Gamma$-equivariant coherent sheaf,
which in turn is isomorphic to the restriction of
$
 \omega_{(\bA^{n+1} \times U)/U} (\chi)
$
to
$
 \bfW^{-1}(0) \setminus (\bszero \times U)
$
since $\bfW$ is a section of $\cO_{\bA^{n+1} \times U}$ of degree $\chi$.
This sheaf is $\Gamma$-equivariantly trivial,
and we fix its trivialization,
which is unique up to scaling if $d_0 > 0$.
In addition,
there is a $\bGm$-action on $\bA^{n+1} \times U$
given by
\begin{align} \label{eq:Gm-action_on_Y}
 \lb (x_0,x_1,\ldots,x_n),(u_\bsj)_{\bsj \in J} \rb
  \mapsto \lb (t^{-1} x_0,x_1,\ldots,x_n), \lb t^{w_{\bsj}} u_\bsj \rb_{\bsj \in J} \rb,
\end{align}
which induces actions on $\scrY$ and $U$
which makes $\pi_\scrY$ equivariant.

\begin{example}[tacnode] \label{tacnode}
When
$n = 2$
and
$\w = x^2 + y^4$,
one has
$(d_1, d_2; h) = (2,1;4)$
and
\begin{align}
 \Gamma_\w
  \coloneqq \lc (t_1, t_2) \in \Gm^2 \relmid t_1^2 = t_2^4 \rc
  \simto \Gm \times \bmu_2, \quad
 (t_1,t_2) \mapsto (t_2, t_1 t_2^{-2}). 
\end{align}
The image of the injective homomorphism
\begin{align}
 \phi \colon \Gm \to \Gamma_\w, \quad
  t \mapsto (t^2, t)
\end{align}
is an index 2 subgroup isomorphic to $\Gm$,
so that there are two choices of $\Gamma$.
By construction, we have the semi-invariance property
\begin{align}
 \w(t_1 x, t_2 y) = \chi(t_1,t_2) \w(x,y),
\end{align}
where
$
 \chi \colon \Gamma \to \Gm
$
is the character sending $(t_1,t_2)$ to $t_1^2 = t_2^4$. 
A semiuniversal unfolding of $\w$ is given by
\begin{align}
 \wtilde (x,y; u_2,u_3,u_4) = x^2 + y^4 + u_2 y^2  + u_3 y  + u_4,
\end{align}
%One has
%\begin{align}
% \bfWtilde (x,y,z; u_2,u_3,u_4) = x^2 + y^4 + u_2 y^2 z^2  + u_3 y z^3 + u_4 z^4
%\end{align}
%and
%$
% \bfW = \bfWtilde.
%$
%The $\Gamma_\w$-action on $\bA^3$ is given by
%\begin{align}
% (t_1, t_2) \cdot (x,y,z) = (t_1 x, t_2 y, t_1 t_2^{-1} z),
%\end{align}
and one has
\begin{align}
 \bfW(x,y,z; u_2, u_3, u_4) = x^2 + y^4 + u_2 y^2 z^2 + u_3 y z^3 + u_4 z^4
\end{align}
if $\Gamma = \phi(\Gm)$, and
\begin{align}
 \bfW(x,y,z; u_2, u_4) = x^2 + y^4 + u_2 y^2 z^2 + u_4 z^4.
\end{align}
if $\Gamma = \Gamma_\w$.
\end{example}

\begin{example}[$E_{12}$-singularity]
When $n = 3$
and
$
 \w(x,y,z) = x^2 + y^3 + z^7,
$
one has
$
 (d_1, d_2, d_3; h) = (21, 14, 6; 42),
$
%$I_\w = \{(2,0,0),(0,3,0),(0,0,7)\}$,
$
 \Gamma_\w
%  =\lc (t_1,t_2,t_3) \in \Gm^3 \relmid t_1^2 =t_2^3 =t_3^7 \rc
  \cong \Gm,
$
%$
% \ker \chi_\w \cong \bmu_{42},
%$
$
 \Jac_\w = \bA [x,y,z] / (2x,3y^2,7y^6),
$
and
$
 \mu = 12.
$
One can take
\begin{align}
 J_\w = \lc (i,j,k) \in \bN^3 \relmid i = 0, \ j \le 1, \ k \le 5 \rc,
\end{align}
so that a semiuniversal unfolding
$
 \wtilde \colon \bA^3 \times \Utilde \to \bA^1
$
of $\w$ is given by
\begin{align}
 \wtilde = x^2 +y^3+z^7 +\sum_{\substack{j=0,1,\\k=0,1,2,3,4,5}} u_{jk} y^j z^k.
\end{align}
Since $\phi(\Gm) = \Gamma_\w$,
the choice of $\Gamma$ is unique in this case.
The integer
\begin{align}
 w_{jk} = 42-14j-6k
\end{align}
is positive unless $(j,k) = (1,5)$,
so that $U \subset \Utilde$ is the 11-dimensional subspace
defined by $u_{15} = 0$, and
$
 \bfW \colon \bA^4 \times U \to \bA^1
$
is given by
\begin{align}
 \bfW = x^2 +y^3+z^7 + \sum_{(j,k) \ne (1,5)} u_{jk} y^j z^k v^{w_{jk}}.
\end{align}
%For this example,
%it is known that all the singularity types occurring over $U$
%occur already over $U$ \cite{brieskorn}. 
\end{example}

\section{Hochschild cohomology via matrix factorizations}
 \label{sc:HH_mf}

%One of the key ingredients in our construction of moduli of $A_\infty$-structures
%will be a Hochschild cohomology computation of singular stacks.
%
The Hochschild cohomology of a scheme $Y$
(or more generally a perfect derived stack \cite{MR2669705})
is defined as
\begin{align} \label{eq:HHY}
 \HH^*(Y) \coloneqq \Ext^*_{Y \times Y} \lb \cO_\Delta, \cO_\Delta \rb,
\end{align}
where
$
 \cO_\Delta \coloneqq \Delta_* \cO_Y
$
and
$
 \Delta \colon Y \to Y \times Y
$
is the diagonal embedding. 
The right hand side of \pref{eq:HHY} is isomorphic to the endomorphism
\begin{align} \label{eq:HHQcY}
 \HH^*(\Qcoh Y) \coloneqq \Hom^*_{\FunL(\Qcoh Y, \Qcoh Y)} \lb \id_{\Qcoh Y}, \id_{\Qcoh Y} \rb
\end{align}
of the identity in the $\infty$-category of colimit-preserving endofunctors of $\Qcoh Y$
\cite{MR2276263,MR2669705}.

When $Y$ is a smooth variety over $\bfk$
(see \cite{antieauvezzosi} for a partial extension to positive characteristics),
one can compute the Hochschild cohomology by appealing to Hochschild--Kostant--Rosenberg isomorphism
\begin{align}
 \HH^n (Y) \cong \bigoplus_{p+q=n} H^p \lb Y, \Lambda^q T_Y \rb.
\end{align}
However,
our main interest is in the case when $Y$ is a singular stack.
A generalization of the above decomposition to singular varieties
is given by Buchweitz--Flenner \cite{BuFl} which states
\begin{align}
 \HH^n(Y) \cong \bigoplus_{p+q=n} \Ext^p \lb \wedge^q \mathbb{L}_Y, \cO_Y \rb
\end{align}
where $\mathbb{L}_Y$ is the cotangent complex over $\bfk$ and
$\wedge^q$ is the derived exterior product.
However,
it is not always straightforward to compute with this,
even when $Y$ is a variety.
We will instead use another strategy which uses the function $\w$ more directly.

Let
%$\bfk$ be a field of characteristic zero,
$
 S \coloneqq \Sym V
%  \cong \bfk[x_1,\ldots,x_n] 
$
be the symmetric algebra
over the vector space
$
 V \coloneqq \vspan \lc x_0, x_1, \ldots, x_n \rc
$
of dimension $n+1$, and
$
 \bA^{n+1} = \Spec S
$
be the affine space.
Let further $\Gamma$ be a finite extension of $\Gm$
acting linearly on $V$,
$\chi \in \Gammahat \coloneqq \Hom(\Gamma, \Gm)$ be a character of $\Gamma$,
and
$
 \bfW \in H^0 \lb \cO_{[\bA^{n+1}/\Gamma]}(\chi) \rb
  \cong (S \otimes \chi)^\Gamma
$
be a non-zero element of weight $\chi$.
The quotient ring $R \coloneqq S/(\bfW)$ inherits a $\Gamma$-action.

When $\chi$ is isomorphic to the top exterior power of the dual $V^\dual$
as a $\Gamma$-module,
the bounded derived category $\coh \eY$
of coherent sheaves on the quotient stack
$
 \eY \coloneqq \ld \lb \Spec R \setminus \bszero \rb \middle/ \Gamma \rd
$
is quasi-equivalent to the idempotent-complete dg category
$
 \emf{\bA^{n+1}}{\Gamma}{\bfW}
$
of $\Gamma$-equivariant matrix factorizations;
\begin{align} \label{eq:Orlov}
 \coh \eY \cong \emf{\bA^{n+1}}{\Gamma}{\bfW}.
\end{align}
This is first proved by Orlov
\cite[Theorem 40]{MR2641200}
when $\Gamma \cong \Gm$
in the context of triangulated categories. The generalization to a finite extension of $\Gm$ is straightforward.
The quasi-equivalence of dg categories can be found in \cite{BFK,MR3084707,isik,shipman}.
Note also that by \cite[Theorem 39] {MR2641200},
$\emf{\bA^{n+1}}{\Gamma}{\bfW}$
is equivalent to the bounded stable derived category of the graded ring $R$, denoted by  $D^b_{\sing} (\gr R)$.
The equivalence \pref{eq:Orlov} implies the isomorphism
\begin{align}
 \HH^*(\eY) \cong \HH^*(\bA^{n+1}, \Gamma, \bfW),
\end{align}
where the right hand side is the Hochschild cohomology of the dg category
$
 \emf{\bA^{n+1}}{\Gamma}{\bfW},
$
which can be computed as follows:

\begin{thm}[\cite{MR2824483,MR3084707,MR3108698,BFK}] \label{HHcomp}
Let $\Gamma$ be an abelian finite extension of $\Gm$
acting linearly on $\bA^{n+1} = \Spec S$, and
$\bfW \in S$ be a non-zero element of degree
$
 \chi \in \Gammahat \coloneqq \Hom(\Gamma, \Gm).
$
Assume that the singular locus of the zero set
$Z_{(-\bfW) \boxplus \bfW}$ of the Sebastiani--Thom sum $(-\bfW) \boxplus \bfW$
is contained in the product of the zero sets $Z_\bfW \times Z_\bfW$.
Then
$
 \HH^t \lb \bA^{n+1}, \Gamma, \bfW \rb
$
is isomorphic to 
\begin{multline} \label{eq:HHmf}
 \lb \bigoplus_{\substack{\gamma \in \ker \chi, \ l \geq 0 \\ t - \dim N_\gamma = 2u }}
  H^{-2l}(d \bfW_\gamma) \otimes \chi^{\otimes (u+l)} \otimes \Lambda^{\dim N_\gamma} N_\gamma^\dual \right. \\
  \left. \oplus \bigoplus_{\substack{\gamma \in \ker \chi, \ l \geq 0 \\ t - \dim N_\gamma = 2u+1}}
  H^{-2l-1}(d \bfW_\gamma) \otimes \chi^{\otimes (u+l+1)} \otimes \Lambda^{\dim N_\gamma} N_\gamma^\vee \rb^\Gamma.
\end{multline}
\end{thm}

Here $H^i(d \bfW_\gamma)$ is the $i$-th cohomology of the Koszul complex
\begin{align} \label{eq:Koszul}
 C^*(d \bfW_\gamma) \coloneqq \lc
 \cdots
 \to \Lambda^2 V_\gamma^\dual \otimes \chi^{\otimes (-2)} \otimes S_\gamma
  \to V_\gamma^\dual \otimes \chi^\dual \otimes S_\gamma
  \to S_\gamma \rc,
\end{align}
%\begin{align}
% C(d\bfW) \coloneqq \lc \bigwedge^n \lb V^\dual(-\chi) \rb \otimes \cO_{\bA^{n+1}}
%  \to \cdots \to V^\dual(-\chi) \otimes \cO_{\bA^{n+1}}
%  \to \cO_{\bA^{n+1}} \rc,
%\end{align}
where the rightmost term $S_\gamma$ sits in cohomological degree 0, and
the differential is the contraction with
\begin{align} \label{eq:bfW_gamma}
 d \bfW_\gamma \in \lb V_\gamma \otimes \chi \otimes S_\gamma \rb^\Gamma.
\end{align}
The vector space $V_\gamma$ is the subspace of $\gamma$-invariant elements in $V$,
$S_\gamma$ is the symmetric algebra of $V_\gamma$,
$\bfW_\gamma$ is the restriction of $\bfW$ to $\Spec S_\gamma$, and
$N_\gamma$ is the complement of $V_\gamma$ in $V$
so that
%equipped with a $\Gamma$-action satisfying
$V \cong V_\gamma \oplus N_\gamma$
as a $\Gamma$-module.
The vector space on the right hand side of \pref{eq:bfW_gamma}
is the degree $\chi$ part of the space
$
\Omega_{S_\gamma} \cong V_\gamma \otimes S_\gamma
$
of K\"ahler differentials of $S_\gamma$,
and $d \bfW_\gamma$ is
the exterior derivative
of the polynomial $\bfW_\gamma$.
The zero-th cohomology of the Koszul complex \pref{eq:Koszul}
is isomorphic to the Jacobi algebra $\Jac_{\bfW_\gamma}$.
If $\bfW_\gamma$ has an isolated critical point at the origin,
then the cohomology of \pref{eq:Koszul}
is concentrated in degree 0,
so that only the summand
\begin{align} \label{eq:l=0}
 \lb \Jac_{\bfW_\gamma} \otimes \chi^{\otimes u} \otimes \Lambda^{\dim N_\gamma} N_\gamma^\dual \rb^\Gamma
\end{align}
with $l = 0$ contributes in \pref{eq:HHmf}.

The formula \pref{eq:HHmf} is an adaptation of \cite[Theorem 1.2]{BFK},
to which we refer the reader for a proof.
%The notation here
The slight difference between \cite[Theorem 1.2]{BFK} and \pref{eq:HHmf}
comes from the convention for the Koszul complex;
the latter
is convenient in that
when $V$ has an additional $\Gm$-action,
\pref{eq:HHmf} is equivariant with respect to it.

If the $\Gamma$-action on $V$ satisfies
$
 \dim \lb S \otimes \rho \rb^\Gamma
  < \infty
$
%is finite-dimensional
for any $\rho \in \Gammahat$,
then one has
\begin{align} \label{eq:HHt_finite}
 \dim \HH^t(\bA^{n+1}, \bfW, \Gamma) < \infty
\end{align}
for any $t \in \bZ$,
since the Koszul complex \pref{eq:Koszul} is bounded,
the group $\ker \chi$ is finite,
each direct summand in \eqref{eq:HHmf}
is finite-dimensional,
and there are only finitely many $u$
contributing to a fixed $t$.

We compute the Hochschild cohomologies
in several classes of examples
in Sections \ref{sc:cones}--\ref{sc:odp} below.
The results in Sections \ref{sc:fermat}--\ref{sc:EUS}
% \ref{sc:double_cover},
% \ref{sc:sylvester},
% and \ref{sc:EUS}
give examples
where the condition (4) in \pref{th:main} is satisfied,
and the results in \ref{sc:cusp} and \ref{sc:odp} are used
in the proof of \pref{th:HMS}
given in \pref{sc:HMS}.

\subsection{Cones over isolated hypersurface singularities}
\label{sc:cones}

Let $\w \in \bfk[x_1,\ldots,x_n]$ be a weighted homogeneous polynomial
of weight $(d_1,\ldots,d_n;h)$
satisfying $d_0 > 0$ and $\Gamma$ be a subgroup of $\Gamma_\w$
containing $\phi(\Gm)$ as a subgroup of finite index
as in \pref{sc:whsing}.
Assume that $\w$ has an isolated critical point at the origin and
let $\bfW$ be the image of $\w$
by the inclusion $\bfk[x_1,\ldots,x_n] \hookrightarrow \bfk[x_0,\ldots,x_n]$.
Then
$
 \eY \coloneqq \ld \lb \bfW^{-1}(0) \setminus \bszero \rb \middle/ \Gamma \rd
$
has a $\Gm$-action given by
$
 t \cdot [x_0:x_1:\cdots:x_n] = [t x_0:x_1: \cdots:x_n],
$
which induces a $\Gm$-action on
$
 \HH^*(\eY).
%  \cong \HH^*(\bA^{n+1},\bfW, \Gamma).
$
Let
$
 \HH^*(\eY)_{<0}
$
be the negative weight part of this $\Gm$-action.

Since $\bfW$ does not contain the variable $x_0$,
the Koszul complex
$C^*(d \bfW_\gamma)$
is isomorphic to the tensor product
of $C^*(d \w_\gamma)$
and the complex
$
 \lc \bfk x_0^\dual \otimes \chi^\dual \otimes \bfk[x_0] \to \bfk[x_0] \rc
$
concentrated in cohomological degree $[-1,0]$
with the zero differential
if $V_\gamma$ contains $\bfk x_0 \subset V$,
and
to $C^*(d \w_\gamma)$ otherwise.
Only direct summands coming from
$H^k(d \bfW_\gamma)$
with $k = 0, -1$
contribute to \pref{eq:HHmf}
in the former case,
and those with $k=0$ in the latter case.
Summands with $k = 0$
contribute
\begin{align} \label{eq:HH0_0}
 \lb \Jac_{\w_\gamma} 
  \otimes \bfk[x_0]
  \otimes \chi^{\otimes u}
  \otimes \Lambda^{\dim N_\gamma} N_\gamma^\dual \rb^\Gamma
\end{align}
to $\HH^{2u+\dim N_\gamma}(\eY)$,
and those with $k=-1$ contribute
\begin{align} \label{eq:HH0_1}
 \lb \bfk x_0^\dual \otimes \Jac_{\w_\gamma}
  \otimes \bfk[x_0]
  \otimes \chi^{\otimes u}
  \otimes \Lambda^{\dim N_\gamma} N_\gamma^\dual \rb^\Gamma
\end{align}
to $\HH^{2u+\dim N_\gamma + 1}(\eY)$
since
\begin{align}
 H^{-1}(d \bfW_\gamma) \cong
  \bfk x_0^\dual \otimes \chi^\dual \otimes \Jac_{\w_\gamma} \otimes \bfk[x_0].
\end{align}

\begin{cor} \label{cr:HH0}
Under the above assumptions,
one has
$
 \HH^0(\eY) \cong \bfk,
$
$
 \HH^1(\eY)_0 \not \cong 0,
$
and
$
 \HH^1(\eY)_{<0} \cong 0.
$
\end{cor}

\begin{proof}
If $u \le -1$, then \pref{eq:HH0_0} vanishes,
and if $u = 0$,
then \pref{eq:HH0_0} contribute to $\HH^0(\eY)$
only if $N_\gamma = 0$,
where it is $\bfk$.
\pref{eq:HH0_0} cannot contribute to $\HH^1(\eY)$,
since $\dim N_\gamma=1$ is impossible
for $\gamma = (t_0,t_1,\ldots,t_n) \in \Gamma$
because of the condition
$
 t_0 \cdots t_n=1.
$
One always has $u \ge -1$ in \pref{eq:HH0_1},
and one can have $u = -1$
only if $N_\gamma = \vspan \{ x_1,\ldots,x_n \}$.
Each such $\gamma$ contribute $\bfk(-1)$ to $\HH^{n-1}(\eY)$.
The summand with $u = 0$ and $\gamma = 0$ contributes
$
 \lb \bfk x_0^\dual \otimes \Jac_\w \otimes \bfk[x_0] \rb^\Gamma
$
to $\HH^1(\eY)$,
which has non-negative $\Gm$-weights.
In particular,
the element $x_0^\dual \otimes x_0$
gives a non-zero contribution to $\HH^1(\eY)_0$.
Summands with $u = 0$ and $\gamma \ne 0$ or $u \ge 1$
contribute to $\HH^{\ge 2}(\eY)$.
\end{proof}

\begin{definition} \label{df:cond_U}
We say that the pair $(\w, \Gamma)$ \emph{does not have twisted deformations}
if $\HH^2(\eY)_{<0}$ comes only from the direct summand
$\lb \Jac_\w \otimes \bfk[x_0] \otimes \chi \rb^\Gamma$
corresponding to $u = 1$ and $\gamma = 0$ in \pref{eq:HH0_0}.
\end{definition}

This condition means that direct summands with $\gamma \ne 0$,
called \emph{twisted sectors} in string theory,
do not contribute to $\HH^2(\eY)_{<0}$,
so that all deformations
corresponding to $\HH^2(\eY)_{<0}$
comes from deformations of the defining polynomial $\w$,
and one has
$
 \dim \HH^2(\eY)_{<0} = \dim U.
$

\subsection{Projective hypersurfaces}
 \label{sc:fermat}

Consider the case
\begin{align}
% \bfW(x_0,\ldots,x_n) = 
  \w(x_1,\ldots,x_n)
  = x_1^{n+1}+\cdots+x_n^{n+1} 
\end{align}
with
\begin{align} \label{eq:weight_fermat_mirror}
 (d_1,\ldots,d_n;h)=(1,\ldots,1;n+1)
\end{align}
and
\begin{align} \label{eq:Gamma_fermat_mirror}
 \Gamma
  = \lc (t_0,\ldots,t_n) \in (\Gm)^{n+1} \relmid t_1^{n+1}=\cdots=t_n^{n+1}=t_0 \cdots t_n \rc.
%  \cong \Gm \times (\bZ/4 \bZ)^2.
\end{align}
This case appears in mirror symmetry for the Calabi--Yau hypersurface of degree $n+1$ in $\bP^n$,
and gives the $D_4$-singularity $x^3 + y^3$ for $n=2$.
The group $\Gammahat$ of characters of $\Gamma$
is isomorphic to $\bZ \times (\bZ/(n+1)\bZ)^{n-1}$,
and we write the character
$
 (t_0,\ldots,t_n) \mapsto t_1^{i_1+\cdots+i_n} t_2^{-i_2} \cdots t_n^{-i_n}$
for $(i_1,\ldots,i_n) \in \bZ \times (\bZ/(n+1)\bZ)^{n-1}$ as $\rho_{i_1,\ldots,i_n}$.
One has
$
 \bfk x_0^\dual \cong \rho_{1,\ldots,1},
 \bfk x_1^\dual \cong \rho_{1,0,\ldots,0},
 \bfk x_2^\dual \cong \rho_{1,n,0,\ldots,0},
 \ldots,
 \bfk x_n^\dual \cong \rho_{1,0,\ldots,0,n},
 \chi \cong \rho_{n+1,0,\ldots,0},
$
and
$
 \ker \chi
%  = \la \frac{1}{3}(1,1,1), \frac{1}{3}(1,2,0) \ra
  \cong (\bZ/(n+1)\bZ)^n.
$
%We also have an additional $\Gm$-action
%$(x_0, x_1,\ldots, x_n) \mapsto (t x_0,\ldots,x_n)$ on $\bA^{n+1}$
%as in Example \ref{cuspidalcurve}.

When $\gamma$ is the identity element,
one has
$
 V_\gamma=V,
$
$
 N_\gamma=0,
$
$
 \bfW_\gamma = \w
%  = x^4+y^4+z^4,
$
and
\begin{align}
 \Jac_\w &\cong \bfk[x_1,\ldots,x_n]/((n+1)x_1^n,\ldots,(n+1)x_n^n).
\end{align}
The element
\begin{align}
 x_0^{(n+1)(u-i)+i} x_1^i \cdots x_n^i
  \in \lb \Jac_\w \otimes \bfk[x_0] \otimes \chi^{\otimes u} \rb^\Gamma
\end{align}
for $i=0,\ldots,\min\{u,n-1\}$
contributes $\bfk((n+1)(u-i)+i)$ to $\HH^{2u}$,
and the element
\begin{align}
 x_0^\dual \otimes x_0^{(n+1)(u-i)+i+1} x_1^i \cdots x_n^i
  \in \lb x_0^\dual \otimes \Jac_\w \otimes \bfk[x_0] \otimes \chi^{\otimes u} \rb^\Gamma
\end{align}
for $i=0,\ldots,\min\{u,n-1\}$
contributes $\bfk((n+1)(u-i)+i)$ to $\HH^{2u+1}$.

When $V_\gamma = 0$ and $N_\gamma = V$,
one has $\bfW_\gamma = 0$ and the summand
\begin{align}
 \lb \chi^{\otimes u} \otimes \Lambda^{\dim N_\gamma} N_\gamma^\dual \rb^\Gamma
  &\cong \bfk x_0^\dual \wedge \cdots \wedge x_n^\dual,
\end{align}
contributes $\bfk(-1)$ to $\HH^{2u+\dim N_\gamma} = \HH^{-2+n+1} = \HH^{n-1}$.
The number $\upsilon_1(n)$ of such $\gamma$ is $2,21,204,\ldots$
for $n=2,3,4,\ldots$ respectively.

When $V_\gamma = \bfk x_0$ and
$
 N_\gamma = \bfk x_1 \oplus \cdots \oplus \bfk x_n,
$
one has $\bfW_\gamma = 0$ and the summand
\begin{align}
 \lb \Jac_{\bfW_\gamma} \otimes \chi^{\otimes u} \otimes \Lambda^{\dim N_\gamma} N_\gamma^\dual \rb^\Gamma
  \cong \bfk x_0^{(n+1)u+n} \otimes x_1^\dual \wedge \cdots \wedge x_n^\dual
\end{align}
in $\HH^{2u+\dim N_\gamma}$ contributes $\bfk((n+1)u+n)$ to $\HH^{2u+n}$ for $u \ge 0$,
and the summand
\begin{align}
 \lb x_0^\dual \otimes \Jac_{\bfW_\gamma} \otimes \chi^{\otimes u} \otimes \Lambda^{\dim N_\gamma} N_\gamma^\dual \rb^\Gamma
  &\cong \bfk x_0^\dual \otimes x_0^{(n+1)u+n+1} \otimes x_1^\dual \wedge \cdots \wedge x_n^\dual
\end{align}
in $\HH^{2u+\dim N_\gamma+1}$ contributes $\bfk((n+1)u+n)$ to $\HH^{2u+n+1}$ for $u \ge -1$.
The number $\upsilon_2(n)$ of such $\gamma$ is $2,6,52,\ldots$
for $n=2,3,4,\ldots$ respectively.

Note that one has
\begin{align}
 \upsilon_1(n) + \upsilon_2(n) = n^n,
\end{align}
since the left hand side is equal to
is the number of elements of the set
\begin{align}
 \lc (t_1, \ldots, t_n) \in (\bGm \setminus \{ 1 \})^n
 \mid \gamma_1^{n+1} = \cdots = \gamma_n^{n+1} = 1 \rc.
\end{align}

%There are nine $\gamma$ where
%the fixed locus is a plane containing the $w$-axis.
When
$
 V_\gamma = \bfk x_0 \oplus \cdots \oplus \bfk x_i
$
and
$
 \Lambda^{\dim N_\gamma} N_\gamma^\dual = \bfk x_{i+1}^\dual \wedge \cdots \wedge x_n^\dual
$
for $0 < i < n$,
one has
$
 \bfW_\gamma = x_1^{n+1} + \cdots + x_i^{n+1}
$
and
\begin{align}
 \Jac_{\bfW_\gamma} = \bfk[x_0] \otimes
  \linspan \lc 1, x_1, \ldots, x_1^{n-1} \rc
  \otimes \cdots \otimes \linspan \lc 1, x_i, \ldots, x_i^{n-1} \rc.
\end{align}
Since the weight of
\begin{align}
 x_0^{k_0} \cdots x_i^{k_i} \otimes x_{i+1}^\dual \wedge \cdots \wedge x_n^\dual
  \in \Jac_{\bfW_\gamma} \otimes \Lambda^{\dim N_\gamma} N_\gamma^\dual
\end{align}
%then 
%if and only if there exists an element
%$
% m \in \Jac_{\bfW_\gamma} \otimes \Lambda^{\dim N_\gamma} N_\gamma^\dual
%$
%such that
%the projection of the class
%$
%% \ld \Jac_{\bfW_\gamma} \otimes \Lambda^{\dim N_\gamma} N_\gamma^\dual \rd
% [m]
%   \in \bZ \ld \Gammahat \rd
%$
%to
%$
% \bZ \ld \Gammahat / \bZ \chi \rd
%  \cong \bZ \ld (\bZ/(n+1)\bZ)^n \rd
%$
%is zero.
%To have
%\begin{align}
% \ld x_0^{k_0} \cdots x_i^{k_i} \otimes x_{i+1}^\dual \wedge \cdots \wedge x_n^\dual \rd
%  = 0 \in \bZ \ld \Gammahat / \bZ \chi \rd
%\end{align}
for $(k_0, \ldots, k_i) \in \bN \times \{0, \ldots, n-1\}^i$
can never be proportional to $\chi$,
%then one simultaneously has $k_0 \equiv 0$ and $-1$ modulo $n+1$,
%which is impossible.
%Hence
one has
\begin{align}
 \lb \Jac_{\bfW_\gamma} \otimes \chi^{\otimes u}
  \otimes \Lambda^{\dim N_\gamma} N_\gamma^\dual \rb^\Gamma
  \cong 0 
\end{align}
for any $u \in \bZ$
and similarly for
$
 \lb x_0^\dual \otimes \Jac_{\bfW_\gamma} \otimes \chi^{\otimes u}
  \otimes \Lambda^{\dim N_\gamma} N_\gamma^\dual \rb^\Gamma,
$
so that such $\gamma$ does not contribute to $\HH^*$.
In total, one has
\begin{align}
 \HH^0(\eY) &\cong \bfk, \\
 \HH^1(\eY) &\cong \bfk \oplus \bfk(-1)^{\oplus 4}, \\
 \HH^{2i+2}(\eY) &\cong \HH^{2i+3}(\eY)
  \cong \bfk(3i+1) \oplus \bfk(3i+2)^{\oplus 2} \oplus \bfk(3i+3) \quad \text{for } i \ge 0
\end{align}
for $n=2$,
\begin{align}
 \HH^0(\eY) &\cong % \bfk, \\
 \HH^1(\eY) \cong \bfk, \\
 \HH^2(\eY) &\cong \bfk(-1)^{\oplus 27} \oplus \bfk(1) \oplus \bfk(4), \\
 \HH^3(\eY) &\cong \bfk(1) \oplus \bfk(3)^{\oplus 6} \oplus \bfk(4), \\
 \HH^{2i+4}(\eY) &\cong \bfk(4i+2) \oplus \bfk(4i+3)^{\oplus 6} \oplus \bfk(4i+5) \oplus \bfk(4i+8) \quad \text{for } i \ge 0, \\
 \HH^{2i+5}(\eY) &\cong \bfk(4i+2) \oplus \bfk(4i+5) \oplus \bfk(4i+7)^{\oplus 6} \oplus \bfk(4i+8) \quad \text{for } i \ge 0
\end{align}
for $n=3$,
\begin{align}
 \HH^0(\eY) &\cong % \bfk, \\
 \HH^1(\eY) \cong \bfk, \\
 \HH^2(\eY) &\cong \bfk(1) \oplus \bfk(5), \\
 \HH^3(\eY) &\cong \bfk(-1)^{\oplus 256} \oplus \bfk(1) \oplus \bfk(5), \\
 \HH^4(\eY)
 &\cong
 \HH^5(\eY)
  \cong \bfk(2) \oplus \bfk(4)^{\oplus 52} \oplus \bfk(6) \oplus \bfk(10), \\
 \HH^{2i+6}(\eY)
 &\cong
 \HH^{2i+7}(\eY) \\
 &\cong \bfk(5i+3) \oplus \bfk(5i+7) \oplus \bfk(5i+9)^{\oplus 52} \oplus \bfk(5i+11) \oplus \bfk(5i+15) \quad \text{for } i \ge 0
\end{align}
for $n=4$,
and so on.

For $n=2$ there are twisted deformations where $\HH^2(\eY)_{-2} \cong
\bfk^{\oplus 2}$ comes from $\gamma \ne 0$, but there are no twisted deformations for all $n \ge 3$.

\subsection{Double covers of projective spaces}
 \label{sc:double_cover}

Consider the case
\begin{align}
% \bfW(x_0,\ldots,x_n) = 
  \w(x_1,\ldots,x_n)
  = x_1^2 + x_2^{2n}+\cdots+x_n^{2n} 
\end{align}
with
\begin{align} \label{eq:weight_fermat_mirror2}
 (d_1,\ldots,d_n;h)=(n,1,\ldots,1;2n)
\end{align}
and
\begin{align} \label{eq:Gamma_fermat_mirror2}
 \Gamma
  = \lc (t_0,\ldots,t_n) \in (\Gm)^{n+1} \relmid t_1^2 = t_2^{2n}=\cdots=t_n^{2n}=t_0 \cdots t_n \rc.
%  \cong \Gm \times (\bZ/4 \bZ)^2.
\end{align}
This case appears in mirror symmetry for the double cover of $\bP^{n-1}$
branched over a hypersurface of degree $2n$,
and gives the tacnode singularity $x^2+y^4$ for $n=2$.
One has $\Gammahat \cong \bZ \times \bZ/2\bZ \times (\bZ/2n\bZ)^{n-2}$
%The group $\Gammahat$ of characters of $\Gamma$
%is isomorphic to $\bZ \times (\bZ/(n+1)\bZ)^{n-1}$,
%and we write the character
%$
% (t_0,\ldots,t_n) \mapsto t_1^{i_1+\cdots+i_n} t_2^{-i_2} \cdots t_n^{-i_n}$
%for $(i_1,\ldots,i_n) \in \bZ \times (\bZ/(n+1)\bZ)^{n-1}$ as $\rho_{i_1,\ldots,i_n}$.
%One has
%$
% \bfk x_0^\dual \cong \rho_{1,\ldots,1},
% \bfk x_1^\dual \cong \rho_{1,0,\ldots,0},
% \bfk x_2^\dual \cong \rho_{1,n,0,\ldots,0},
% \ldots,
% \bfk x_n^\dual \cong \rho_{1,0,\ldots,0,n},
% \chi \cong \rho_{n+1,0,\ldots,0},
%$
and
$
 \ker \chi
%  = \la \frac{1}{3}(1,1,1), \frac{1}{3}(1,2,0) \ra
  \cong \bZ/2\bZ \times (\bZ/2n\bZ)^{n-1}.
$
%We also have an additional $\Gm$-action
%$(x_0, x_1,\ldots, x_n) \mapsto (t x_0,\ldots,x_n)$ on $\bA^{n+1}$
%as in Example \ref{cuspidalcurve}.

When $\gamma$ is the identity element,
one has
$
 V_\gamma=V,
$
$
 N_\gamma=0,
$
$
 \bfW_\gamma = \w
%  = x^4+y^4+z^4,
$
and
\begin{align}
 \Jac_\w &\cong \bfk[x_1,\ldots,x_n]/(2 x_1, 2n x_2^{2n-1},\ldots,2n x_n^{2n-1})
\end{align}
The element
\begin{align}
 x_0^{2(u-i)n+2i} x_2^{2i} \cdots x_n^{2i}
  \in \lb \Jac_\w \otimes \bfk[x_0] \otimes \chi^{\otimes u} \rb^\Gamma
\end{align}
for $i=0,\ldots, \min\{u,n-1\}$
contributes $\bfk(2(u-i)n+2i)$ to $\HH^{2u}$,
and the element
\begin{align}
 x_0^\dual \otimes x_0^{2(u-i)n+2i+1} x_2^{2i} \cdots x_n^{2i}
  \in \lb x_0^\dual \otimes \Jac_\w \otimes \bfk[x_0] \otimes \chi^{\otimes u} \rb^\Gamma
\end{align}
for $i=0,\ldots,\min\{u,n-1\}$
contributes $\bfk(2(u-i)n+2i)$ to $\HH^{2u+1}$.

When $V_\gamma = 0$ and $N_\gamma = V$,
one has $\bfW_\gamma = 0$ and the summand
\begin{align}
 \lb \chi^{\otimes u} \otimes \Lambda^{\dim N_\gamma} N_\gamma^\dual \rb^\Gamma
  &\cong \bfk x_0^\dual \wedge \cdots \wedge x_n^\dual,
\end{align}
contributes $\bfk(-1)$ to $\HH^{2u+\dim N_\gamma} = \HH^{-2+n+1} = \HH^{n-1}$.
The set of such $\gamma$ is bijective
with the set of $(i_0,i_2,\ldots,i_{n-1}) \in \{ 0,\ldots, 2n-1 \}^{n-1}$
satisfying $i_0+n+i_2+\cdots+i_n \equiv  0$ modulo $2n$.
The number $\upsilon_3(n)$ of such $\gamma$ is $2,21,300,\ldots$
for $n=2,3,4,\ldots$ respectively.

When $V_\gamma = \bfk x_0$ and
$
 N_\gamma = \bfk x_1 \oplus \cdots \oplus \bfk x_n,
$
one has $\bfW_\gamma = 0$ and the summand
\begin{align}
 \lb \Jac_{\bfW_\gamma} \otimes \chi^{\otimes u} \otimes \Lambda^{\dim N_\gamma} N_\gamma^\dual \rb^\Gamma
  \cong \bfk x_0^{2nu+2n-1} \otimes x_1^\dual \wedge \cdots \wedge x_n^\dual
\end{align}
in $\HH^{2u+\dim N_\gamma}$ contributes $\bfk(2nu+2n-1)$ to $\HH^{2u+n}$ for $u \ge 0$,
and the summand
\begin{align}
 \lb x_0^\dual \otimes \Jac_{\bfW_\gamma} \otimes \chi^{\otimes u} \otimes \Lambda^{\dim N_\gamma} N_\gamma^\dual \rb^\Gamma
  &\cong \bfk x_0^\dual \otimes x_0^{2nu+2n} \otimes x_1^\dual \wedge \cdots \wedge x_n^\dual
\end{align}
in $\HH^{2u+\dim N_\gamma+1}$ contributes $\bfk(2nu+2n-1)$ to $\HH^{2u+n+1}$ for $u \ge -1$.
The number $\upsilon_4(n)$ of such $\gamma$ is $1,4,43,\ldots$
for $n=2,3,4,\ldots$ respectively.
One has
\begin{align}
 \upsilon_3(n) + \upsilon_4(n) = (2n-1)^{n-1}
\end{align}
just as in the case of $\upsilon_1(n) + \upsilon_2(n)$.

Other $\gamma$ do not contribute,
and the result is summarized as
\begin{align}
 \HH^0(\eY) &\cong \bfk, \\
 \HH^1(\eY) &\cong \bfk \oplus \bfk(-1)^{\oplus 3}, \\
 \HH^{2i+2}(\eY) &\cong \HH^{2i+3}(\eY)
  \cong \bfk(4i+2) \oplus \bfk(4i+3) \oplus \bfk(4i+4) \quad \text{for } i \ge 0
\end{align}
for $n=2$,
\begin{align}
 \HH^0(\eY) &\cong % \bfk, \\
 \HH^1(\eY) \cong \bfk, \\
 \HH^2(\eY) &\cong \bfk(-1)^{\oplus 25} \oplus \bfk(2) \oplus \bfk(6), \\
 \HH^3(\eY) &\cong \bfk(2) \oplus \bfk(5)^{\oplus 4} \oplus \bfk(6), \\
 \HH^{2i+4}(\eY) &\cong \bfk(6i+4) \oplus \bfk(6i+5)^{\oplus 4} \oplus \bfk(6i+8) \oplus \bfk(6i+12) \quad \text{for } i \ge 0, \\
 \HH^{2i+5}(\eY) &\cong \bfk(6i+4) \oplus \bfk(6i+8) \oplus \bfk(6i+11)^{\oplus 4} \oplus \bfk(6i+12) \quad \text{for } i \ge 0
\end{align}
for $n=3$,
\begin{align}
 \HH^0(\eY) &\cong % \bfk, \\
 \HH^1(\eY) \cong \bfk, \\
 \HH^2(\eY) &\cong \bfk(2) \oplus \bfk(8), \\
 \HH^3(\eY) &\cong \bfk(-1)^{\oplus 256} \oplus \bfk(1) \oplus \bfk(5), \\
 \HH^4(\eY)
 &\cong
 \HH^5(\eY)
  \cong \bfk(4) \oplus \bfk(7)^{\oplus 43} \oplus \bfk(10) \oplus \bfk(16), \\
 \HH^{2i+6}(\eY)
 &\cong
 \HH^{2i+7}(\eY) \\
 &\cong \bfk(8i+6) \oplus \bfk(8i+12) \oplus \bfk(8i+15)^{\oplus 43} \oplus \bfk(8i+18) \oplus \bfk(8i+24) \quad \text{for } i \ge 0
\end{align}
for $n=4$,
and so on.
There are twisted deformations for $n=2$,
but there are no twisted deformations for all $n \ge 3$.

\subsection{Sylvester's sequence}
 \label{sc:sylvester}

Consider the case
$
% \bfW(x_0,\ldots,x_n) = 
  \w(x_1,\ldots,x_n)
  = x_1^{s_1}+\cdots+x_n^{s_n}
$
where
$
 (s_i)_{i=1}^\infty = (2,3,7,43,1807,\ldots)
$
is the Sylvester's sequence
defined by
$
 s_i = 1 + s_1 \cdots s_{i-1}.
%  \prod_{i=0}^{i-1} s_i.
$
This case appears in mirror symmetry
for the Calabi--Yau hypersurface
in $\bP(1,s_1,\ldots,s_n)$,
and gives the cusp singularity $x^2 + y^3$
for $n=2$.
One has
\begin{align}
 (d_0,d_1,\ldots,d_n;h)
  = (1,h/s_1,\ldots,h/s_n;s_{n+1}-1)
\end{align}
%$
% h = s_{n+1}-1,
%$
%$
% d_0 = 1,
%$
%$
% d_i = h/s_i,
%$
and
$
 \phi \colon \Gm \to \Gamma
$
is an isomorphism.
%\begin{align} \label{eq:Gamma_quartic_mirror}
% \Gm \simto \Gamma, \quad
%  t \mapsto \lb t^{d_0}, \ldots, t^{d_n} \rb.
%\end{align}

When $\gamma$ is the identity element,
one has
$
 V_\gamma=V,
$
$
 N_\gamma=0,
$
$
 \bfW_\gamma = \w
%  = x^4+y^4+z^4,
$
and
\begin{align}
 \Jac_\w &\cong \bfk[x_1,\ldots,x_n]/( s_1 x_1^{s_1-1},\ldots, s_n x_n^{s_n-1}).
\end{align}
The monomial
$
 x_0^{w_\bsj+(u-1)h} x_1^{j_1} \cdots x_n^{j_n}
$
from the summand
\begin{align}
 \lb \Jac_\w \otimes \bfk[x_0] \otimes \chi^{\otimes u} \rb^\Gamma
\end{align}
contributes $\bfk(w_\bsj+(u-1)h)$
to $\HH^{2u}$
for each $\bsj = (j_1,\ldots,j_n)$
satisfying
$
 0 \le j_i \le s_i-1
$
for $i=1,\ldots,n$
and
$
 w_\bsj \coloneqq h - d_1 j_1 - \cdots - d_n j_n
  \ge -(u-1) h.
$ 
Such $\bsj$ also contributes
$\bfk(w_\bsj+(u-1)h)$
to $\HH^{2u+1}$
just as in \pref{sc:fermat}.

Each $\gamma$
with $V_\gamma = 0$
contributes $\bfk(-1)$ to $\HH^{n-1}$.
The set of such $\gamma$ can be identified
with the set of integers from $0$ to $h-1$
prime to all $s_i$
for $i=1,\ldots,n$.
The cardinality of this set is given by
$2,12,504,\ldots$
for $n=2,3,4,\ldots$ respectively.

One never has
$V_\gamma = \bfk x_0$
in this case.
For any $\gamma$ with $V_\gamma \ne 0, V$ does not contribute to $\HH^*$
just as in \pref{sc:fermat}.

The result is summarized as
\begin{align}
 \HH^0(\eY) &\cong \bfk, \\
 \HH^1(\eY) &\cong \bfk \oplus \bfk(-1)^{\oplus 2}, \\
 \HH^{2i+2}(\eY) &\cong \HH^{2i+3}(\eY)
  \cong \bfk(6i+4) \oplus \bfk(6i+6) \quad \text{for } i \ge 0
\end{align}
for $n=2$,
\begin{align}
 \HH^0(\eY) &\cong \bfk, \\
 \HH^1(\eY) &\cong \bfk, \\
 \HH^2(\eY) &\cong \bfk(-1)^{\oplus 12} \oplus \bfk(\bsw), \\
 \HH^3(\eY) &\cong \bfk(\bsw), \\
 \HH^{2i+4}(\eY) &\cong \HH^{2i+5}(\eY)
  \cong \bfk(\bswtilde+ 42 (i+1) ) \quad \text{for } i \ge 0
\end{align}
where $\bsw = (4,10,12,16,18,22,24,28,30,36,42)$ and
$\bswtilde = (-2,\bsw)$
for $n=3$,
and so on.
There are no twisted deformations for all $n \ge 2$.

\subsection{Exceptional unimodal singularities}
 \label{sc:EUS}

Consider the weighted homogeneous polynomials
given in \pref{tb:unimodal},
which define Arnold's 14 exceptional unimodal singularities
\cite[Table 14]{MR0467795}.
We take $\Gamma = \phi(\Gm)$.
The Hilbert polynomial for the Jacobi ring
\begin{align}
 \Jac_\w \coloneqq \bfk[x_1,x_2,x_3]/(\partial_1 \w, \partial_2 \w, \partial_3 \w)
\end{align}
is given by
\begin{align} \label{eq:weight}
 \frac{(1-T^{h-d_1})(1-T^{h-d_2})(1-T^{h-d_3})}{(1-T^{d_1})(1-T^{d_2})(1-T^{d_3})}.
\end{align}
We define a non-decreasing sequence $\bswtilde = \lb w_0 \le \cdots \le w_{\mu-1} \rb$
of integers in such a way that \pref{eq:weight} is equal to
$
 \sum_{i=0}^{\mu-1} T^{h-w_i}.
$
Then one always has $w_0 = -2$,
and $\bsw \coloneqq (w_i)_{i=1}^{\mu-1}$ is as in \pref{tb:unimodal}.
The identity element $\gamma = \id_{V}$ contributes
$\bfk$ to $\HH^0$ and $\HH^1$,
$\bfk(\bsw)$ to $\HH^2$ and $\HH^3$, and
$\bfk(\bswtilde+(i+1)h)$ to $\HH^{2i+4}$ and $\HH^{2i+5}$ for $i \ge 0$.
By adding the term $x_0^h$,
one obtains a smooth Deligne--Mumford stack $\eY_1$
derived-equivalent to a K3 surface.
Since $V^\gamma$ for $\gamma \ne \id_V$ does not contain the $x_0$-axis,
contributions from $\gamma \ne \id_{V}$ is the same for $\eY$
and $\eY_1$.
On the other hand,
the rank of the total Hochschild cohomology
of $\eY_1$ is 24,
and $\gamma = \id_V$ contributes $\bfk$ to $\HH^0(\eY_1)$
via the element $1 \in \Jac_{\w}$ of degree 0,
$\bfk^{\oplus (\mu-2)}$ to $\HH^2(\eY_1)$
via elements of degrees between $1$ and $h+1$, and
$\bfk$ to $\HH^4$
via the element of degree $h+2$.
It follows that $\gamma \ne \id_V$ contribute $\bfk^{\oplus (24-\mu)}$
to $\HH^2(\eY_1)$.
Since $V_\gamma$ does not contain the $x_0$-axis,
each of these contributions contains $x_0^\dual$
from $\Lambda^{\dim N_\gamma} N_\gamma$,
and hence the $\Gm$-weight for the contribution to $\HH^2(\eY)$ is $1$.
This shows
\begin{align}
 \HH^0(\eY)
  &\cong \bfk, \\
 \HH^1(\eY)
  &\cong \bfk, \\
 \HH^2(\eY)
	&\cong \bfk(-1)^{\oplus (24-\mu)} \oplus \bfk(\bsw), \\
 \HH^3(\eY)
  &\cong \bfk(\bsw), \\
 \HH^{2i+4}(\eY)
  &\cong \HH^{2i+5}(\eY)
  \cong \bfk(\bswtilde+(i+1) h)
   \quad \text{for } i \ge 0.
\end{align}
%where 
%$\bsw$ is a vector of $\mu-1$ positive numbers
%given in Table \ref{tb:unimodal} and $\bswtilde = (-2, \bsw)$.
There are no twisted deformations in all these cases.

\begin{table}[t]
\begin{align}
\begin{array}{cccccc}
  \toprule
	\text{Name} & \text{Normal form} & (d_1,d_2,d_3;h) &\mu & \bsw \\
  \midrule
Q_{10} & x^2 z + y^3 + z^4 & (9,8,6;24) &10 &(4,6,7,10,12,15,16,18,24)\\\
Q_{11} & x^2 z + y^3 + y z^3 & (7,6,4;18) &11 &(2,4,5,6,8,10,11,12,14,18)\\
Q_{12} & x^2 z + y^3 + z^5 & (6,5,3;15) &12 &(1,3,4,4,6,7,9,9,10,12,15)\\
Z_{11} & x^2 + y^3 z + z^5 & (15,8,6;30)&11 &(4,6,10,12,14,16,18,22,24,30) \\
Z_{12} & x^2 + y^3 z + y z^4 & (11,6,4;22)& 12 & (2,4,6,8,10,10,12,14,16,18,22) \\
Z_{13} & x^2 + y^3 z + z^6 & (9,5,3;18) & 13 & (1,3,4,6,7,8,9,10,12,13,15,18) \\
S_{11} & x^2 z + x y^2 + z^4 & (6,5,4;16) & 11 & (2,3,4,6,7,8,10,11,12,16) \\
S_{12} & x^2 z + x y^2 + y z^3 & (5,4,3;13) & 12 & (1,2,3,4,5,6,7,8,9,10,13) \\
W_{12} & x^2 + y^4 + z^5 & (10,5,4;20) &12 & (2,3,6,7,8,10,11,12,15,16,20) \\
W_{13} & x^2 + y^4 + y z^4 & (8,4,3;16) &13 & (1,2,4,5,6,7,8,9,10,12,13,16)\\
E_{12} & x^2 + y^3 + z^7 & (21,14,6;42) & 12& (4,10,12,16,18,22,24,28,30,36,42) \\
E_{13} & x^2 + y^3 + y z^5 & (15,10,4;30) & 13& (2,6,8,10,12,14,16,18,20,22,26,30) \\
E_{14} & x^2 + y^3 + z^8 & (12,8,3;24) & 14& (1,4,6,7,9,10,12,13,15,16,18,21,24) \\
U_{12} & x^3 + y^3 + z^4 & (4,4,3;12) &12& (1,2,2,4,5,5,6,8,8,9,12) \\
  \bottomrule
\end{array}
\end{align}
\caption{14 exceptional unimodal singularities}
\label{tb:unimodal}
\end{table}

\subsection{Cusp singularities}
 \label{sc:cusp}

Consider the case
\begin{align} \label{eq:cusp_polynomial}
 \bfW(x_0,\ldots,x_n)
  = x_1^{n+1}+\cdots+x_n^{n+1} + x_0 \cdots x_n
\end{align}
with the same weight \pref{eq:weight_fermat_mirror}
and the group \pref{eq:Gamma_fermat_mirror}
as in \pref{sc:fermat}.

When $\gamma$ is the identity element,
one has
$
 V_\gamma=V,
$
$
 N_\gamma=0,
$
and
$
 \bfW_\gamma = \bfW.
$
The subring of $S$ consisting of semi-invariants
with respect to $\chi$
is equal to the invariant ring with respect to $\ker \chi \cong (\bmu_{n+1})^n$.
This ring is generated by $n+2$ monomials
$
 x_0^{n+1}, \ldots, x_n^{n+1}, x_0 \cdots x_n
$
with one relation
$
 x_0^{n+1} \cdots x_n^{n+1} = (x_0 \cdots x_n)^{n+1}.
$
The $n+1$ monomials
$
 x_1^{n+1}, \ldots, x_n^{n+1}, x_0 \cdots x_n
$
are zero in $\Jac_{\bfW}$,
so that
\begin{align}
 \dim \lb \Jac_{\bfW} \otimes \chi^{\otimes u} \rb^\Gamma =
\begin{cases}
 0 & u \le -1, \\
 1 & u \ge 0.
\end{cases}
\end{align}
The Grothendieck ring $\rep_\Gamma$ of finite-dimensional $\Gamma$-vector spaces
can be identified with the group ring of $\Gammahat$,
generated by $[x_0], \ldots, [x_n]$ and their inverses
with relations
$
 [x_0]^{n+1} = \cdots = [x_n]^{n+1} = [x_0] \cdots [x_n].
$
The ring $S$ is a $\Gammahat$-graded ring,
and the class
$\ld C^*(d\bfW) \rd$ 
of the Koszul complex
is an element
of a suitable completion
%$\rep_\Gamma'$
of $\rep_\Gamma$
given by
\begin{comment}
\begin{align}
 \ld C^*(d\bfW) \rd &= \frac{(1-\chi[x]^{-1})(1-\chi[y]^{-1})(1-\chi[z]^{-1})(1-\chi[w]^{-1})}
 {(1-[x])(1-[y])(1-[z])(1-[w])} \\
  &= \frac{(1-[x]^3)(1-[y]^3)(1-[z]^3)(1-[w]^3)}
 {(1-[x])(1-[y])(1-[z])(1-[w])} \\
  &= (1+[x]+[x]^2)(1+[y]+[y]^2)(1+[z]+[z]^2)(1+[w]+[w]^2).
\end{align}
One can show,
using Macaulay 2,
that the projection of the class of the zero-th cohomology
\begin{align}
 H^0(d \bfW) &\cong \bfk[x,y,z,w]/(4x^3+yzw,4y^3+xzw,4z^3+xyz,xyz)
\end{align}
of the Koszul complex
by the projection $\varpi \colon \rep' \to \bZ((t))$
sending $[x]$, $[y]$, $[z]$, $[w]$ to $t$ is given by
\begin{align}
 \varpi(H^0(d \bfW)) = 1 + 4 t + 10 t^2 + 16 t^3 + 19 t^4 + 16 t^5 + 10 t^6 + 10 t^7 + 10 t^8 + \cdots.
\end{align}
The coefficient 19 of $t^4$ comes from
the number $35 = \binom{7}{4}$ of homogeneous monomials of degree 4 in 4 variables
minus the number $16 = 4 \times 4$ of relations in the degree 4 part
of the Jacobi ideal. The Jacobi ideal is generated by 4 polynomials of degree 3,
and there are 4 monomials to be multiplied to one of the generators
to obtain an element of degree 4. There are no relations among them.
One can also show,
using Macaulay 2,
that the cohomology of the Koszul complex is concentrated in degree -1 and 0.
Since
\begin{align}
 \varpi \lb \ld C^*(d\bfW) \rd \rb
  &= (1+t+t^2)^4 \\
  &= 1+4t+10t^2+16t^3+19t^4+16t^5+10t^6+4t^7+t^8,
\end{align}
one has
\begin{align}
 \varpi([H^{-1}(d \bfW)])
  &= \varphi ([H^{-1}(d \bfW)])-\varpi([C^*(d\bfW)]) \\
  &= 6t^7+9t^8+10t^9+10t^{10}+\cdots.
\end{align}
\end{comment}
\begin{align} \label{eq:Koszul_coh}
 \ld C^*(d\bfW) \rd
%  &= \frac{(1-\chi[x]^{-1})(1-\chi[y]^{-1})(1-\chi[z]^{-1})(1-\chi[w]^{-1})}
% {(1-[x])(1-[y])(1-[z])(1-[w])} \\
%  &= \frac{(1-[x]^3)(1-[y]^3)(1-[z]^3)(1-[w]^3)}
% {(1-[x])(1-[y])(1-[z])(1-[w])} \\
  &= (1+[x_0]+\cdots+[x_0]^{n-1}) \cdots (1+[x_n]+\cdots+[x_n]^{n-1}).
\end{align}
Among $n^{n+1}$ monomials in \pref{eq:Koszul_coh},
only $[x_0]^i\cdots[x_n]^i$ for $i=0,\ldots,{n-1}$ are proportional to a power of $[\chi]$.
By projecting to the subring generated by $T \coloneqq [x_0] \cdots [x_n]$,
one obtains
\begin{align}
 \ld \lb C^*(d \bfW) \rb^\Gamma \rd
  = 1 + T + \cdots + T^{n-1}.
\end{align}
Since $\lb \partial_i \bfW \rb_{i=0}^{n-1}$ is a regular sequence in $S$,
the cohomology of the Koszul complex is concentrated in degree $-1$ and $0$.
It follows that
\begin{align}
 \ld \Jac_{\bfW} \rd - \ld H^{-1}(d \bfW) \rd
  = 1 + T + \cdots + T^{n-1},
\end{align}
so that
\begin{align}
 \dim \lb H^{-1}(d\bfW) \otimes \chi^{\otimes (u+1)} \rb^\Gamma =
\begin{cases}
 0 & u \le n-2, \\
 1 & u \ge n-1.
\end{cases}
\end{align}
Hence $\gamma=0$ contributes $\bfk$ to $\HH^{2u}$ for $u \ge 0$
and $\HH^{2u+1}$ for $u \ge n-1$.

Contributions from non-trivial $\gamma$ is the same as in \pref{sc:fermat}.
The result is summarized as
\begin{align}
 \HH^0(\eY) &\cong \bfk, \\
 \HH^1(\eY) &\cong \bfk^{\oplus 4}, \\
 \HH^{i+2}(\eY) &\cong \bfk^{\oplus 3} \quad \text{for } i \ge 0
\end{align}
for $n=2$,
\begin{align}
 \HH^0(\eY) &\cong \bfk, \\
 \HH^1(\eY) &\cong 0, \\
 \HH^2(\eY) &\cong \bfk^{\oplus 28}, \\
 \HH^3(\eY) &\cong \bfk^{\oplus 6}, \\
 \HH^{4+i}(\eY) &\cong \bfk^{\oplus 7} \quad \text{for } i \ge 0
\end{align}
for $n=3$,
\begin{align}
 \HH^0(\eY) &\cong \bfk, \\
 \HH^1(\eY) &\cong 0, \\
 \HH^2(\eY) &\cong \bfk, \\
 \HH^3(\eY) &\cong \bfk^{\oplus 256}, \\
 \HH^4(\eY) &\cong \bfk^{\oplus 53}, \\
 \HH^5(\eY) &\cong \bfk^{\oplus 52}, \\
 \HH^{6+i}(\eY) &\cong \bfk^{\oplus 53} \quad \text{for } i \ge 0
\end{align}
for $n=4$,
and so on.

Similarly,
the case
\begin{align} \label{eq:cusp_polynomial2}
 \bfW(x_0,\ldots,x_n)
  = x_1^2+x_2^{2n} + \cdots+x_n^{2n} + x_0 \cdots x_n
\end{align}
with the same weight \pref{eq:weight_fermat_mirror2}
and the group \pref{eq:Gamma_fermat_mirror2}
as in \pref{sc:double_cover} gives
\begin{align}
 \HH^0(\eY) &\cong \bfk, \\
 \HH^1(\eY) &\cong \bfk^{\oplus 3}, \\
 \HH^{i+2}(\eY) &\cong \bfk^{\oplus 2} \quad \text{for } i \ge 0
\end{align}
for $n=2$,
\begin{align}
 \HH^0(\eY) &\cong \bfk, \\
 \HH^1(\eY) &\cong 0, \\
 \HH^2(\eY) &\cong \bfk^{\oplus 26}, \\
 \HH^3(\eY) &\cong \bfk^{\oplus 4}, \\
 \HH^{4+i}(\eY) &\cong \bfk^{\oplus 5} \quad \text{for } i \ge 0
\end{align}
for $n=3$,
and so on.

\subsection{Ordinary double points}
 \label{sc:odp}

Consider the case
$
 \bfW(x_0,x_1,\ldots,x_{n+1}) = x_0^{n+1} + \cdots + x_n^{n+1} - (n+1) x_0 \cdots x_n
$
with the same weight \pref{eq:weight_fermat_mirror}
and the group \pref{eq:Gamma_fermat_mirror}
as in \pref{sc:fermat}.

When $\gamma$ is the identity element,
one has
$
 V_\gamma=V,
$
$
 N_\gamma=0,
$
and
$
 \bfW_\gamma = \bfW.
$
The generators
$
 x_0^{n+1}, \ldots, x_n^{n+1}, x_0 \cdots x_n
$
of the invariant ring $S^{\ker \chi}$
belongs to the same class in
$
 \Jac_{\bfW},
%  = \bfk[x_0,\ldots,x_n]/((n+1)x_0^n-(n+1)x_1 \cdots x_n, \ldots, (n+1) x_n^n - (n+1) x_0 \cdots x_{n-1} ), 
$
so that 
\begin{align}
 \dim \lb H^0(d\bfW) \otimes \chi^{\otimes k} \rb^\Gamma =
\begin{cases}
 0 & k \le -1, \\
 1 & k \ge 0.
\end{cases}
\end{align}
The same reasoning as in \pref{sc:cusp} shows that
$\gamma=0$ contributes $\bfk$ to $\HH^{2i}$ for $i \ge 0$
and $\HH^{2i+1}$ for $i \ge 2$.

Contributions from non-trivial $\gamma$ is the same as in \pref{sc:cusp},
except that the coordinate $x_0$ behaves exactly the same way as other coordinates.
The result is summarized as
\begin{align}
 \HH^0(\eY) &\cong \bfk, \\
 \HH^1(\eY) &\cong \bfk^{\oplus 2}, \\
 \HH^{i+2}(\eY) &\cong \bfk \quad \text{for } i \ge 0
\end{align}
for $n=2$,
\begin{align}
 \HH^0(\eY) &\cong \bfk, \\
 \HH^1(\eY) &\cong 0, \\
 \HH^2(\eY) &\cong \bfk^{\oplus 22}, \\
 \HH^3(\eY) &\cong 0, \\
 \HH^{4+i}(\eY) &\cong \bfk \quad \text{for } i \ge 0
\end{align}
for $n=3$,
\begin{align}
 \HH^0(\eY) &\cong \bfk, \\
 \HH^1(\eY) &\cong 0, \\
 \HH^2(\eY) &\cong \bfk, \\
 \HH^3(\eY) &\cong \bfk^{\oplus 204}, \\
 \HH^4(\eY) &\cong \bfk, \\
 \HH^5(\eY) &\cong 0, \\
 \HH^{6+i}(\eY) &\cong \bfk \quad \text{for } i \ge 0
\end{align}
for $n=4$,
and so on.

Similarly,
the case
\begin{align} \label{eq:cusp_polynomial3}
 \bfW(x_0,\ldots,x_n)
  = x_1^2+x_2^{2n} + \cdots+x_n^{2n} + x_0^{2n} - n \, x_0^2 x_2^2 \cdots x_n^2
\end{align}
with the same weight \pref{eq:weight_fermat_mirror2}
and the group \pref{eq:Gamma_fermat_mirror2}
as in \pref{sc:double_cover} gives
\begin{align}
 \HH^0(\eY) &\cong \bfk, \\
 \HH^1(\eY) &\cong \bfk^{\oplus 2}, \\
 \HH^{i+2}(\eY) &\cong \bfk \quad \text{for } i \ge 0
\end{align}
for $n=2$,
\begin{align}
 \HH^0(\eY) &\cong \bfk, \\
 \HH^1(\eY) &\cong 0, \\
 \HH^2(\eY) &\cong \bfk^{\oplus 22}, \\
 \HH^3(\eY) &\cong 0, \\
 \HH^{4+i}(\eY) &\cong \bfk \quad \text{for } i \ge 0
\end{align}
for $n=3$,
and so on.

\section{Generators and formality}
 \label{sc:generator}

We use the same notation as in \pref{sc:whsing}
(see \pref{eq:scrY} and \pref{eq:bfW} in particular),
and assume the existence of a tilting object $E$
of $\emf{\bA^n}{\Gamma}{\w}$.
Here,
an object $E$
of $\emf{\bA^n}{\Gamma}{\w}$
is \emph{tilting}
if the cohomologies of the endomorphism dg algebra $\dgend E$
is concentrated in cohomological degree 0
and
$\emf{\bA^n}{\Gamma}{\w}$
is generated by $E$
by shifts, cones, and direct summands.
Let $\eE$ be the pull-back of $E$ to 
$\emf{\bA_U^n}{\Gamma}{\w}$,
so that one has
$
 \End(\eE)
  \cong A^0 \otimes \bsk
$
where
$
 \bsk \coloneqq \bfk[U]
$
is the coordinate ring of $U$
and
$
 A^0 \coloneqq \End E.
$
Let further $\eS$ be the push-forward of $\eE$
to $\emf{\bA_U^{n+1}}{\Gamma}{\bfW}$,
considered as an object of $\coh \scrY$
via a variation
\begin{align}
 \emf{\bA_U^{n+1}}{\Gamma}{\bfW} \simeq \coh \scrY
\end{align}
of \cite[Theorems 16]{MR2641200},
which can be proved by a straightforward adaptation
of the original proof
(see the proof of \pref{th:generation} below).
The relation between push-forward of matrix factorizations
and Orlov's theorem
is discussed in \cite{MR3044454}.

\begin{thm} \label{th:generation}
The object $\eS$ split-generates $\perf \scrY$.
\end{thm}

\begin{proof}
For the simplicity of notation,
we assume
$\Gamma \cong \bGm$,
so that $\scrY$ is an anti-canonical hypersurface
in $\bP \coloneqq \bP_U(d_0, \ldots, d_n)$;
the extension to the general case is straightforward
(cf.~e.g., \cite[Section 3]{0604361}).
We write
$
 \bsR
  \coloneqq \bsk[x_0,\ldots,x_n]/(\bfW)
$
and
$
 \bsRbar
  \coloneqq \bsk[x_1,\ldots,x_n]/(\w)
  \cong \bsR / (x_0)
  \cong \Rbar \otimes \bsk.
$
We will work with
$D^b_\sing \lb \gr \bsRbar \rb$
and $D^b_\sing(\gr \bsR)$
instead of
$
 \emf{\bA_U^n}{\Gamma}{\w}
$
and
$
 \emf{\bA_U^{n+1}}{\Gamma}{\bfW},
$
which are equivalent by
%(the relative version of)
\cite[Theorem 39]{MR2641200}.
Here $D^b_\sing \lb \gr \bsR \rb$
is the quotient of $D^b \gr \bsR$
by the full subcategory
consisting of bounded complexes of projective modules,
denoted by $\mathbf{D}^{\mathrm{gr}}_\sing \lb \bsR \rb$
in \cite{MR2641200},
and similarly for $D^b_\sing \lb \gr \bsRbar \rb$.
%Note that
%$
% D^b_\sing(\gr \bsRbar)
%$
%is split-generated by the structure sheaf of the origin and its graded shifts
%since $\w$ has an isolated critical point at the origin (\cite{MR2824483}, \cite[Prop. A2]{MR2776613}).
Since the object
$
 \bsRbar / (x_1, \ldots, x_n)
$
of
$
 D^b_\sing \lb \gr \bsRbar \rb
$
can be described as a cone constructed out of $\eE$,
and its push-forward to
$
 D^b_\sing(\gr \bsR)
$
is $\bsR/\frakm$ where
$
 \frakm \coloneqq (x_0, \ldots, x_n),
$
it suffices to show that the images of $\bsR/\frakm(i)$ for $i \in \bZ$
under the equivalence
\begin{align}
 D^b_\sing(\gr \bsR) \cong \coh \scrY
\end{align}
split-generate $\perf \scrY$.
Since $\bsR$ is the quotient of a polynomial ring
in $n+1$ variables
by the ideal generated by a homogeneous polynomial
whose degree is the sum of degrees of the variables,
% it is ``Gorenstein of dimension $n$ and
% parameter $0$ over $\bsk$'',
% so that
one has
\begin{align}
 \hom_\bsR(\bsR/\frakm(-i), \bsR(j)) =
\begin{cases}
 \bsk[-n] & i = -j, \\
 0 & \text{otherwise}.
\end{cases}
\end{align}
%which can be shown, e.g., by a direct calculation
%using the Koszul-type resolution of $\bsR/\frakm$ over $\bsk$
%as in, e.g.~ \cite[Lemma 4.1]{FU1},
Now \cite[Lemma 15]{MR2641200} gives
semiorthogonal decompositions
\begin{align}
 D^b(\gr \bsR_{\ge 0})
  = \la \cD_0, \cS_{\ge 0} \ra
  = \la \cP_{\ge 0}, \cT_0 \ra,
\end{align}
and the proof of \cite[Theorem 16]{MR2641200}
gives
equivalences
\begin{align}
 \cD_0 \cong \coh \scrY, \quad
 \cT_0 \cong D^b_\sing(\gr \bsR),
\end{align}
and an equality
\begin{align}
\cD_0 = \cT_0,
\end{align}
where $D^b(\gr \bsR_{\ge 0})$
is the derived category of finitely-generated non-negatively graded $\bsR$-modules, and
$\cS_{\ge 0}$ and $\cP_{\ge 0}$ are its full subcategories
generated by torsion modules
(i.e., modules $M$ such that $\frakm^k M = 0$ for some $k \in \bN$
which may depend on $M$)
and free modules respectively.
%Therefore,
In order to send an object
$
 \Zbar \in D^b_\sing(\gr \bsR)
$
by the equivalence
\begin{align}
 D^b_\sing(\gr \bsR)
  \cong \cT_0
  = \cD_0
  \cong \coh \scrY,
\end{align}
we
\begin{enumerate}
 \item
find an object
$
 Z \in D^b(\gr \bsR_{\ge 0})
$ 
which goes to $\Zbar$ by the localization functor
$
 D^b(\gr \bsR_{\ge 0}) \to D^b_\sing(\gr \bsR),
$
 \item
take the semiorthogonal component $M$ of $Z$,
i.e.,
find a distinguished triangle
\begin{align}
 M \to Z \to N \to M[1]
\end{align}
such that $M \in \cT_0 = \! \,^\bot \cP_{\ge 0}$
and $N \in \cP_{\ge 0}$, and
 \item
take the image $\cM$ of $M$ by the localization functor
$
 \pi \colon D^b(\gr \bsR_{\ge 0}) \to \coh \scrY.
$
\end{enumerate}
If we start with $Z_i = (\bsR/\frakm)(-i)[-n+1]$
for $0 \le i < h$,
then
%$M$ is the cone over the morphism
\begin{align} \label{eq:Z'}
% Z' \coloneqq 
 \Cone \lb (\bsR/\frakm)(-i)[-n] \to \bsR(-i) \rb
\end{align}
belongs to
$
 \cS_{\ge i}^\bot,
$
which is equal to
$
 \!\ ^{\bot} \cP_{\ge i}
$
in the semiorthogonal decomposition
\begin{align}
 D^b(\gr \bsR_{\ge 0})
  = \la \cP_{\ge 0}, \cT_0 \ra
  = \la \cP_{\ge i}, \bsR(-i+1), \bsR(-i+2), \ldots, \bsR, \cT_0 \ra.
\end{align}
Since $(\bsR/\frakm)(-i)$ is orthogonal to
$\bsR(-i+1), \ldots, \bsR$ and
its image in $\coh \scrY$ is zero,
the image $\cM_i \in D^b \coh \scrY$
of the semiorthogonal component $M_i \in \cT_0 = \cD_0$ of $Z_i$
is isomorphic
to the image of the semiorthogonal component of $\bsR(-i)$.

Let
$
\bsT \coloneqq \bsk[x_0,\ldots,x_n]
$
be the coordinate ring
of the ambient space $\bP$.
The fact that $\deg \bfW = h$
implies the existence of an isomorphism
\[
  \hom_{\gr \bsR} \lb \bsR(-i), \bsR(-j) \rb
  \simeq
  \hom_{\gr \bsT} \lb \bsT(-i), \bsT(-j) \rb
\]
of $\bsk$-modules
for $0 \le j \le i < h$,
so that the operation of taking the semiorthogonal component of $\bsR(-i)$
is the same as that for the polynomial ring $\bsT$.
The resulting object $\cM_i$ is the restriction to $\scrY$
of the object $\cF_i$ in $\coh \bP$
obtained by mutating $\cO_\bP(-i)$
across $\cO_\bP(-i+1), \ldots, \cO_\bP(-1)$.
Since mutation preserves fullness of the collection,
the collection $(\cF_i)_{i=0}^{h-1}$
is full by \cite{MR509388}.
Now \cite[Lemma 5.4]{MR3364859} shows $\bigoplus_{i=0}^{h-1} \cM_i$
split-generates $\per \scrY$.
%which is attributed to Kontsevich.
\end{proof}

It follows from \cite[Theorem 1.1]{MR3223358} that
a choice of a section of
$
 \omega_{\bA_U^{n+1}/U}(\chi)
$
%the dualizing sheaf $\omega_{\scrY/U}$
gives an isomorphism
$
 \End(\eS) \cong A \otimes \cO_U,
$
where $A$ is the degree $n-1$ trivial extension algebra of $A^0$.
(The definition of the trivial extension algebra is recalled
in \pref{sc:introduction}; see \pref{eq:trivial extension}.)
Let $\eA$
be the minimal model
of the Yoneda dg algebra
$
 \dgend \lb \eS \rb,
$
so that one has a quasi-equivalence
\begin{align}
 \Qcoh \scrY \simeq \Mod(\eA)
\end{align}
of $\bsk$-linear pretriangulated $A_\infty$-categories.
%and an isomorphism
%\begin{align}
% \HH_U^*(\scrY) \cong \HH_U^*(\eA)
%\end{align}
%of the Hochschild cohomologies relative to $U$.

Let
$
 \eA_0 \coloneqq \eA \otimes_{\bsk} \bfk
$
be the $A_\infty$-algebra over $\bfk$
obtained by restricting $\eA$ to the origin $0 \in U$.
By using a $\Gm$-action,
we can prove the following:

\begin{thm} \label{th:formal}
$\eA_0$ is formal.
\end{thm}

\begin{proof}
We fix a $\bGm$-equivariant structure
%(i.e., a $\bZ$-grading)
on $\eS_0$
with respect to the $\bGm$-action
$(x_0,x_1,\ldots,x_n) \mapsto (\alpha x_0,x_1,\ldots,x_n)$
on $\bA^{n+1}$
in such a way that
$\End^0(\eS_0) \cong \End^0(E)$ is $\bGm$-invariant
%(i.e., concentrated in degree 0).
(this is possible since $\eS_0$ is obtained
by push-forward from an object
on the $\bGm$-invariant subspace).
Note that
$
 \omega_{\bA^{n+1}}(\chi)
$
is isomorphic to $\cO_{\bA^{n+1}}$ as a $\Gamma$-module,
but has weight 1 with respect to the $\Gm$-action.
It follows that the weight for the $\Gm$-action
on $\End^{n-1}(\eS_0) \cong \lb \End^0(E) \rb^\dual$ is one.
This shows that the cohomological degree on the $\bN$-graded algebra $\End^*(\eS_0)$
is $(n-1)$ times the $\Gm$-weight.
Since the group $\Gm$ is reductive,
the chain homotopy to transfer the dg structure on $\dgend(\eS_0)$
to the minimal model $\eA_0$ can be chosen to be $\Gm$-equivariant,
so that the resulting $A_\infty$-operations are $\Gm$-equivariant.
Since the $A_\infty$-operation $\mu^d$
has the cohomological degree $2-d$
and the cohomological degree is proportional to the $\Gm$-weight,
one must have $\mu^d = 0$ for $d \ne 2$.
\end{proof}

As a result,
we have an isomorphism
\begin{align} \label{eq:HH}
 \HH^*(A) \cong \HH^*(\eY_0)
\end{align}
of graded vector spaces.
Moreover,
the proof of \pref{th:formal} shows that
the `cohomological degree minus length' grading on the left hand side
is mapped to $(n-1)$ times the weight of the $\bGm$-action.

\section{Moduli of $A_\infty$-structures}
 \label{sc:moduli_K3}

We prove \pref{th:main} in this section.

\begin{proof}[Proof of \pref{th:main}]
We use the same notations as in \pref{sc:generator}.
\pref{cr:HH0} and \pref{eq:HH} together with \cite[Corollary 3.2.5]{Pol}
shows that the moduli functor of $A_\infty$-structures on $A$
is represented by an affine scheme $\cU_\infty(A)$.
We define the morphism \pref{eq:main}
as the classifying morphism
for the family $\eA$ of minimal $A_\infty$-structures on $A$ over $U$.
We consider the $\bGm$-action
on $\scrY$ as in \pref{eq:Gm-action_on_Y},
and equip $\eS$ with the $\bGm$-equivariant structure
such that $\End(\eS)$ is $\bGm$-equivariantly isomorphic to $A \otimes \bsk$,
where the $\bGm$-weight on $A$ is proportional to the cohomological grading
as in the proof of \pref{th:formal}.
Then the dg algebra $\dgend(\eS)$ is also $\bGm$-equivariant,
and so is the $A_\infty$-algebra $\eA$.
This means that the morphism \pref{eq:main} is $\bGm$-equivariant.

%A point $\cA$ on $\cU_\infty(A)$
%(i.e., a minimal $A_\infty$-structure $(\frakm_i)_{i=2}^\infty$
%on the graded vector space underlying $A$
%such that $\frakm_2$ is the multiplication on $A$) 
%%up to quasi-equivalences
%%whose initial component is the identity map)
%allows us to reconstruct
%not only the stable $\infty$-category $\coh Y$
%but also the stack $Y$,
%since $\cA$ is an $A_\infty$-structure on a fixed graded algebra $A$,
%which is identified with the endomorphism algebra of a particular generator $\cS$
%of $\coh Y$;
%for every point $y \in Y$,
%the skyscraper sheaf $\cO_y$ is described as a particular direct summand
%of a particular object of $\coh Y$
%obtained from $\cS$ by taking shifts and cones,
%and this determines $Y$ uniquely up to isomorphism,
%since a derived equivalence of stacks
%which sends skyscraper sheaves to skyscraper sheaves
%is a composition of a shift, an automorphism,
%and a tensor product by a line bundle.

In order to prove that $\varphi$ is an isomorphism,
first assume that $d_0 = 1$
and $G \coloneqq \Gamma/\phi(\bGm)$ is the trivial group.
Recall from \cite[Section (A.5)]{MR761312}
that an \emph{$\Rbar$-polarized scheme}
consists of a projective scheme $Y$,
an ample Weil divisor $X \subset Y$,
and an isomorphism
$R/t R \cong \Rbar$
of graded $\bfk$-algebras,
where
$
 R \coloneqq \bigoplus_{i=0}^\infty H^0(\cO_Y(i X))
$
and
$
 t \in R_1
$
is the element corresponding to 1.
It is shown in \cite[Proposition A.6]{MR761312}
that
$
% (Y, X \times U) \to
 U
$
is the fine moduli space of $\Rbar$-polarized schemes,
and the universal family is given by the coarse moduli scheme $\cY$
of $\scrY$.
We will show that
one can reconstruct
the family
$
 \cY
% (Y, X \times U) \to U
$
of $\Rbar$-polarized schemes
from the family $\eA$ of $A_\infty$-algebras.
%(which is the pulled-back by $\varphi$
%of the family of $A_\infty$-structures on $\cU_\infty(A)$
%by the fine moduli interpretation of $\cU_\infty(A)$).
Then the fine moduli interpretation of $U$
gives a morphism $\psi$
from the image of $\varphi$ to $U$
such that $\psi \circ \varphi = \id_{U}$.
This implies that the map
on tangent spaces induced by $\varphi$ is an injection,
and hence an isomorphism
since
%$\cU_\infty(A)$ is a closed subscheme of $\HH^2(A)_{<0}$ and
$\dim U = \dim \HH^2(A)_{<0} \ge \dim \cU_\infty(A)$.
Since $\varphi$ is a $\Gm$-equivariant morphism
from an affine space to an affine scheme with good $\Gm$-actions
inducing an isomorphism on tangent spaces,
it is an isomorphism of schemes.

In order to reconstruct the family
$
 \cY \to U
$
of
%$\Rbar$-polarized
schemes
from the family $\eA$ of $A_\infty$-algebras,
first note
from \pref{th:generation}
that $\cO_\scrY(i)$ for any $i \in \bZ$
can be described
as a particular object
obtained from the generator $\eS$
by taking shifts, cones, and direct summands.
This allows one to reconstruct the $\bZ$-algebra
$\lb \Hom^0(\cO_{\scrY}(i), \cO_{\scrY}(j)) \rb_{i,j \in \bZ}$
up to isomorphism
from $\eA$.
Recall that
\begin{itemize}
 \item
a \emph{$\bZ$-algebra} 
as defined in \cite{MR1230966}
is a category whose set of objects is identified with $\bZ$,
 \item
a module over a $\bZ$-algebra $C$
is a contravariant functor
from $C$ to the category of vector spaces,
 \item
the category
$\operatorname{Gr} C$
of $C$-modules is a Grothendieck category,
 \item
a $C$-module is \emph{torsion}
if it is a colimit of modules $M$
satisfying $M(i)=0$ for $i \ll 0$,
 \item
the category $\Qgr C$
is defined
as the quotient
$\operatorname{Gr} C
/
\Tor C
$
of
$\operatorname{Gr} C$
by the full subcategory
$\Tor C$
consisting of torsion modules,
\item
a $\bZ$-graded algebra
$
 B = \bigoplus_{i \in \bZ} B_i
$
gives a $\bZ$-algebra
$
 \Bv = \bigoplus_{i,j \in \bZ} \Bv_{ij}
$
by
$
 \Bv_{ij} = B_{i-j},
$
and
 \item
one has
$
\Qgr B \cong \Qgr \Bv
$
for any $\bZ$-graded algebra $B$.
\end{itemize}
See e.g.~\cite[Section 2]{MR2836401} and references therein
for more on $\bZ$-algebras and their $\Qgr$.
%(This together with
%\cite[Theorem 1.1]{MR2776790}
%implies that the $\bZ$-graded rings
%$S_u^G$ and $S_{u'}^G$
%are related by a Zhang twist,
%where $S_u^G$ is the invariant part of
%the coordinate ring
%$
% S_u \coloneqq \bfk[x_0,\ldots,x_n]/(\bfW_u)
%$
%of $Y_u$
Note that
%the $\bZ$-algebra
%$
% \lb \Hom^0(\cO_{\cY}(i), \cO_{\cY}(j)) \rb_{i,j \in \bZ}
%$
%is isomorphic to the $\bZ$-algebra
%$
% \lb \Hom^0(\cO_{\scrY}(i), \cO_{\scrY}(j)) \rb_{i,j \in \bZ}.
%$
$
 \Hom^0(\cO_{\scrY}(i), \cO_{\scrY}(j))
  \cong \Hom^0(\cO_{\cY}(i), \cO_{\cY}(j))
$
for any $i,j \in \bZ$.
The abelian category $\Qcoh \cY$
can be reconstructed from the $\bZ$-algebra
$
 \lb \Hom^0(\cO_{\cY}(i), \cO_{\cY}(j)) \rb_{i,j \in \bZ}
$
(since $\Qcoh \cY$ is the $\Qgr$ of the graded ring
$\bigoplus_{i \in \bZ} H^0(\cO_\cY(i))$,
and
$
 \lb \Hom^0(\cO_{\cY}(i), \cO_{\cY}(j)) \rb_{i,j \in \bZ}
$
is isomorphic to the $\bZ$-algebra
associated with this graded ring),
which in turn allows the reconstruction of $\cY$
by the Gabriel--Rosenberg reconstruction theorem.
%Since the locus of $\scrY$
%having a non-trivial stabilizers has codimension at least two
%(it must be strictly contained
%in the divisor $X \times U = \{ x_0 = 0 \} \subset \scrY$ at infinity),
%the $\bZ$-graded ring
%$
% \bigoplus_{n=0}^\infty H^0 \lb \cO_{\scrY}(n) \rb
%$
%can be computed on the coarse moduli scheme $\cY$.
This allows us to recover the monoidal structure on $\Qcoh \cY$,
and hence the $\bZ$-graded ring
$\bigoplus_{i \in \bZ} H^0(\cO_\cY(i))$,
from the $A_\infty$-algebra $\eA$.

Since $\coh \eX$ is a semiorthogonal summand of
$\emf{\bA^n}{\Gamma}{\w}$
by \cite[Theorem 16]{MR2641200}
and the isomorphism $\End E \cong A^0$ is given,
one has a fixed isomorphism
of the homogeneous coordinate ring
of the divisor $X \times U$ at infinity
with $\Rbar \otimes \cO_{U}$.
This concludes the reconstruction
of the family of $\Rbar$-polarized schemes
from the family of $A_\infty$-algebras
in the case when $d_0 = 1$ and $\Gamma = \phi(\bGm)$.

When $\Gamma \supsetneq \phi(\bGm)$,
then $G \coloneqq \Gamma / \phi(\bGm)$ acts on $\Rbar$,
and hence on $X$.
The affine space $U$,
defined in \pref{sc:whsing}
as the fixed locus of the natural $G$-action
on the positive part of $\Utilde$,
is the fine moduli scheme of $\Rbar$-polarized schemes
equipped with a $G$-action
extending that on $X$
by \cite[Theorem A.2]{MR761312}.
Now one can run exactly the same argument as above
to show that $\varphi$ is an isomorphism.

The generalization to the case where $d_0 \ne 1$ is completely parallel
to the generalization to the case
where
$\Gamma \supsetneq \phi(\bGm)$
given above;
if one introduces a variable $t$ of degree 1
and set $x_0 = t^{d_0}$,
then $U$ is the fixed locus of the $\bmu_{d_0}$-action
on the positive part of $\Utilde$
induced by
$
 \bmu_{d_0} \ni \zeta \colon (x_1, \ldots, x_n)
  \mapsto (\zeta^{d_1} x_1, \ldots, \zeta^{d_n} x_n).
$
\end{proof}

\section{Hochschild cohomology of the Fukaya category of the Milnor fiber}
\label{sc:HHSH}

For an object $a$ of an $A_\infty$-category $\scrA$,
the \emph{left Yoneda module}
$
\Yl_a \in \Mod \scrA^\op
$
is defined on objects by
\begin{align}
 \Yl_a(x) = \hom_{\scrA}(a, x).
\end{align}
The \emph{right Yoneda module}
$
\Yr_a \in \Mod \scrA
$
is defined similarly by
\begin{align}
 \Yr_a(x) = \hom_{\scrA}(x, a).
\end{align}
The functors
\begin{align}
  \Yl \colon \scrA^\op \to \Mod \scrA^\op, \quad
  a \mapsto \Yl_a
\end{align}
and
\begin{align}
  \Yr \colon \scrA \to \Mod \scrA, \quad
  a \mapsto \Yr_a
\end{align}
are full and faithful
by the Yoneda lemma.

An $(\scrA, \scrB)$-bimodule $X$
defines functors
\begin{align}
  (-) \otimes_\scrA X \colon \Mod \scrA \to \Mod \scrB
\end{align}
and
\begin{align}
  X \otimes_\scrB (-) \colon \Mod \scrB^\op \to \Mod \scrA^\op.
\end{align}

For a functor $F \colon \scrA \to \scrB$,
the \emph{graph bimodule} $\Gamma_F$
is the $(\scrA, \scrB)$-bimodule
defined on objects by
\begin{align}
 \Gamma_F(b, a) = \hom_\scrB(b, F(a))
\end{align}
for $a \in \scrA$ and $b \in \scrB$.
One has
\begin{align}
  \Yr_a \otimes_\scrA \Gamma_F \simeq \Yr_{F(a)}
\end{align}
and
\begin{align} \label{eq:Gamma_F otimes (-)0}
  \lb \Gamma_F \otimes_\scrB \Yl_b \rb (a) \simeq \hom_\scrB(b,F(a)).
\end{align}
Note that \pref{eq:Gamma_F otimes (-)0} implies
\begin{align} \label{eq:Gamma_F otimes (-)}
  \Gamma_F \otimes_\scrB \Yl_{F(a)} \simeq \Yl_a
\end{align}
if $F$ is full and faithful.

% The composition of $A_\infty$-functors is compatible
% with the tensor product of bimodules;
% for $F \colon \scrA \to \scrB$ and $G \colon \scrB \to \scrC$,
% one has
% \begin{align}
%  \Gamma_F \otimes_\scrB \Gamma_G
%   \coloneqq \Gamma_F \otimes T \scrB \otimes \Gamma_G
%   \simeq \Gamma_{G \circ F},
% \end{align}
% where $T \scrB$ is the bar complex of $\scrB$.

The Hochschild cohomology of an $A_\infty$-category $\scrA$
is defined as the endomorphism of the \emph{diagonal bimodule},
which in turn is defined as the graph bimodule
$
 \Delta_\scrA \coloneqq \Gamma_{\id_{\scrA}}
$
of the identity functor $\id_{\scrA}$.

% Although Hochschild cohomology is less functorial than Hochschild homology,
% it has the restriction morphism $F^* \colon \HH^*(\scrB) \to \HH^*(\scrA)$
% with respect to a full and faithful functor $F \colon \scrA \to \scrB$.

% \pref{th:Keller} below
% is proved in \cite[Theorem 4.6.b)]{Keller},
% using the triangular matrix algebra technique
% which goes back to \cite{MR1035222}.
% The main theorem of \cite{Keller} is \cite[Theorem 3.2]{Keller},
% where \pref{th:Keller} is enhanced to an isomorphism of $B_\infty$-algebras.

\begin{thm}[{\cite[Theorem 4.6.b)]{Keller}}] \label{th:Keller}
Let $X$ be an $(\scrA, \scrB)$-bimodule.
If the functors
\begin{align} \label{eq:Keller's assupmtion 1}
  \Yr (-) \otimes_\scrA X \colon \scrA \to \Mod \scrB
\end{align}
and
\begin{align} \label{eq:Keller's assumption 2}
  X \otimes_\scrB \Yl(-) \colon \scrB^\op \to \Mod \scrA^\op
\end{align}
are full and faithful,
then there exists an isomorphism
\begin{align}
  \HH^*(\scrA) \simto \HH^*(\scrB)
\end{align}
of graded vector spaces.
\end{thm}

See \cite{Keller} and references therein
for more on history, background, and enhancement of \pref{th:Keller}.

Let $\Vv$ be the Milnor fiber
of a weighted homogeneous polynomial
$\wv \colon \bC^n \to \bC$
with an isolated critical point at the origin.
The Fukaya category $\cF \lb \Vv \rb$ is a full subcategory
of the wrapped Fukaya category $\cW \lb \Vv \rb$.
Let $\lb S_i \rb_{i=1}^\mu$ be a distinguished basis of vanishing cycles,
and $\scrS$ be the full subcategory of $\cF \lb \Vv \rb$
consisting of $\lb S_i \rb_{i=1}^\mu$.
% The total morphism $A_\infty$-algebra of $\cS$ will be denoted by
% \begin{align} \label{eq:tot_vc}
%  \cA \coloneqq \bigoplus_{i,j=1}^\mu \hom_{\cF \lb \Vv \rb}(S_i,S_j).
% \end{align}
We assume
\begin{align} \label{eq:degree_condition2}
 \dv_0 \coloneqq \hv - \dv_1 - \cdots - \dv_n \neq  0.
\end{align}
%which is the analogue of assumption \pref{eq:degree_condition}
%for $\wv$.
It is shown in \cite[4.c]{seidelgraded} that
\begin{align} \label{eq:grade}
 \lb T_{S_1} \circ \cdots \circ T_{S_\mu} \rb^{\hv} = [2\dv_0],
\end{align}
where $T_S$ is the \emph{twist functor}
defined on objects
as the cone of the evaluation morphism;
\begin{align}
  x \mapsto T_S(x) \coloneqq \Cone \lb \hom(S,x) \otimes S \xto{\ev} x \rb.
\end{align}
It follows by \cite[Lemma 5.4]{MR3364859} that
$\scrS$ split-generates $\cF \lb \Vv \rb$,
so that
\begin{align}
 \cF \lb \Vv \rb \cong \perf \scrS
\end{align}
and hence
\begin{align}
 \HH^* \lb \cF \lb \Vv \rb \rb \cong \HH^* \lb \scrS \rb.
\end{align}

\begin{thm} \label{th:HH duality}
Under the assumption \pref{eq:degree_condition2},
one has an isomorphism
\begin{align}
 \HH^* \lb \cW \lb \Vv \rb \rb \cong \HH^* \lb \scrS \rb.
\end{align}
\end{thm}

\pref{th:HH duality} fails without \pref{eq:degree_condition2};
one can take
$
 \wv = x^2 + y^2
%  \colon \bC^2 \to \bC
$
%and $\Vv \cong \bCx$
as a counter-example.

Recall that a Liouville manifold is said to be \emph{non-degenerate}
if there is a finite collection of Lagrangians
such that the open-closed map
from the Hochschild homology of the full subcategory of the wrapped Fukaya category
consisting of them
to the symplectic cohomology
hits the identity element
\cite{MR2737980}.
Any Weinstein manifold is non-degenerate
\cite{1712.09126,GPS}. 

\begin{thm}[{\cite{MR3121862}}] \label{th:Ganatra}
% There is an $A_\infty$-category $\cW^2 \lb \Vv \rb$,
% containing product Lagrangians $L \times L'$
% and the diagonal $\Delta_{\Vv}$ in $\Vv^- \times \Vv$,
% and an $A_\infty$-functor
% \begin{align}
%  \bfM \colon \cW^2 \lb \Vv \rb \to \Bimod \cW \lb \Vv \rb
% \end{align}
% which is full on the full subcategory of $\cW^2 \lb \Vv \rb$
% consisting of product Lagrangians,
% and $\Delta_{\Vv}$ is sent to the diagonal bimodule $\Delta_{\cW \lb \Vv \rb}$.
% One has
% \begin{align}
%  \Hom_{\cW^2 \lb \Vv \rb}^* \lb \Delta_{\Vv}, \Delta_{\Vv} \rb
%   \cong \SH^* \lb \Vv \rb.
% \end{align}
If $\Vv$ is a non-degenerate Liouville manifold,
then one has
% $\Delta_{\Vv}$ is split-generated by product Lagrangians,
% and $\bfM$ induces
% an isomorphism
\begin{align}
%  \Hom_{\cW^2 \lb \Vv \rb}^*(\Delta_{\Vv}, \Delta_{\Vv})
  \SH^* \lb \Vv \rb
  \cong \HH^* \lb \cW \lb \Vv \rb \rb.
%  \coloneqq \Hom_{\Bimod \cW \lb \Vv \rb}(\Delta_{\cW \lb \Vv \rb}, \Delta_{\cW \lb \Vv \rb}).
\end{align}
\end{thm}

\pref{th:HH duality}
combined with \pref{th:Ganatra}
gives a proof of \cite[Conjecture 4]{seidelICM}
in our case:

\begin{cor} \label{cr:seidelconj}
Under the assumption \pref{eq:degree_condition2},
one has an isomorphism
\begin{align}
 \SH^* \lb \Vv \rb \cong \HH^* \lb \cF \lb \Vv \rb \rb.
\end{align} 
\end{cor}

To prove \pref{th:HH duality},
we apply \pref{th:Keller}
to the case where
$\scrA = \scrS$,
$\scrB = \cW \lb \Vv \rb$, and
$X$ is the graph of the inclusion functor.
To show that
the functor \pref{eq:Keller's assumption 2}
is full and faithful,
we use \pref{pr:key} below:

\begin{prop} \label{pr:key}
Let $\scrA$ be an $A_\infty$-category
whose set of objects
consists of finitely many spherical objects
$
S_1, \ldots, S_\mu,
$
and $\scrB$ be another $A_\infty$-category
equipped with a full and faithful functor
$
F \colon \scrA \to \scrB.
$
Assume the following:
\begin{enumerate}[(i)]
\item \label{as:hom with S is bounded}
For any $S \in \scrA$ and any $L \in \scrB$,
the complex $\hom(L,F(S))$ of $\bfk$-modules is perfect.
\item \label{as:spherical twists generate a shift}
There exist a positive (resp. negative) integer $m$
and an isomorphism
\begin{align} \label{eq:grade2}
T_{F(S_\mu)} \circ \cdots \circ T_{F(S_1)} \simeq [m]
\end{align}
of endofunctors on $\scrB$.
\item \label{as:hom in B is bounded below}
For any $K,L \in \scrB$,
the complex $\hom(K,L)$ is bounded below (resp. above).
\end{enumerate}
Then
the functor
\begin{align}
  \Gamma_F \otimes_\scrB \Yl(-) \colon \scrB^\op \to \Mod \scrA^\op
\end{align}
is full and faithful.
\end{prop}

\begin{proof}
Set
\begin{align}
  G \coloneqq \Gamma_F \otimes_\scrB (-) \colon \Mod \scrB^\op \to \Mod \scrA^\op.
\end{align}
We henceforth sometimes omit $\Yl$ and $F$
to avoid unnecessarily heavy notations.
Recall that the \emph{dual twist functor},
defined on objects as the shifted cone of the coevaluation morphism
\begin{align}
  x \mapsto T^\dual_S(x) \coloneqq \Cone \lb x \xto{\coev} \hom(x,S)^\dual \otimes S \rb[-1],
\end{align}
is inverse to the twist functor.
For any $K \in \scrB$,
one has distinguished triangles
\begin{equation}
\begin{tikzcd}[column sep = tiny]
 \cdots \quad T^\dual_{S_{\mu-1}} \circ T^\dual_{S_\mu}(K) \arrow[rr] && T^\dual_{S_\mu}(K) \arrow[rr] \arrow[dl] &&  K \arrow[dl] \\
& \hom \lb T^\dual_{S_\mu}(K), S_{\mu-1} \rb^\dual \otimes S_{\mu-1} \arrow[ul, dotted] 
&& \hom \lb K, S_\mu \rb^\dual \otimes S_\mu \arrow[ul, dotted]
\end{tikzcd}
\end{equation}
in $\Mod \scrB$.
The octahedral axiom
and
\pref{eq:grade2}
give a distinguished triangle
\begin{align} \label{eq:phi}
  K[-m] \xto{\phi} K \to K_1 \to K[-m+1]
\end{align}
for some $K_1 \in \perf \scrA$.
The shift
\begin{align}
  K[-2m] \xto{\phi[-m]} K[-m] \to K_1[-m] \to K[-2m+1]
\end{align}
of \pref{eq:phi}
and the octahedral axiom
give an object $K_2 \in \perf \scrA$
and distinguished triangles
\begin{align}
  K[-2m] \xto{\phi \circ \phi[-m]} K \to K_2 \to K[-2m+1]
\end{align}
and
\begin{align}
  K_1[-m] \to K_2 \xto{\psi_1} K_1 \to K_1[-m+1].
\end{align}
By iteration,
one obtains a sequence
\begin{align}
  \cdots \xto{\psi_2} K_2 \xto{\psi_1} K_1
\end{align}
and distinguished triangles
\begin{align} \label{eq:eta_i}
  K[-i m] \xto{\phi \circ \cdots \circ \phi[-(i-1)m]} K \xto{\eta_i} K_i \to K[-im+1]
\end{align}
and
\begin{align} \label{eq:psi_i}
  K_1[-im] \to K_{i+1} \xto{\psi_i} K_i \to K_1[-im+1]
\end{align}
for $i=1,2,\ldots$.
For any $S \in \scrA$
and any $j \in \bZ$,
one has isomorphisms
\begin{align}
  (-) \circ \psi_i \colon \hom^j(K_i, S) \simto \hom^j(K_{i+1},S)
\end{align}
and
\begin{align}
  (-) \circ \eta_i \colon \hom^j(K_i,S) \simto \hom^j(K,S)
\end{align}
for $i \gg 1$
because of \pref{as:hom with S is bounded},
so that
\begin{align}
  \colim_i \hom(K_i,S) \simeq \hom(K,S)
\end{align}
and hence
\begin{align}
  \colim_i \Yl(K_i) \simeq G \circ \Yl(K)
\end{align}
in $\Mod \scrA^\op$ by \pref{eq:Gamma_F otimes (-)}.
Now for any $L \in \scrB$,
one has
\begin{align}
\hom_{\Mod \scrA^\op} \lb G \circ \Yl(K), G \circ \Yl(L) \rb
&\simeq
\hom_{\Mod \scrA^\op} \lb \colim_i \Yl(K_i), G \circ \Yl(L) \rb
\\
&\simeq
\lim_i \hom_{\Mod \scrA^\op} \lb \Yl(K_i), G \circ \Yl(L) \rb
\label{eq:definition of colimit} \\
&\simeq
\lim_i \hom_{\Mod \scrB^\op} \lb \Yl \circ F(K_i), \Yl(L) \rb
\label{eq:left adjunction} \\
&\simeq
\lim_i \hom_{\scrB^\op} \lb F (K_i), L \rb
\label{eq:the Yoneda embedding is full and faithful} \\
&\simeq
\hom_{\scrB^\op} \lb K, L \rb,
\label{eq:lim}
\end{align}
where \pref{eq:lim} comes from the isomorphisms 
\begin{align}
  \psi_i \circ (-) \colon \hom^j (L,F(K_{i+1})) \simto \hom^j (L,F(K_i)) 
\end{align}
and
\begin{align}
  \eta_i \circ (-) \colon \hom^j (L,K) \simto \hom^j (L,F(K_i)) 
\end{align}
for any $j$ and sufficiently large $i$ depending on $j$,
which in turn come from
\pref{as:hom in B is bounded below}
using the distinguished triangles  \pref{eq:eta_i} and \pref{eq:psi_i}.
\end{proof}

\pref{as:hom in B is bounded below}
in \pref{pr:key}
is satisfied in our case
by \pref{lm:bounded} below.

\begin{lemma} \label{lm:bounded}
Let $\Vv$ be the Milnor fiber
of a weighted homogeneous isolated hypersurface singularity.
If $\dv_0$ is positive (resp.~negative),
then for any $K,L \in \cW \lb \Vv \rb$,
the complex $\hom(K,L)$ is bounded below (resp.~above).
\end{lemma}

\begin{proof}
By applying a small Hamiltonian isotopy to $K$ and $L$
if necessary,
one may assume that
a basis of $\hom(K,L)$ consists of intersection points
in the interior of the Liouville domain
and Hamiltonian chords in the symplectization end.
The former is finite and hence their Maslov indices are bounded.
The latter correspond bijectively
to Reeb chords between Legendrians on the contact boundary.
The contact boundary can be identified
with the link of the weighted homogeneous singularity
in such a way that
the Reeb flow on the link is the circle action
acting on the coordinates with weights $(\dv_1,\dv_2,\ldots,\dv_n)$
(see \cite[4.c]{seidelgraded}).
The Reeb flow is periodic;
the time one Reeb flow is the identity,
corresponding to going around the $S^1$ once.
We say a Reeb chord is \emph{short} (resp.~\emph{long})
if the length is less than or equal to (resp.~greater than) one.
Because the Reeb flow is periodic, every long Reeb chord is a concatenation
of a short Reeb chord and a Reeb orbit.
The set of Reeb chords form non-degenerate Morse--Bott components, and
only finitely many components consists of short chords.
Any component consisting of long chords is obtained from
a component consisting of short chords
by concatenating Reeb orbits.
In \cite[Lemma 4.15]{seidelgraded},
the index cost of going around the circle once was computed to be $2\dv_0$.
Since $\dv_0 \neq 0$ by assumption,
additivity of Maslov index implies that the complex $\hom(K,L)$ is bounded below (resp.~above)
if $\dv_0$ is positive (resp.~negative).
\end{proof}

\begin{corollary}
  Let $\Vv$ be the Milnor fiber
  of a weighted homogeneous isolated hypersurface singularity
  satisfying \pref{eq:degree_condition2}.
  Then there exists an isomorphism
  \begin{equation} \label{eq:dualityiso}
    \HH^*(\cW(\Vv)) \cong \HH^*(\cF(\Vv)).
  \end{equation}
\end{corollary}

\begin{remark} \label{rm:embedding}
\pref{pr:key} and \pref{lm:bounded} give a full and faithful functor
\(
  \cW \lb \Vv \rb^\op \to \Mod \cF \lb \Vv \rb^\op.
\)
By using right modules instead of left modules,
one can obtain
a full and faithful functor
\(
  \cW \lb \Vv \rb \to \Mod \cF \lb \Vv \rb.
\)
Note that there exists
a full and faithful functor
\(
  \coh X \to \Qcoh X \simeq \Mod \lb \perf X \rb
\)
for a perfect stack $X$.
\end{remark}

\begin{rmk}
An isomorphism
\begin{align}
  \HH^*(\coh X) \cong \HH^*(\perf X)
\end{align}
similar to \pref{eq:dualityiso}
exists for a derived stack $X$
of finite type over a perfect field
\cite[Corollary B.5.1.(i)]{1101.5834}.
\end{rmk}

\begin{rmk}
Combined with the isomorphism
\begin{align}
  \HH^* \lb \cW \lb \Vv \rb \rb \cong \HH_{*-n} \lb \cW \lb \Vv \rb \rb
\end{align}
induced by a smooth Calabi--Yau structure on $\cW(\Vv)$
and the isomorphism
\begin{align}
  \HH^* \lb \cF \lb \Vv \rb \rb \simeq \HH_{*-n} \lb \cF \lb \Vv \rb \rb^\dual
\end{align}
induced by a proper Calabi--Yau structure on $\cF(\Vv)$,
\pref{eq:dualityiso}
gives an isomorphism
\begin{align} \label{eq:dualityiso2}
  \HH_* \lb \cW \lb \Vv \rb \rb \simeq \HH_* \lb \cF \lb \Vv \rb \rb^\dual.
\end{align}
The appearance of the linear dual in \pref{eq:dualityiso2}
is consistent with the fact that
$\cF \lb \Vv \rb$ and $\cW \lb \Vv \rb$ are not Morita equivalent.
\end{rmk}

\begin{thm} \label{th:duality}
In addition to \pref{eq:degree_condition2},
assume that
the full exceptional collection $\lb S_i \rb_{i=1}^\mu$
in $\cF \lb \wv \rb$ is strong,
and that
there exists a sequence
$(L_i)_{i=1}^\mu$
of objects
% of $\cW \lb \Vv \rb$
generating $\cW \lb \Vv \rb$
such that
\begin{align} \label{eq:dualitySL}
  \dim_{\bfk} \hom^*(L_i, S_j) = \delta_{ij},
  \qquad 1 \le i,j \le \mu,
\end{align}
where $\delta_{ij}$ is the Kronecker delta.
Then there exist equivalences
\begin{align}
  \Funex \lb \cF \lb \Vv \rb, \perf \bfk \rb \simeq \cW \lb \Vv \rb,
   \label{eq:Fun(F,k)=W}\\
  \Funex \lb \cW \lb \Vv \rb, \perf \bfk \rb \simeq \cF \lb \Vv \rb.
   \label{eq:Fun(W,k)=F}
\end{align}
\end{thm}

\begin{proof}
Let
$
\scrF \coloneqq \dgend \lb \bigoplus_{i=1}^\mu S_i \rb
$
and
$
\scrW \coloneqq \dgend \lb \bigoplus_{i=1}^\mu L_i \rb
$
be the endomorphism $A_\infty$-algebras
of the generators,
which are augmented
over the semisimple ring
$
\bbk \coloneqq \bfk^{\times \mu}
$
because of \pref{eq:dualitySL}.
The assumption \pref{eq:dualitySL}
should be understood as a Koszul duality
between $\scrF$ and $\scrW$;
\begin{align}
  \scrF &\simeq \hom_\scrW(\bbk,\bbk),
  \label{eq:KoszulFW} \\
  \scrW^\op &\simeq \hom_{\scrF^\op}(\bbk,\bbk).
  \label{eq:KoszulWF}
\end{align}
The quasi-isomorphism \pref{eq:KoszulFW} is obtained
as the composition of the sequence
\begin{align}
  \scrF
  &\coloneqq \dgend_{\cF \lb \Vv \rb} \lb \bigoplus_{i=1}^\mu S_i \rb \\
  &\simeq \dgend_{\cW \lb \Vv \rb} \lb \bigoplus_{i=1}^\mu S_i \rb \\
  &\simeq \dgend_{\scrW} \lb \bbk \rb \label{eq:KoszulFW last step}
\end{align}
of quasi-isomorphisms,
where
\pref{eq:KoszulFW last step}
comes from
the fact that
the functor
\begin{align}
  \hom_{\cW \lb \Vv \rb} \lb \bigoplus_{i=1}^\mu L_i, - \rb
  \colon \cW \lb \Vv \rb \to \Mod \scrW
\end{align}
is fully faithful
since $\bigoplus_{i=1}^\mu L_i$ generates $\cW \lb \Vv \rb$
and sends $\bigoplus_{i=1}^\mu S_i$
to $\bbk$.
The quasi-isomorphism
\pref{eq:KoszulWF}
is obtained similarly
using \pref{pr:key}.

It follows from \cite[Theorem 7.2]{MR2276263} that
the $\bfk$-linear $\infty$-category of exact functors
on the left hand side of \pref{eq:Fun(F,k)=W}
is equivalent
to the full subcategory of $\Mod \scrF$
consisting of $\scrF$-modules
which are perfect as $\bfk$-modules.
Since the cohomology algebra of $\scrF$ is
the trivial extension algebra
of the total morphism algebra
of a strong exceptional collection,
the augmentation ideal of
$
\scrF \simeq \dgend_{\scrW} \lb \bbk \rb
$
is nilpotent.
It follows that
the full subcategory of $\Mod \scrF$
consisting of $\scrF$-modules
which are perfect as $\bfk$-modules
is generated by $\bbk$,
and hence is equivalent to $\perf \scrW \simeq \cW \lb \Vv \rb$,
which is generated by $\bigoplus_{i=1}^\mu L_i$.

For any
$
K \in \Funex \lb \cW \lb \Vv \rb, \perf \bfk \rb
$
(which can be identified with a $\scrW$-module
which is perfect as a $\bfk$-module),
the smoothness of $\cW \lb \Vv \rb$
shown in \cite[Theorem 1.2]{MR3121862}
implies that
the cohomology of
$
\dgend(K)
$
is bounded.
If follows that the morphism
$
\phi \circ \cdots \circ \phi[-(i-1)m] \colon K[-im] \to K
$
in \pref{eq:eta_i}
is zero for $i \gg 1$,
so that $K$ is a direct summand of an object $K_i$
of $\cF \lb \Vv \rb$,
and hence $K$ itself is an object $\cF \lb \Vv \rb$
by our convention that all Fukaya categories are idempotent-completed.
This shows \pref{eq:Fun(W,k)=F},
and \pref{th:duality} is proved.
\end{proof}

\begin{rmk} Koszul duality between endomorphism algebras of
generators of compact and wrapped Fukaya category have been observed in
\cite{EtLe,EkLe, Li,1907.09257, 2010.10114, 2104.10050}. \end{rmk}

\section{Symplectic cohomology of the Milnor fiber}
 \label{sc:SH}
 
In this section,
we recall a spectral sequence converging to $\SH^* \lb \Vv \rb$
associated to a normal crossings compactification of $\Vv$
due to \cite{mcleantalk, 1811.03609}.
It is based on a standard model of the Reeb flow
in a neighborhood of compactification divisor and can be perceived
as an elaborate version of the standard Morse-Bott model
discussed in \cite{seidelbiased} when the compactification divisor is smooth.
See also \cite{ganpom} and \cite{diolis} for related results.

Let $\Ytilde$ be a smooth projective variety
containing an affine variety with $c_1 \lb \Vv \rb=0$
in such a way that
$
 \Dv \coloneqq \Ytilde \setminus \Vv
$
is a normal crossing divisor;
\begin{align}
 \Dv = \bigcup_{i \in I} \Dv_i.
\end{align} 
For $J \subset I$, we set $\Dv_{J} = \bigcap_{i \in J} \Dv_i$, and also set $\Dv_\emptyset = \Vv$. 

Choose a sequence $\kappa = (\kappa_i)_{i \in I}$ of positive integers
such that the divisor
$
 \sum_{i\in I} \kappa_i \Dv_i
$
on $\Ytilde$ is ample.
Let $(c_i)_{i \in I}$ be another sequence of integers
such that
$
 \sum_{i \in I} c_i \Dv_i
$
is linearly equivalent to the canonical divisor of $\Ytilde$.
When $\Ytilde$ is a Calabi--Yau manifold,
one can set $c_i=0$ for all $i \in I$.

Still following \cite{mcleantalk,1811.03609},
for each $J \subset I$,
we let $N\Dv_{J}$ be a small tubular neighborhood of $\Dv_J$
such that $N\Dv_{J} \cap \Dv_{J'}$ is a tubular neighborhood of $\Dv_{J \cup J'}$ for all $J' \subset I$.
Moreover, we require that the boundary $\partial N\Dv_{J}$ intersects $\Dv_{J'}$ for all $J' \subset I$.
Next, we let
\begin{align}
 \overset{\circ}{N}\Dv_{J} = N\Dv_{J} \setminus \cup_{i \in I} \Dv_i
\end{align}
be the punctured tubular neighborhood.

\begin{thm}[{\cite{mcleantalk,1811.03609} (see also \cite[Remark 3.17]{ganpom})}]
 \label{th:mcleanss} There is a cohomological spectral sequence converging to $\SH^*(\Vv)$ with $E_1$-page given by
\begin{align}
E_{1}^{p,q} = \bigoplus_{ \lc (k_i) \in \bZ_{\ge 0}^I \relmid \sum k_i \kappa_i = - p \rc }
 H^{p+q-2 \sum_{i} k_i(c_i+1)} \lb \overset{\circ}{N}\Dv_{{J}_{(k_i)}} \rb
\end{align}
\end{thm}
where $J_{(k_i)} = \lc i \in I \relmid k_i \neq 0 \rc$.

Since $\kappa_i$ is positive for all $i$,
for each $p$,
we have $E_1^{p,q} \neq 0$ only for finitely many $q$,
and is a finite sum of finite-dimensional vector spaces.
Moreover, if $c_i>-1$ for all $i$, then the spectral sequence is regular.

%Recall that in a cohomological spectral sequence $r^{th}$-differential is a map
%\begin{align}
%d_r : E_r^{p,q} \to E_r^{p+r, q+1-r}.
%\end{align}

We will apply this spectral sequence
to deduce $\SH^1 \lb \Vv \rb =0$,
where $\Vv$ is the Milnor fiber of a weighted homogeneous singularity.
%satisfying \eqref{eq:degree_condition}. 

%\begin{prop} \label{ampleexists} Let $\w : \C^3 \to \C$ is a weighted homogeneous polynomial satisfying assumption \ref{eq:degree_condition}. There exists a compactification of the Milnor fiber $\Vv$ of $\w$ to a smooth K3 surface $\Ytilde$ such that the compactification divisor $\Dv = \sum_{i\in I} \Dv_i$ is a normal crossings divisor with $\Dv_i$ irreducible. Moreover, there exists an ample line bundle $\mathcal{L}$ on $\Ytilde$ associated to the divisor
%\begin{align} \sum_{i\in I} \kappa_i \Dv_i \text{\ with \ } \kappa_i >0 \end{align} 
%\end{prop}
%
%\begin{proof}
%The compactification with a normal crossing divisor can be obtained by resolving the singularities of $(Y,X)$ by blowing-up. The ample line bundle $\mathcal{L}$ can be constructed by starting with a divisor with positive coefficients and adding a sufficiently large multiple of a very ample divisor.
%\end{proof}

\begin{cor} \label{cr:shifts}
Let $\Vv$ be the Milnor fiber of a weighted homogeneous polynomial
with an isolated critical point at the origin, $d_0>0$ 
and $\dim \Vv \ge 2$, admitting a compactification
to a Calabi--Yau manifold by adding a normal crossing divisor.
One has
$\SH^i \lb \Vv \rb=0$ for $i < 0$,
$\SH^0 \lb \Vv \rb=\bC$,
and
$\SH^1 \lb \Vv \rb=0$.
\end{cor}

\begin{proof}
Since $\Vv$ is simply connected,
we do not get any contribution from $H^1(\Vv)=0$.
The vanishing of $c_i$ and the positivity of $\kappa_i$ imply that
the orbits coming from the normal crossing divisor contribute to $\SH^i \lb \Vv \rb$ for $i \geq 2$. 
\end{proof}

Now we can prove a generalization
of the non-formality result in \cite{LP1},
which corresponds to the case $\w =x^2+y^3$. 

\begin{thm} \label{th:non-formality}
Under the same assumption as \pref{cr:shifts},
$\cA$ is not formal.
\end{thm}

\begin{proof}
By \pref{cr:HH0},
we have $\HH^1(A) \ne 0$.
On the other hand, we know by \pref{cr:seidelconj} that 
$\HH^1(\cA, \cA)$ is isomorphic to $\SH^1(\Vv)$,
which is zero by \pref{cr:shifts}.
Hence we conclude that $\cA$ is not formal.
\end{proof}

A non-zero element of $\HH^1(A)$
is given by the Euler derivation
defined by
\begin{align}
 \eu(x) = \deg(x) x.
\end{align}
Recall that for any $A_\infty$-algebra $\scrA$ with $H^*(\scrA) = A$,
there exists a length spectral sequence
converging to $\HH^*(\scrA)$
with $E_2$-page given by $E_2^{p,q} = \HH^{p+q}(A)_q$.
It is shown in \cite[Equation 3.14]{MR3364859}
that the class of the Euler vector field is killed by the differential on $E_2$
if $\scrA$ is non-formal.

In dimension 2,
\pref{th:non-formality} can also be proved as follows:
If $\cA$ is formal,
then $\HH^*(\cA) \cong \HH^*(Y_0)$ has a dilation
since the BV operator on $\HH^*(Y_0)$
induced by the holomorphic volume form
sends $\eu/2 \in \HH^1$ to $1 \in \HH^0$.
On the other hand,
$\SH^* \lb \Vv \rb$ cannot have a dilation
due to the existence of an exact Lagrangian torus in $\Vv$
proved in \cite{MR3432159}.
Note that this argument uses that the BV operator on $\SH^* \lb \Vv \rb$
agrees with BV operator on $\HH^*(\cA)$,
which holds since any two BV operators differ by an invertible element in $\HH^0$,
which is of rank 1 in our case.

We give computations of the spectral sequence
in a few examples.

\subsection{The affine quartic surface}
 \label{sc:quartic_SH}

Let $\Vv = \w^{-1}(-1)$ be the Milnor fiber
of the quartic polynomial
$
 \w(x,y,z) = x^4 + y^4 + z^4,
$
which can be compactified to a quartic K3 surface $\Ytilde$ in $\P^3$
by adding a smooth curve $\Dv$ of genus 3.
We can take $\kappa = 1$ and $c=0$,
so that the $E_1$-page of the resulting spectral sequence is given in \pref{tb:spectral15}.
\begin{table}[H]
\centering
\begin{tikzpicture}
\matrix (mymatrix) [matrix of nodes, nodes in empty cells, text height=1.5ex, text width=3.0ex, align=center]
{
\tikz\node[overlay] at (-2.6ex,-33.2ex){\footnotesize p};\tikz\node[overlay] at (24ex,2.8ex){\footnotesize q}; 
      &   $\C^6$     &  0   &  0   &   0 & \vdots\\
      &   $\C$     &  $\C$  &  0   &   0 & 9 \\
      &   0     & $\C^6$ &  0   &   0 & 8\\
     &   0     &  $\C^6$   &  0   &   0 & 7\\
      &   0     &  $\C$   &  $\C$   &   0 & 6\\
      &   0     &  0   &  $ \C^6$   &   0 & 5\\
      &   0     &  0   &  $ \C^6$   &   0 & 4\\
      &   0     &  0   &  $\C$   &   0 & 3\\
      &   0     &  0   &  0   &$\C^{27}$& 2\\
      &   0     &  0   &  0   &   0 & 1\\
      &   0     &  0   &  0   &   $\C$ & 0\\
     & \ldots  & $-2$ & $-1$ & $0$ & \\
};
\draw (mymatrix-11-1.south west) ++ (-0.2cm,0) -- (mymatrix-11-6.south east);
\draw (mymatrix-1-6.north west) -- (mymatrix-12-6.south west);

\end{tikzpicture}
\caption{$E_1$ page of the spectral sequence for $x^4+y^4+z^4$.}
\label{tb:spectral15}
\end{table}
We immediately conclude that
$
 \SH^0(\Vv) = \bC,
$
$
 \SH^1(\Vv) = 0,
$
$
 \SH^2(\Vv) = \bC^{28},
$
$
 \SH^3(\Vv) = \bC^6,
$
and
$
 \SH^i(\Vv) = \bC^6
$
or $\bC^7$ for $i > 3$.
%\begin{align} \label{eq:quarticsh}
%\SH^0(\Vv) = \C, \ \SH^1(\Vv) =0, \ \SH^2(\Vv) = \C^{28}, \ \SH^3(\Vv)=\C^6, \ \SH^i(\Vv)= \C^6 \text{\ or\ } \C^7 \text{ for } i > 3. 
%\end{align}

More generally,
let $\Vv= \w^{-1}(-1)$ for the polynomial
\begin{align}
  \w (x_1,\ldots, x_n) = x_1^{n+1} + \cdots + x_n^{n+1}
\end{align}
which compactifies to a Calabi-Yau hypersurface of degree $n+1$ in $\bP^n$ by looking at the zero set of $\mathbf{W}(x_0,x_1,\ldots, x_n) = x_0^{n+1} +\ldots + x_n^{n+1}$ in $\bP^n$. The smooth divisor at infinity $\Dv$ is defined by $\w=0$ in $\bP^{n-1} = \{x_0=0\}$. By standard arguments (cf. \cite{dolgachevweighted}) we can compute the cohomology of $\Dv$ as follows:
\begin{align}
H^*(\Dv) = \begin{cases}
\C & *=2k, \text{\ for } 0 \leq 2k < (n-2) \\
\C^{\left\lfloor \frac{n^{n}}{n+1}\right\rfloor+(-1)^{n}+1}  & *= n-2,  \\
\C & *=2k \text{\ for } (n-2) < 2k \leq 2(n-2).  
\end{cases} 
\end{align}
In view of the Lefschetz hyperplane theorem, the only non-trivial part of the computation is the Betti number $b_{n-2}(\Dv)$. This can be computed via the formula $b_{n-2}(\Dv) = (-1)^n (\chi(\Dv) - 2\left\lfloor \frac{n-1}{2} \right\rfloor) $ and the Euler characteristic can in turn be computed via Chern classes to be $\frac{1}{n+1} ( (-1)^n n^n + n(n+1) -1 )$.

The circle bundle $N\Dv$ has Euler class $(n+1)$ times the hyperplane class.  This implies via the Leray--Serre spectral sequence that the cohomology of $N\Dv$ is given by
\begin{align}
H^*(N\Dv) = \begin{cases}
\C & *=0 \\
\C^{\left\lfloor \frac{n^{n}}{n+1}\right\rfloor+\frac{(-1)^{n}+1}{2}}  & *= n-2, n-1 \\
\C & *=2n.  
\end{cases} 
\end{align}
As for the Milnor fiber, the homotopy type is given as a wedge of $\mu$ spheres where Milnor number $\mu= n^n$ for $\w$. Thus, we have 
\begin{align} H^*(\Vv) = \begin{cases}
\C & *=0, \\
\C^{n^n} & *=n-1.
\end{cases}\end{align}

In constructing the spectral sequence we can,  as before, take $\kappa =1$ and $c=0$.
From the spectral sequence, we can immediately conclude that for $n>3$, we have $SH^0(\Vv) =\C, SH^1(\Vv)= 0, SH^2(\Vv) = \C$ and $SH^{n-1}(\Vv) = \C^{n^n}$ or $\C^{n^n-1}$.

\subsection{The double cover of the plane branched along a sextic}
 \label{sc:degree_2_SH}

Let $\Vv = \w^{-1}(-1)$ be the Milnor fiber of the polynomial
$
\w(x,y,z) = x^2 + y^6 + z^6,
$
which can be compactified to the double cover $\Ytilde$ of $\bP^2$
branched along a smooth sextic curve
by adding a smooth curve $\Dv$ of genus 2.
We can take $\kappa = 1$ and $c=0$,
so that the $E_1$-page of the resulting spectral sequence is given in \pref{tb:spectral16}.
\begin{table}[H]
\centering
\begin{tikzpicture}
\matrix (mymatrix) [matrix of nodes, nodes in empty cells, text height=1.5ex, text width=3.0ex, align=center]
{
\tikz\node[overlay] at (-2.6ex,-33.2ex){\footnotesize p};\tikz\node[overlay] at (24ex,2.8ex){\footnotesize q}; 
         &   $\C^4$     &  0   &  0   &   0 & \vdots\\
      &   $\C$     &  $\C$  &  0   &   0 & 9 \\
      &   0     & $\C^4$ &  0   &   0 & 8\\
     &   0     &  $\C^4$   &  0   &   0 & 7\\
      &   0     &  $\C$   &  $\C$   &   0 & 6\\
      &   0     &  0   &  $ \C^4$   &   0 & 5\\
      &   0     &  0   &  $ \C^4$   &   0 & 4\\
      &   0     &  0   &  $\C$   &   0 & 3\\
      &   0     &  0   &  0   &$\C^{25}$& 2\\
      &   0     &  0   &  0   &   0 & 1\\
      &   0     &  0   &  0   &   $\C$ & 0\\
     & \ldots  & $-2$ & $-1$ & $0$ & \\
};
\draw (mymatrix-11-1.south west) ++ (-0.2cm,0) -- (mymatrix-11-6.south east);
\draw (mymatrix-1-6.north west) -- (mymatrix-12-6.south west);

\end{tikzpicture}
\caption{$E_1$ page of the spectral sequence for $x^2+y^6+z^6$.}
\label{tb:spectral16}
\end{table}
We immediately conclude that
$
 \SH^0(\Vv) = \bC,
$
$
 \SH^1(\Vv) = 0,
$
$
 \SH^2(\Vv) = \bC^{26},
$
$
 \SH^3(\Vv) = \bC^4,
$
and
$
 \SH^i(\Vv) = \bC^4
$
or $\bC^5$ for $i > 3$.
%\begin{align}
%\SH^0(\Vv) = \C, \ \SH^1(\Vv) =0, \ \SH^2(\Vv) = \C^{26}, \ \SH^3(\Vv)=\C^4, \ \SH^i(\Vv)= \C^4 \text{\ or\ } \C^5\text{ for } i > 3. 
%\end{align}

More generally, let $\Vv= \w^{-1}(-1)$ for the polynomial
\begin{align}
  \w (x_1,\ldots, x_n) = x_1^{2} + x_2^{2n} + \cdots + x_n^{2n}
\end{align}
which compactifies to a Calabi-Yau hypersurface in $\bP(n,1,1,\ldots, 1)$ by looking at the zero set of $\mathbf{W}(x_0,x_1,\ldots, x_n) = x_0^{2n} + x_1^{2} + x_2^{2n} + \ldots + x_n^{2n}$ in $\bP(1,n,1,1,\ldots,1)$. The smooth divisor at infinity $\Dv$ is defined by $\w=0$ in $\bP(n,1,\ldots,1) = \{x_0=0\}$. By standard arguments (cf. \cite{dolgachevweighted}) we can compute the cohomology of $\Dv$ as follows:
\begin{align}
H^*(\Dv) = \begin{cases}
\C & *=2k, \text{\ for } 0 \leq 2k < (n-2) \\
\C^{\lfloor \frac{(2n-1)^{n-1}}{2n}\rfloor+(-1)^{n}+1}  & *= n-2,  \\
\C & *=2k \text{\ for } (n-2) < 2k \leq 2(n-2).
\end{cases} 
\end{align}
In view of the Lefschetz hyperplane theorem, the only non-trivial part of the computation is the Betti number $b_{n-2}(\Dv)$. This can be computed via the formula $b_{n-2}(\Dv) = (-1)^n (\chi(\Dv) - 2\lfloor \frac{n-1}{2} \rfloor) $ and the Euler characteristic can in turn be computed via Chern classes to be $\frac{1}{2n}((-1)^n(2n-1)^{n-1} + 2n(n-1)+1)$.

The circle bundle $N\Dv$ has Euler class $2n$ times the hyperplane class.  This implies via the Leray-Serre spectral sequence that the cohomology of $N\Dv$ is given by
\begin{align}
H^*(N\Dv) = \begin{cases}
\C & *=0 \\
\C^{\lfloor \frac{(2n-1)^{n-1}}{2n}\rfloor+\frac{(-1)^{n}+1}{2}}  & *= n-2, n-1 \\
\C & *=2n.  
\end{cases} 
\end{align}

As for the Milnor fiber, the homotopy type is given as a wedge of $\mu$ spheres where Milnor number $\mu= (2n-1)^{n-1}$ for $\w$. Thus, we have 
\begin{align} H^*(\Vv) = \begin{cases}
\C & *=0, \\
\C^{(2n-1)^{n-1}} & *=n-1.
\end{cases}\end{align}

In constructing the spectral sequence we can,  as before, take $\kappa =1$ and $c=0$.
From the spectral sequence, we can immediately conclude that for $n>3$, we have $SH^0(\Vv) =\C, SH^1(\Vv)= 0, SH^2(\Vv) = \C$ and $SH^{n-1}(\Vv) = \C^{(2n-1)^{n-1}}$ or $\C^{(2n-1)^{n-1}-1}$.

\section{Homological mirror symmetry for Milnor fibers}
 \label{sc:HMS}

We prove \pref{th:HMS} in this section.

\begin{proof}[Proof of \pref{th:HMS}]
Let
$
 \Vv \coloneqq \lc (x_1,\ldots,x_n) \in \bC^n \relmid x_1^{n+1}+\cdots+x_{n}^{n+1}=1 \rc
$
be the Milnor fiber of $\w=x_1^{n+1}+\cdots+x_{n}^{n+1}$.
A distinguished basis
$\lb S_i \rb_{i=1}^{n^n}$
of vanishing cycles
generates the compact Fukaya category of $\Vv$,
and the cohomology $A$ of the total morphism $A_\infty$-algebra
$
 \cA \coloneqq \bigoplus_{i,j=1}^{n^n} \hom \lb S_i, S_j \rb
$
is the degree $n-1$ trivial extension algebra of the tensor product 
$\fA_n^{\otimes n}$
of the Dynkin quiver $\fA_n$
of type $A_n$.
%\begin{figure}[htb!]
%\centering
%
%\begin{tikzpicture}
%        \tikzset{vertex/.style = {style=circle,draw, fill,  minimum size = 2pt,inner sep=1pt}}
%\tikzset{edge/.style = {->,>=stealth',shorten >=8pt, shorten <=8pt  }}
%
%
%% vertices
%\node[vertex] (a) at  (0,0) {};
%\node[vertex] (a1) at (1.5,0) {};
%\node[vertex] (a2) at (3,0) {};
%
%% edges
%
%\draw[edge] (a)  to (a1);
%\draw[edge] (a1) to (a2);
%
%\node [shape=circle,minimum size=2pt, inner sep=1pt] (a3) at (4.5,0) {};
%
%\node at (0.75,0.2) {\tiny $a_{21}$};
%\node at (2.25,0.2) {\tiny $a_{32}$};
%
%\node at (7,0) {$a_{32}a_{21}=0$, \ \ $|a_{32}|=|a_{21}|=1$ };
%
%\end{tikzpicture}
%
%\caption{$A_3$ quiver in linear orientation}
%\label{a3}
%\end{figure}
The $A_\infty$-algebra $\cA$
is not formal by \pref{th:non-formality},
and $\HH^*(\cF(\Vv))$ is isomorphic to $\SH^*(\Vv)$
computed in \pref{sc:quartic_SH}.

The graded algebra $A$ also appears
as the cohomology of the Yoneda dg algebra $\eA_u$
of a generator $\eS_u$ of $\perf \eY_u$
where
$
 \eY_u
$
for
$
 u
%  = (u_1,u_4)
  \in U \coloneqq \Spec \bC[u_1,u_{n+1}]
$
is the quotient stack
$
 \ld
  \lb
   \Spec
%    \lb \bC[w,x,y,z]/(x^4+y^4+z^4+ u_1 xyzw + u_4 w^4) \rb
    S_u
    \setminus \bszero
  \rb \middle/ \Gamma \rd
$
for
$
 S_u \coloneqq \bC[x_0,\ldots,x_n]/(x_1^{n+1}+\cdots+x_{n}^{n+1}+ u_1 x_0 \cdots x_n + u_{n+1} x_0^{n+1})
$
and
$
 \Gamma \coloneqq \lc (t_1,\ldots,t_n) \in \Gm^n \relmid t_1^{n+1}=\cdots=t_n^{n+1} \rc.
$
The moduli space $\cU_\infty(A)$ of minimal $A_\infty$-structures on $A$
is identified with $U$.

In order to identify $u \in U$
satisfying $\cA \simeq \eA_u$,
we compare $\HH^*(\eA_u)$ and $\HH^*(\cA) \cong \SH^*(\Vv)$
as graded vector spaces.
Since $\SH^*(\Vv)$ is infinite-dimensional over $\bfk$,
the mirror $\eY_u$ must be singular.
Up to the action of $\Gm$ on $U$,
there are precisely two non-zero $u \in U$ such that $\eY_u$ is singular,
i.e., $(u_1,u_{n+1}) = (1,0)$ and $(-n-1,1)$.
The Hochschild cohomologies of these singular stacks are computed in Sections \ref{sc:cusp}
and \ref{sc:odp}.
Comparing this with $\SH^*(\Vv)$ computed in \pref{sc:quartic_SH},
we conclude that the mirror of $\Vv$ is the stack associated with $(u_1,u_{n+1}) = (1,0)$.

The equivalence \pref{eq:HMS_Fermat2}
follows from \pref{eq:HMS_Fermat1},
\pref{eq:Fun(F,k)=W},
and
\begin{align}
  \Funex \lb \perf [Z/K], \perf \bfk \rb
  \simeq
  \coh [Z / K]
\end{align}
in \cite[Remark 1.1.6.(ii)]{MR3730514}.
The assumption \pref{eq:dualitySL} for Brieskorn--Pham singularities
is proved in \cite[Section 2.1]{2104.10050}.

The proof for
$
 \Vv \coloneqq \lc (x_1,\ldots,x_n) \in \bC^n \relmid x_1^2+x_2^{2n}+\cdots+x_n^{2n}=1 \rc
$
goes along the same lines.
The cohomology $A$ of the total morphism $A_\infty$-algebra
of a distinguished basis of vanishing cycles
is given by the degree $n-1$ trivial extension algebra
of $\fA_{2n-1}^{\otimes (n-1)}$.
The moduli space $\cU_\infty(A)$ of minimal $A_\infty$-structures is identified with
$
U \coloneqq \Spec \bC[u_2,u_{2n}]
$
parametrizing
$
\eY_u \coloneqq
 \ld
  \lb
   \Spec
    S_u
    \setminus \bszero
  \rb \middle/ \Gamma \rd
$
for
$
 S_u \coloneqq \bC[x_0,\ldots,x_n]/
 (x_1^2+x_2^{2n}+\cdots+x_n^{2n} + u_{2n} x_0^{2n} + u_{2} x_0^2 x_2^2 \cdots x_n^2)
$
and
$
 \Gamma \coloneqq \lc (t_1,\ldots,t_n) \in \Gm^n \relmid t_1^2 = t_2^{2n}=\cdots=t_n^{2n} \rc.
$
There are precisely two non-zero $u \in U$ up to the action of $\Gm$
such that $\eY_u$ is singular,
i.e., $(u_2,u_{2n}) = (1,0)$ and $(1,-n)$.
The Hochschild cohomologies of these singular stacks are computed in Sections \ref{sc:cusp}
and \ref{sc:odp}.
Comparing this with $\SH^*(\Vv)$ computed in \pref{sc:degree_2_SH},
we conclude that the mirror of $\Vv$ is the stack associated with $(u_2,u_{2n}) = (1,0)$.
\end{proof}

\end{document}